\documentclass[a4paper]{article}
\usepackage{amsmath}
\usepackage{amssymb, latexsym, amscd, amsthm,amsfonts,amstext}
\usepackage{enumerate}
\usepackage{xcolor}
\usepackage{textcomp}
\usepackage{mathtools}
\usepackage{relsize}
\usepackage{graphicx}
\usepackage{graphicx}
\usepackage{subfigure}
\usepackage{cite}
\usepackage{xcolor}
\usepackage{chngcntr}
\usepackage{textcomp}
\usepackage{enumerate}
\usepackage{cleveref}
\usepackage{graphicx}
\usepackage{color}
\usepackage{footmisc}
\usepackage{mathtools}
\usepackage{pgf,tikz}
\usepackage{pgf,tikz}
\usepackage{mathrsfs}
\usetikzlibrary{arrows}
\usepackage{fancyhdr}
\DeclarePairedDelimiter{\abs}{\lvert}{\rvert}

\DeclarePairedDelimiter{\norm}{\lVert}{\rVert}
\makeatletter
\let\oldabs\abs
\def\abs{\@ifstar{\oldabs}{\oldabs*}}
\makeatletter

\let\oldnorm\norm
\def\norm{\@ifstar{\oldnorm}{\oldnorm*}}
\makeatother
\numberwithin{equation}{section}
\DeclareMathAlphabet{\mathpzc}{OT1}{pzc}{m}{it}

\newtheorem{theorem}{Theorem}[section]
\newtheorem{lemma}[theorem]{Lemma}
\newtheorem{proposition}[theorem]{Proposition}

\newtheorem{remark}[theorem]{Remark}

\newcommand{\nbs}{{{\aleph}}}
\newcommand{\rad}{\boldsymbol{a}}
\newcommand{\mc}[1]{{\mathcal{#1}}}
\newcommand{\bs}[1]{{\boldsymbol{#1}}}
\newcommand{\dist} {\boldsymbol{\delta}}
\newcommand{\cmp}{{\boldsymbol{\mathlarger{c}}}}
\newcommand{\z} {\boldsymbol{z}}
\newcommand{\n}{\text{in}}
\newcommand{\s}{\text{sc}}

\newcommand{\ta}{\text{t}} 
\newcommand{\Id}{{\mathbb{I}}}
\newcommand{\R}{\mathbb{R}}
\newcommand{\C}{\mathbb{C}}
\newcommand{\p}{\partial}
\newcommand{\abso}[1]{\Bigl|{#1}\Bigr|}
\newcommand{\scalar}[1]{\bigl<{#1}\bigr>}
\newcommand{\Hcurl}[2]{{{\mathbb{H}^{#1}}(\curl,{#2})}}

\newcommand{\Hs}[2]{{\mathbb{H}^{{#1}}}{(#2)}}
\newcommand{\LDivsp}[1]{{\mathbb{L}_t^{2,Div}}{(\partial{#1})}}

\newcommand{\LNDsp}[1]{{{\mathbb{L}_t^{2,0}}(\p#1)}}
\newcommand{\nablat}{\nu\times{\nabla}}
\newcommand{\Lp}[2]{{{\mathbb{L}^{#1}}(#2)}}
\newcommand{\LtwospO}[1]{{{\mathbb{L}^2_0}(\p #1)}}
\newcommand{\ol}[1]{\overline{#1}}             
\newcommand{\wh}[1]{\widehat{#1}}              
\newcommand{\stl}[2]{\stackrel{#1}{#2}}        
\newcommand{\bsmc}[1]{\bs{\mc{#1}}}            
\newcommand{\mean}[1]{\boldsymbol{\mathcal{A}}_{#1}} 
\newcommand{\meanm}[2]{\boldsymbol{\mathcal{A}}^{#2}_{#1}}
\newcommand{\cnt}[1]{{\boldsymbol{\mathcal{C}}}_{{#1}}} 
\newcommand{\meancnt}[1]{\bsmc{A}_{\bsmc{C}_{#1}}} 
\newcommand{\mcntm}[2]{\bsmc{A}^{#2}_{\bsmc{C}_{#1}}}
\newcommand{\ind}[1]{{\chi}_{#1}} 
\newcommand{\Qtot}[2]{\mathcal{Q}^{#1}_{#2}} 
\newcommand{\Polt}[1]{\bigl[{\mathcal{P}_{{D{#1}}}} \bigr]}
\newcommand{\Poltscale}[1]{\bigl[{\mathcal{P}_{\bsmc{D}_{#1}}} \bigr]}
\newcommand{\Vmt}[1]{\bigl[{{\mathcal T}_{D{#1}}} \bigr]}
\newcommand{\Vmtscale}[1]{\bigl[\mathlarger{\mathcal{T}_{\bsmc{D}_{#1}}} \bigr]}
\newcommand{\ANPT}[2]{[\mathcal{P}^{{#1}}_{#2}]}
\newcommand{\ANPTBdr}[2]{[{\bs{\mathcal{P}}}^{{#1}}_{\p #2}]}
\newcommand{\ANPTT}[2]{[\mathcal{T}^{{#1}}_{#2}]}

\newcommand{\De}{{\mathfrak{e}}}
\newcommand{\DE}{\mathfrak{E}}

\newcommand{\tensor}[2]{(#1)^{\stl{#2}{\otimes}}}
\newcommand{\Har}[1]{\mathcal{H}^{#1}}
\newcommand{\SLP}[2]{{ \boldsymbol {S}} ^{#1}_{\p #2} }
\newcommand{\DLP}[2]{\boldsymbol{K}^{#1}_{\p #2} }
\newcommand{\SDO}[2]{{S^{#1}_{{#2}_{\bsmc{D}_{#2}}}}}

\newcommand{\MDLP}[2]{{\boldsymbol{M}^{#1}_{{#2}_{D_{#2}}}}}

\newcommand{\MSVP}[2]{\mathcal{N}_{#1}^{#2}}
\newcommand{\MCVP}[2]{\mathcal{M}_{#1}^{#2}}
\newcommand{\MDLPmj}[1]{{\boldsymbol{M}_{m_{ D_j}}^{#1}}}

\newcommand{\TTr}[2]{\nu\times {#1}|_{#2}}
\newcommand{\NTr}[2]{\nu\cdot {#1}|_{#2}}
\newcommand{\DTr}[2]{{#1}|_{#2}}
\def\Div{\operatorname{Div}}
\def\curl{\operatorname{curl}}
\def\div{\operatorname{div}}
\newcommand{\prv}{\mathlarger{\varepsilon}}
\newcommand\prb{\mathlarger{\mu}}
\newcommand{\cnd}{\mathlarger{\sigma}}
\newcommand{\bdmc}[1]{{\mathcal{#1}}}
\usepackage{url}

\title{Foldy-Lax approximation of the electromagnetic fields generated by anisotropic inhomogeneities in the mesoscale regime\\ with complements for the perfectly conducting case}
\author{ Ali Bouzekri\thanks{Laboratoire de Mathematiques Pures et Appliquees, Universite Mouloud Mammeri, Tizi-Ouzou, Algeria
(Email: bouzekri.ali.lmpa@gmail.com).}
\and  Mourad Sini
\thanks{Radon institute (RICAM), Austrian Academy of Sciences, 69 Altenbergerstrasse, A4040, Linz, Austria
(Email: mourad.sini@oeaw.ac.at). This author is partially supported by the Austrian Science Fund (FWF): P28971-N32.}}
\begin{document}
\graphicspath{{Figures-eps/}} {\maketitle}
\begin{abstract}
The Foldy-Lax (or the point-interaction) approximation of the electromagnetic fields generated by a cluster of small scaled inhomogeneities is derived in the mesoscale regime, i.e. when the minimum distance $\delta$ between the particles is proportional to their maximum radi $a$ in the form $\delta=c_r \; a$ with a positive constant $c_r$ that we call the dilution parameter. We consider two types of families of inhomogeneities. In the first one, the small particles are modeled by anisotropic electric permittivities and/or magnetic permeabilities with possibly complex values. In the second one, they are given as perfectly conductive inclusions. In both the cases, we provide the dominating field (the so-called Foldy-Lax field) with explicit error estimates in terms of the dilution parameter $c_r$. In the case of perfectly conductive inclusions, the results provided here improve sharply the ones derived recently in \cite{AB-SM:MMS2019}. Such approximations are key steps in different research areas as imaging and material sciences. 
\end{abstract}

\textbf{Keywords}:  Electromagnetism, Small particles, Multiple scattering, Foldy-Lax approximation.
\bigskip

\textbf{AMS subject classification:}
 35J08, 35Q61, 45Q05. 
\pagestyle{myheadings}
\thispagestyle{plain}
\section{General setting and main results} 
\subsection{General setting}
Understanding the interaction between the waves (as the light or the acoustic fluctuations or the elastic displacements) with the matter has been of fundamental importance since a long time. 
Since the pioneering works of Rayleigh and Kirchhoff, it was known that the wave diffracted by small scaled inhomogeneities is dominated by the first multipoles (poles or dipoles). In modern terminology, 
the dominating fields are given by (polarized) point sources located at the center of the particles. As far as the three types of waves, cited above, are concerned these point sources are the Green's functions 
of the corresponding propagator. In this direction, the next key step is achieved by Mie \cite{Mie(1908)} in his full expansion of the electromagnetic field for spherically shaped particles.  
These formal expansions were later mathematically justified, see for instance \cite{Dassios-Kleinman:1999} in the framework of low frequencies expansions. A further step was achieved in \cite{A-K-2003} where
the full expansion at any order is derived and justified. 
\bigskip

These works dealt with single or well separated inhomogeneities. In other words, only the interaction of the single inhomogeneity and the wave is taken into account. 
In the presence of multiple and close inhomogeneities, then mutual interactions between them and the waves should be taken into account. In this respect, a formal argument to 
handle such multiple interaction was proposed by Foldy in his seminal work \cite{Foldy(1945)}.
To state his formulas, he looks the inhomogeneities as point-like potentials (i.e. Dirac-like potentials).
Then, he states a close form of the scattered wave by simply eliminating the singularity on the locations of these potentials. This elimination of the singularity translates a 
physical motivation saying that 'the scattering coefficient of each
point-like scatterer is proportional to the external field acting on it', which is known as the Foldy assumption.
These formal representations of the scattered waves
are then stated on more sound mathematical arguments by Berezin-Faddeev, see \cite{Berezin-Faddeev(1961)}, in the framework of the Krein's selfadjoint extension of symmetric operators. 
Another related method is the so called regularization method (or the renormalization technique) which aims at computing the Green's function, and hence the Schwartz integral operator, in the presence of the collections 
of inhomogeneities. This idea consists in taking the Fourier transform of the formal equation,'cut' or regularize and invert the related equations, via  Weinstein-Aronsza theorem,
in the Fourier domain and then comeback. More details on these ideas can be found in the book \cite{A-G-HK-H:2005}. The Faddeev approach was extended to singular potentials supported 
not only on points but also on curves and surfaces, see \cite{MA-PA-SM:JDE2016, MA-PA-SM:JST2018} for more details. This approach gives as, via the Krein's resolvent representations, 
exact formulas to represent the scattered waves generated by singular potentials. However, our goal is to deal with cluster of small scaled inhomogeneities. The intuitive believe is that the dominant part of the generated waves would be reminiscent to 
the exact formulas described above, for Dirac-like potentials supported on the centers of the inhomogeneities, but with scattering coefficients modeled by geometric or contrasts 
properties of the inhomogeneities. This is called the Foldy-Lax approximation or the point-interaction approximation. 
\bigskip

Several methods were proposed in the literature to justify such approximations, see \cite{Ramm:2005, M-M-N:Springerbook:2013, DPC-SM:MMS2013, DPC-SM:MANA2015} for instance regarding acoustic and elastic waves. 
Descriptions and relation/differences between these works can be found in \cite{DPC-SM:ZFAMP2016}. Let us emphasize here that those works dealt with exterior problems (impenetrable inclusions, holes or voids). 
Regarding electromagnetic waves, very few works are proposed, apart from \cite{Ramm-2} where both the results and the justifications are quite questionable. In our previous work \cite{AB-SM:MMS2019}, we considered 
the case of perfectly conductive inclusions and we gave a rigorous justification of this Foldy-Lax approximation under general conditions on the cluster of such inclusions as their  minimum distance between them 
$\delta$ and their maximum radius $a$ of the form $\ln(\delta^{-1})\; \frac{a}{\delta}$ is bounded by a constant depending only on the Lipschitz character of the shapes of the inclusions. 
The only limitation of this result, to handle the mesoscale regime (i.e. $\delta \sim a$), is the appearance of the term $\ln(\delta^{-1})$. This term appears naturally in the analysis, 
in that work, which is heavily based on the scales of the related layer potentials knowing that the dyadic Maxwell fundamental solution has a singularity of order $3$ (while the ones of Laplace or 
Lam\'e have singularities of the order $1$).   
\bigskip

In this present work, we get rid of this logarithmic term and state the approximation in the mesoscale regime. 
But the most important contribution is to handle the transmission problem and the impenetrable problem in a unified way. In addition, anisotropic and eventually complex valued electromagnetic 
material parameters can be handled as well. Our arguments can be summarized as follows. To handle the anisotropic transmission problem, we provided a representation of the solution using the  
electromagnetic Lippmann-Schwinger operator and give an estimate of the total field. To overcome the logarithmic constraint, instead of using Neumann series
to estimates the density of the used representation, we make use of a Rellich identity and, inspired by some arguments from \cite{DMM96}, with appropriate changes, we prove that both the exterior and 
the interior traces have an equivalent norm modulo a constant which depends on the geometry of the inhomogeneities and the material parameters (via their contrasts). 
Regarding the perfect conductor problem, we use layer potential representation of the solution. Using an appropriate Helmholtz decomposition for
the density, appearing in the layer potential representation, we transforme the boundary integral interaction operator into a volume one to get Lippmann-Schwinger like integral representation.
This allows us to translate the result obtained for the transmission problem to the perfectly conductive case.
\bigskip

It is worth mentioning that even in the scalar case, i.e. related to the Laplace operator, with a cluster of small obstacles with Neumann boundary conditions was left open to our best knowledge.
The approach we follow here definitely handles this case and provides the corresponding Foldy-Lax approximation in the same generality as we are proposing in it this work.
\bigskip

In the case of periodically distributed small inhomogeneities, the homogenization applies, see \cite{B-A-JL-P-G:AMS2011, JVV-KSM-OA:SpSc2012}, and provides the equivalent media with averaged materials. 
As compared to homogenization, the Foldy-Lax approximation has several advantages. The first one is that we have the dominating field (i.e. the Foldy-Lax field) for general (and not only periodic) distributions 
of the small inhomogeneities. This reduces the complexity of the forward problem to compute the scattered fields by inverting an algebraic system. 
Second, higher order approximations are possible with more effective dominating fields, i.e. with generalized Foldy-Lax fields. So far, this is not fully justified,
but we believe it to be true and we will report on it in the future. Third, as we have freedom in distributing these inhomogeneities, then we can generate not only volumetric
equivalent materials but also low dimensional ones as surfaces and curves. This opens the way to applications in low dimensions metamaterials as well, see \cite{AH-DPC-CAP-MS:JDE2019, AH-DPC-CAP-MS:MMS2019Arxiv} 
for instance. As far as the Maxwell model is concerned, the Foldy-Lax approximation provided here shows that one can generate volumetric materials and Gradient-metasurfaces. The first situation is modeled by 
modifying both the background permittivity and permeability. The second is modeled by an equivalent interface with jumps of both the electric and magnetic fields across it.  
\bigskip

The rest of the paper is described as follows. In the next subsection, we state clearly the models and the obtained results with critical discussion about them. 
In section 2 and section 3, we provide the full proofs of the results for the transmission and the perfectly conducting models respectively. 
A short Appendix is added at the end to include technical tools and in particular a useful lemma on the counting of the number of small particles distributed 
in any given bounded set in terms of the parameters $\delta$ and $a$.

\subsection{Main results}
We deal with the scattering of time-harmonic electromagnetic plane waves at a frequency $\omega$ in a medium composed of an isotropic and constant background and an anisotropic material represented by multiply connected, bounded, Lipschitz\footnote{This means that the boundary is locally described by the graph of a Lipschitz function. More details are given later.} domain $D^-=\cup_{m=1}^{\nbs} D_m$ where ${\nbs}$ is the number of connected components. 
The Maxwell equations read as follows
 \begin{align}
 	& \left\{
	\begin{aligned}
 		&\nabla \times \mathcal{E}-i\omega{\mathlarger\mu} \mathcal{H}=0, ~~\text{in}~\R^3\setminus \p D,\\
 		& \nabla \times \mathcal{H}+i\omega{\mathlarger\varepsilon} \mathcal{E}={\mathlarger\sigma} \mathcal{E},\label{Maxwell-Eq}~~\text{in}~\R^3\setminus \p D ,\\
	&\mathcal{E}=\mathcal{E}^\n+\mathcal{E}^\s,\\
	&\nu\times\DTr{\mathcal{E}}{+}=\nu\times\DTr{\mathcal{E}}{-},\,\, \nu\times\DTr{\mathcal{H}}{+}=\nu\times\DTr{\mathcal{H}}{-} \mbox{ on } \p D
	\end{aligned} \right.
 \end{align} with the notation $\p D:=\cup_{m=1}^{\nbs} \p D_m$ where $\prv$ and $\cnd$ are respectively the electric permittivity and the conductivity and $\prb$ corresponds to the magnetic permeability. These parameters can be real or complex tensor or scalar valued functions.
Here $\mathcal{E}^\n$ stands for the incident wave. It is solution of the first two equations above everywhere in the space. The vector field $\mathcal{E}^\s$ stands for the scattered vector field.\\ 

We also consider the scattering from a perfect conductor modeled by the following problem 
\begin{align}\label{Maxwell-Eq-Perfect-Cond-1}
&\left\{\begin{aligned}
&\nabla \times \mathcal{E}-i \omega{\mathlarger\mu} \mathcal{H}=0,\\
&\nabla \times \mathcal{H}+i\omega{\mathlarger\varepsilon} \mathcal{E}=0, ~~\text{in}~\R^3\setminus \ol{D}, \\
&\mathcal{E}=\mathcal{E}^\n+\mathcal{E}^\s,\\
&\TTr{\mathcal{E}}{+}=0, \mbox{ on } \partial D,
\end{aligned}\right.
\end{align} where $\nu$ is the unit outward normal vector to the boundary of $D^-$. 
The surrounding background of $D^-$ is homogeneous with constant parameter $\prv_0,\prb_0$ and null conductivity ${\mathlarger\sigma}$. In both the two models above, the scattered field must satisfy the radiation conditions
\begin{align}
\Bigl(\sqrt{\prb_0\prv_0^{-1}}\Bigr)\mathcal{H}^\s(x)\times \frac{x}{\abs{x}}-\mathcal{E}^\s(x)=O\Bigl(\frac{1}{\abs{x}^2}\Bigr).\label{Radiation-Cdt}
\end{align}    
Setting $k:=w\sqrt{\epsilon_0\mu_0}$,  $\mathcal{E}:=\sqrt{{\prv_0}{\prb_0}^{-1}}E$ and $\mathcal{H}:=\sqrt{{\prv_0}{\prb_0}^{-1}}H$ with $\mu_r:=(\mu_0)^{-1}\mu$ and $\prv_r:=(\prv_0)^{-1}(\prv+i\sigma/\omega)$, we arrive at 

\begin{equation}\label{Maxwell-Eq-Anisotropic}
\left\{\begin{aligned}
&\curl E-ik\prb_r H=0, ~~\text{in}~\R^3\setminus \p D,\\
&\curl H+ik\prv_r E=0, ~~\text{in}~\R^3\setminus \p D,\\
&E=E^\s+E^\n,\\
&\TTr{E}{-}-\TTr{E}{+}=\TTr{H}{-}-\TTr{H}{+}=0,\\
&H^\s\times \frac{x}{\abs{x}}-E^\s=O\Bigl(\frac{1}{\abs{x}^2}\Bigr),~~\abs{x}\rightarrow\infty,
\end{aligned}\right.\tag{$\mathcal{P}_1$}
\end{equation} 
for the inhomogeneous (i.e. transmission) problem and 
\begin{align}\label{Maxwell-Eq-Perfect-Cond}
&\left\{\begin{aligned}
&\nabla \times E-i k H=0, ~~\text{in}~\R^3\setminus \ol{D}\\
&\nabla \times H+ik E=0, ~~\text{in}~\R^3\setminus \ol{D}, \\
&E=E^\n+E^\s,\\
&\TTr{E}{+}=0, \mbox{ on } \partial D,
\end{aligned}\right.\tag{$\mathcal{P}_2$}.
\end{align} 
for the conductive (i.e. exterior) problem.
For both the two models the incident field $(E^\n, H^\n)$ satisfies in the whole space the system
\begin{align}\label{Incident-fields}
\left\{\begin{aligned}
&\curl E^\n-ik H^\n=0, \\
&\curl H^\n+ik E^\n=0.
\end{aligned} \right.
\end{align}
Motivated by applications, typical incident electric fields are plane waves, i.e~of the form $E^\n:=E^\n(x, \theta):=P\; e^{i k\theta \cdot x}, x\in \mathbb{R}^3$, 
where $P$ is the (constant) vector modeling the polarization direction and $\theta$, with $\vert \theta\vert=1$, is the incident direction such that $P\cdot \theta=0$. 
The related magnetic incident field is then $H^\n:=H^\n(x, \theta):= P\times \theta\; e^{i k\theta \cdot x}/ik$.
\bigskip
 
We suppose that, 
\begin{equation}\label{Definition-of-the-Di's}
D_m:=\rad \bsmc{D}_m+\z_i, i=1,...,\nbs, 
\end{equation} 
 where each set $\bsmc{D}_i$, contained in the ball $B_0^{1/2}:=B(0,1/2),$ and contains the origin, is assumed to be a Lipschitz bounded domain. The points $(\z_i)_{i=1}^{\nbs}$ are their given locations in $\mathbb{R}^3$ and $\rad\in\R^+$ is a small parameter measuring the maximum relative radius.\\
Let $\dist_{mj}:=\min_{x\in D_m, y\in D_j} {{d(x,y)}},$ be the distance between two bodies $D_m, D_j, m\neq j$, and set $$\dist:=\min_{m\neq j\in \{1,...,\nbs\}} {\dist_{mj}}. $$ 

Let us recall that a bounded open connected domain $B$, is said to be a Lipschitz domain with character $(l_{\p B},L_{\p B})$ 
if for  each $x\in\partial D$ there exist a coordinate system $(y_i)_{i=1,2,3}$, a truncated cylinder $\mathfrak{C}$ centered at $x$ whose axis is parallel to $y_3$ with length 
$l$ satisfying $l_{\p B}\leq l\leq 2l_{\p B}$, and a Lipschitz function $f$ that is $\abs{f(s_1)-f(s_2)}\leq L_{\p B}\abs{s_1-s_2}$ for every $s_1,s_2\in \mathbb{R}^2$ , 
such that $B\cap \mathfrak{C}= \{(y_i)_{i=1,2,3}: y_3>f(y_1,y_2) \}$ and ${\p B}\cap \mathfrak{C}= \{(y_i)_{i=1,2,3}: y_3=f(y_1,y_2) \}.$ In this work, we assume that the sequence of 
Lipschitz characters $(l_{\partial \bsmc{D}_m},L_{\partial \bsmc{D}_m})^{\nbs}_{i=1}$ of the bodies $\bsmc{D}_m, i=1,..., \nbs,$ is bounded.
 \bigskip

Regarding the problem \ref{Maxwell-Eq-Anisotropic}, we need some assumptions on the electromagnetic material properties of the small particles. 
Precisely, we suppose that the contrast of magnetic permeability and electric permittivity , which are assumed to be respectively real and complex valued $3\times3$-tensor, are, 
with their derivative, essentially uniformly bounded , i.e.
 \begin{align}
 (\norm{\cnt{A}}_{\mathbb{W}^{1,\infty}(\cup_{i=1^m}D_m)})_{A=\prv_r,~\prb_r}\leq \bs{c}_\infty,  
 \end{align} and essentially uniformly coercive that is, for almost every $x\in D,$ we have
 \begin{equation}\label{Coercivity-Cond-Contrast-introduction}
 \begin{aligned}
 \Re\Bigl(\cnt{\prv_r}(x)U\cdot \overline{U}\Bigr)&\geq c^{\prv-}_\infty\abs{U}^2,\\
 \cnt{\prb_r}(x)U\cdot U &\geq c^{\prb-}_\infty\abs{U}^2, 
\end{aligned} 
\end{equation}
with positive constants $c^{\prv-}_\infty$ and $c^{\prb-}_\infty$. Here we used the notation $\cnt{A}:=A-\Id$, where $\Id$ is the identity matrix of $\R^3\times\R^3$.  $I$ will stand for the identity operator. 
\bigskip

 Under these conditions, both the scattering problem \eqref{Maxwell-Eq-Anisotropic} and \eqref{Maxwell-Eq-Perfect-Cond} under their respective transmission/boundary and radiating conditions are 
 well posed in appropriate spaces following the lines described in \cite{ColtonKress:2013, nedelec1973} for instance. More details are given in the text. In addition, due to the Stratton-Chu 
 formula when $\Im k$ is different from zero, the scattered electromagnetic fields have a fast decay at infinity as we have attenuation.
 But when $\Im k=0$, i.e.~in the absence of attenuation, we have the following behavior (as spherical-waves) of the scattered electric fields far away from the sources $D_m$'s
 \begin{equation}\label{electric-farfield}
  E^\s(x)=\frac{e^{ik\vert x\vert}}{\vert x\vert} \{ E^{\infty}(\hat{x}) +O(\vert x\vert^{-1}) \}, ~~~ \vert x\vert \longmapsto \infty,
 \end{equation}
and we have a similar behavior for the scattered magnetic field as well
\begin{equation}\label{magnetic-farfield}
  {H}^\s(x)=\frac{e^{ik\vert x\vert}}{\vert x\vert}\{{H}^\infty(\hat{x}) +O(\vert x\vert^{-1})\}, ~~~ \vert x\vert \longmapsto \infty
 \end{equation}
where $(E^\infty(\hat{x}), {H}^\infty(\hat{x}))$ is the electromagnetic far field pattern in the direction of propagation $\hat{x}:=\frac{x}{\vert x\vert}$.
 \bigskip

 We set, for $m\in\{1,...,\nbs\}$,
 \begin{equation}\label{Average-Notation}
 \meanm{f}{m}:=\frac{1}{\abs{D_m}}\int_{D_m}f dv,\end{equation}
and 
\begin{equation} 
  \mean{f}(x):=\sum_{m=1}^\nbs\meanm{f}{m}\ind{D_m}(x).
 \end{equation} 
Here $dv$ is the volume measure of $\R^3,$ and $dv(x)$ will be denoted $dx$ while for the surface measure of $\R^3$ write $ds$ and $ds_x$ when the variable of integration is specified. 
 Finally, we recall the Green's function for the Helmholtz operator (i.e. the fundamental solution for the Helmholtz equation)
  \[\Phi_k(x,y)=\frac{1}{4\pi}\frac{e^{ik\abs{x-y}}}{\abs{x-y}},~~ x \neq y,\]and the electromagnetic dyadic Green's function
 \begin{equation}\label{Dyadic-Green-function}
 \Pi(x,y):=k^2\Phi_k(x,y)\Id+\nabla_x\nabla_x \Phi_k(x,y) 
 =k^2\Phi_k(x,y)\Id-\nabla_x\nabla_y \Phi_k(x,y), ~~ x \neq y. 
 \end{equation}   
  
Our main results are stated in the following two theorems. 
 \begin{theorem}\label{Theorem-Anisotropic-Case}
 For the scattering by a cluster of small anisotropic particles embedded in a homogeneous background whose parameter satisfy the conditions (\ref{Coercivity-Cond-Contrast-introduction}), with maximal diameter $\rad$ and minimal distance separating them $\dist=\bs{c}_r\rad.$ The far field of the scattered wave admits, provided that  $\bs{c}_r=O(\abs{k}),$ the following expansions, 
 \begin{itemize}
 	\item If both $\prv_r$ and $\prb_r$ are symmetric, then we have the following approximation
 	\begin{equation} \label{Farfield-Anis-Approximation-FullError}
 	\begin{aligned}E^\infty(\hat{x})=&
 	\sum_{m=1}^{\nbs}\biggr(\frac{k^2}{4\pi}e^{-ik\hat{x}\cdot \z_m}\hat{x}\times\Bigl({\mc{R}}_m^{\prv_r}\times \hat{x}\Bigr)+\frac{ik}{4\pi}e^{-ik\hat{x}\cdot \z_m}\hat{x}\times{\mc{Q}}_m^\prb_r\Biggl)
 	\\&+O\Bigl(\frac{\abs{k}(2\abs{k}+1)}{\bs{c}_r^{3}}\Bigl[\frac{1}{\bs{c}_r^4}+\rad\,\abs{\ln(\bs{c}_r\rad)}+ \rad \Bigr]\Bigr) ,
 	\end{aligned}
 	\end{equation} where $({\mc{R}}_m^{\prv_r},{\mc{Q}}_m^\prb_r)_{m=1}^\nbs$ is the solution of the following invertible linear system
\begin{equation}
\begin{aligned}
\ANPT{\prb_r^*}{D_m}^{-1}\mc{Q}^{\prb_r}_m=&\sum_{\stackrel{j\geq 1}{j\neq m}}^{\nbs} \Bigl[ \Pi_{k}(\z_m,\z_j) \mc{Q}^{\prb_r}_j-ik \nabla\Phi_{k}(\z_m,\z_j)\times\mc{R}^{\prv_r}_j\Bigr]+ {H}^\n(\z_m) \\
\ANPT{\prv_r^*}{D_m}^{-1}\mc{R}^{\prv_r}_m=&\sum_{\stackrel{j\geq 1}{j\neq m}}^{\nbs}\Bigl[ \Pi_{k}(\z_m,\z_j) \mc{R}^{\prv_r}_j+ ik\nabla\Phi_{k}(\z_m,\z_j)\times\mc{Q}^{\prb_r}_j\Bigr]+{E}^\n(\z_m).
\end{aligned}
\mbox{  For }~m=1, ..., \nbs.
\end{equation}

\item If, in the contrary, $\prv_r$ or $\prb_r$ is not symmetric, then, with $(\meanm{\prv_r}{m})_{m=1}^\nbs$ and $(\meanm{\prb_r}{m})_{m=1}^\nbs$ standing for their respective average in each particle $D_m$, we have
\begin{equation}\label{Farfield-Anis-NonSymCase-Approximation-FullError}
\begin{aligned} E^\infty(\hat{x})=&\sum_{m=1}^{\nbs}\biggr(\frac{k^2}{4\pi}e^{-ik\hat{x}\cdot \z_m}\hat{x}\times\Bigl(\mc{R}_m\times \hat{x}\Bigr)+\frac{ik}{4\pi}e^{-ik\hat{x}\cdot \z_m}\hat{x}\times\mc{Q}_{m}\Biggl)\\
&+O\Bigl(\frac{\abs{k}(2\abs{k}+1)}{\bs{c}_r^{3}}\Bigl[\frac{1}{\bs{c}_r^4}+\rad\,\abs{\ln(\bs{c}_r\rad)}+ \rad \Bigr]\Bigr),\\
\end{aligned}  
\end{equation} where, in this case, $\Bigl(\mc{R}_m\Bigr)_{m=1}^\nbs$ and $\Bigl(\mc{Q}_m\Bigr)_{m=1}^\nbs$ is the solution the invertible following linear system 
\begin{align}
\begin{aligned}
{\ANPTT{\meanm{\prb_r^*}{m}}{D_m}^{-1}}\mc{Q}_m
=&\sum_{\stl{j\geq 1}{j\neq m}}^{\nbs}\Bigl[\Pi_{k}(\z_m,\z_j){\mc{Q}}_j- ik\nabla\Phi_{k}(\z_m,\z_j)\times{\mc{R}}_j\Bigr]+ {H}^\n(\z_m), \\
{\ANPTT{\meanm{\prv_r^*}{m}}{D_m}^{-1}}\mc{R}_m =&\sum_{\stl{j\geq 1}{j\neq m}}^{\nbs}\Bigl[\Pi_{k}(\z_m,\z_j){\mc{R}}_j+ ik\nabla\Phi_{k}(\z_m,\z_j)\times{\mc{Q}}_j\Bigr]+ {E}^\n(\z_m),\\ 
\end{aligned} \mbox{  For }~m=1, ..., \nbs.
\end{align} \begin{equation*}
{\ANPTT{\meanm{\prb_r^*}{m}}{D_m}^{-1}}:= \mcntm{\prb_r^*}{m}\ANPT{\meanm{\prb_r^*}{m}}{D_m}^{-1}{(\mcntm{\prb_r}{m})}^{-1},\\
\end{equation*} and
\begin{equation*}
{\ANPTT{\meanm{\prv_r^*}{m}}{D_m}^{-1}}:= \mcntm{\prv_r^*}{m}\ANPT{\meanm{\prv_r^*}{m}}{D_m}^{-1}{(\mcntm{\prv_r}{m})}^{-1}.
\end{equation*}
 \end{itemize}
For a given matrix $B$, we have
 \begin{equation}
\ANPT{B}{D_m}:=\int_{D_m}\nabla V~ (B-\Id) dv
\end{equation} 	with $V$ is the solution of the following integral equation\begin{equation}
	 V-\div\int_{D_m}\frac{1}{\abs{x-y}}(B-\Id)\nabla V(y)dy=(y-\z_m), ~~ y \in D_m.
\end{equation} 	
 \end{theorem}
\begin{theorem}\label{Theorem-PerfCond-Case}
For the scattering by a cluster of small conducting particles, with maximal diameter $\rad$ and minimal distance separating them $\dist=\bs{c}_r\rad$, with $\bs{c}_r=O(\abs{k})$, the far field of the scattered wave admits the following expansion 
		\begin{equation}
		\begin{aligned}E^\infty(\hat{x})=
		\sum_{m=1}^{\nbs}&\biggr(\frac{k^2}{4\pi}e^{-ik\hat{x}\cdot \z_m}\hat{x}\times\Bigl({\mc{R}}_m\times \hat{x}\Bigr)+\frac{ik}{4\pi}e^{-ik\hat{x}\cdot \z_m}\hat{x}\times{\mc{Q}}_m\Biggl)
		\\+&O\Bigl(\frac{\bs{c}_{L}\abs{k}(2\abs{k}+1)}{\bs{c}_r^{3}}\Bigl[\frac{1}{\bs{c}_r^4}+\rad\,\abs{\ln(\bs{c}_r\rad)}+ \rad \Bigr]\Bigr) ,
		\end{aligned}
		\end{equation} where $({\mc{R}}_m,{\mc{Q}}_m)_{m=1}^\nbs$ is the solution the invertible linear system
		\begin{equation}
		\begin{aligned}
		\Vmt{m}^{-1}\mc{Q}_m=&\sum_{\stackrel{j\geq 1}{j\neq m}}^{\nbs} \Bigl[ \Pi_{k}(\z_m,\z_j) \mc{Q}_j-ik \nabla\Phi_{k}(\z_m,\z_j)\times\mc{R}_j\Bigr]+ {H}^\n(\z_m) \\
		\Polt{m}^{-1}\mc{R}_m=&\sum_{\stackrel{j\geq 1}{j\neq m}}^{\nbs}\Bigl[ \Pi_{k}(\z_m,\z_j) \mc{R}_j+ ik\nabla\Phi_{k}(\z_m,\z_j)\times\mc{Q}_j\Bigr]+{E}^\n(\z_m),
		\end{aligned}
		\mbox{  For }~m=1, ..., \nbs,
		\end{equation}
		with
		\begin{align}\label{Polarization&virtualmass-Tensor-defintion}
\Polt{m}:=\int_{\p D_m}[-{I}/{2}+(\DLP{0}{D_m})^*]^{-1}(\nu) y^*ds_y,~~
\Vmt{m}:= \int_{\p D_m}[\frac{1}{2} I+(\DLP{0}{D_m})^*]^{-1}(\nu) y^*ds_y.
\end{align} 
\end{theorem}
Before we provide the proofs of the above theorem, we would like to address some remarks.
\begin{itemize}
	\item In the error of approximation, the constant appearing in the Landau notation are bounded by the largest ratio of the eigenvalues of both $\prv,$ $\prb,$ for the Theorem \ref{Theorem-Anisotropic-Case},
	and the largest Lipschitz constant for Theorem \ref{Theorem-PerfCond-Case}.  
	\item In our opinion, and in the current form of the algebraic systems, it is hard to improve the approximation error order, except maybe for rotation invariant geometries by using the fundamental Newton's theorem for fields that are also rotation invariant.  
	\item The tensor that appears could be explicitly calculated for simple geometries (sphere, ellipsoid) for more details see \cite{Ammaripola}. Further more, for a perfect conductor case, the tensor can be explicitly calculated, for convex geometries, using Neumann series as the spectral radius of the double layer potential have a spectral radius that is smaller than $\frac{1}{2}.$
	\item It is also possible to evaluate $\ANPT{\prb_r}{D_m}$ and
	$\ANPT{\prv_r}{D_m},$ using boundary integral equation when both $\prv_r^*$ and $\prb_r^*$ are symmetric definite positive matrix see \Cref{Ani-Pol-Tensor} here after.	    
\end{itemize}  

\section{Scattering by Anisotropic Inhomogenities. Problem (\ref{Maxwell-Eq-Anisotropic})}

Let us introduce the Newton-like potential
\begin{equation} 
\bdmc{S}_D^{k}(V):=\int_{D}\Phi_k(x,y) V(y)dy,
\end{equation}  defined, for $V$ in $\Lp{2}{D}$ and maps continuously $\Lp{2}{D}$ into $\Hs{2}{D},$ (see Theorem 9.11 \cite{gilbarg2015elliptic}) precisely 
\begin{align}
\norm{\bdmc{S}_D^{k}(V)}_{\Hs{2}{D}}\leq c_{2,k} \norm{V}_{\Lp{2}{D}},
\end{align} which is a compact integral operator from $L^2(D)$ to $H^s(D),~~ s<1$. The constant $c_{2,k}$ remains independent of $D$. To show it, it suffices to write 
\begin{equation*}
\bdmc{S}_D^{k}(V)=\bdmc{S}_{B(v,R)}^{k}(V\ind{D})
\end{equation*} for any sufficient large radius $R$ to contain $D$ and $v\in D$, and  
\begin{align}\label{Continuity-Newton-Like-Oper}
\norm*{\bdmc{S}_D^{k}(V)}_{\Hs{2}{D}}\leq\norm*{\bdmc{S}_{B(v,R)}^{k}(V\ind{D})}_{\Hs{2}{B(v,R)}}\leq \bs{c}_{2,k} \norm{V\ind{D}}_{\Lp{2}{B(v,R)}}.
\end{align}  
Finally by $\Hcurl{}{D}$, we mean the subspace of $\Lp{2}{D}$-vector fields, with $\Lp{2}{D}$ rotational, that is
\begin{align}
\Hcurl{}{D}=\Bigl\{u\in\Lp{2}{D}~|~ \curl u\in\Lp{2}{D}\Bigr\}.
\end{align} 
We also define, for a bounded tensor $C$, 
\begin{equation} 
\bdmc{S}_D^{k, C}(V):=\int_{D}\Phi_k(x,y) C(y)V(y)dy,
\end{equation}  for $V$ in $\Lp{2}{D}$.

\subsection{Anisotropic polarization tensor}\label{Ani-Pol-Tensor} 
Let us set \begin{equation}\label{Def-Vol-Polt-Tensor}
\ANPT{B}{D_m}:=\int_{D_m}\nabla V_m^{\cnt{B}}~\cnt{B} ~ dv,
\end{equation} where $V_m^{\cnt{B}}:=V_m^\s+(x-\z_m)$ is the  solution of the following problem
\begin{equation}
\left\{\begin{aligned}
&\Delta V_m^{\cnt{B}}=0 \mbox{ in }\R^3\setminus D_m,\\
&\div (A\nabla V_m^{\cnt{B}})=0 \mbox{ in } D_m,\\
&V_m^{\cnt{B}}=V^\s+(x-z_i),\\
&\DTr{V_m^{\cnt{B}}}{-}-\DTr{V_m^{\cnt{B}}}{+}=0 \mbox{ on } \p D_m,\\
&\NTr{B\nabla V_m^{\cnt{B}}}{-}-\NTr{\nabla V_m^{\cnt{B}}}{+}=0,\\
&V^\s_m\rightarrow 0, \abs{x}\rightarrow\infty.
\end{aligned}\right.\label{Anisotropic-Pr-Lip-Sch}\tag{$\mathcal{P}r^{Ani}(1)$}
\end{equation} 
The problem \eqref{Anisotropic-Pr-Lip-Sch} is solved by the following Lippmann-Schwinger integral equation with $V_l:=(V_m^{\cnt{B}})_l=V_m^{\cnt{B}}\cdot e_l,$ for $l=1,2,3$,
\begin{equation}\label{Eq-Lipp-Sch-Anisotropic-Tensor}
V_l-\div\bdmc{S}_{D_m}^{0,\cnt{B}}(\nabla V_l)=(x-\z_m)\cdot e_l.
\end{equation} The following proposition summarizes the needed properties of the polarization tensor introduced above. 
\begin{proposition}\label{Pol-Ten-Behaviour} Every solution to \eqref{Anisotropic-Pr-Lip-Sch} satisfies the following estimate \begin{equation}\label{Estim-Aniso-Pol-Ten}
	\norm*{V}_{\Hs{1}{D_m}}\leq \left(\frac{1}{2}-\rad^2\norm*{\cnt{}}_{{\Lp{\infty}{D_m}}}\sqrt{\frac{3}{\pi}}\right)^{-1}(\norm*{\cnt{}}_{{\Lp{\infty}{D_m}}}+1)\rad^{3/2}.
	\end{equation} Furthermore, whenever $B\in \mathbb{W}^{1,\infty}(D_m)$ and $\rad$ sufficiently small, the tensor \eqref{Def-Vol-Polt-Tensor} behaves like $\cnt{B}$ 
	in terms of positive or negative definiteness and symmetry, namely 
	\begin{equation}\label{Anis-Pol-ten-Behaviour}
	(\ANPT{B}{D_m} U,\ol{U})>0, \mbox{ whenever } (\cnt{B}U,\ol{U})>0,
	\end{equation} and 
	\begin{equation}
	\Re(\ANPT{B}{D_m} U,\ol{U})>0, \mbox{whenever } \Re(\cnt{B}U,\ol{U})>0.
	\end{equation}  In addition, we have the following scaling property 
 \begin{equation} \label{Ani-PolTen-Scaling}
\ANPT{B}{D_m}=\rad^3\ANPT{\wh{B}}{\bsmc{D}_m},
\end{equation} where for $s\in\bsmc{D}_m,$ $\wh{B}(s):=B(\rad s+z_m).$
\end{proposition} 
Before stating the proof, we need to introduce the anisotropic polarization tensor, defined in (\cite{Ammaripola}, p.121-122 ), as \begin{align*}
\Bigl(\ANPTBdr{B}{D_m}\Bigr)_{ij}=&\int_{\p D_m}\nu\cdot(\cnt{B} e_j) (\Theta_m)_i|_- ds,
\end{align*}  where $\Theta_m$ solves, for a fixed $m\in\{1,...,\nbs\}$, the following transmission problem
\begin{equation}\label{Anisotropic-Prob}
\left\{\begin{aligned}
&\Delta \Theta=0 \mbox{ in }\R^3\setminus D_m,\\
&\div (B\nabla \Theta)=0 \mbox{ in } D_m,\\
&\DTr{\Theta}{-}-\DTr{\Theta}{+}=(x-\z_m) \mbox{ on } \p D_m,\\
&\NTr{B\nabla \Theta}{-}-\NTr{\nabla \Theta}{+}=\nu\cdot \nabla(x-\z_m),\\
&\Theta\rightarrow 0, \abs{x}\rightarrow\infty.
&\end{aligned}\right.\tag{$\mathcal{P}r^{Ani}$}
\end{equation}
Obviously, we have
\begin{equation}
\Theta_m:=\left\{\begin{aligned}
&V_m^\s \mbox{ in } \R^3\setminus D_m,\\
&V_m^\s+(x-\z_m)  \mbox{ in } D_m.
\end{aligned}\right. 
\end{equation} 
 Hence, being
\begin{align*}
\Bigl(\ANPTBdr{B}{D_m}\Bigr)_{ij}=&\int_{\p D_m}\nu\cdot(\cnt{B} e_j) (\Theta_m)_i|_- ds=\int_{D_m}\div\Bigl((\cnt{B} e_j)~ (\Theta_m\cdot e_i) \Bigr)dv,
\\=&\int_{D_m}((\nabla\Theta_m)^*e_i)\cdot(\cnt{B} e_j)dv +\int_{D_m}(\Theta_m\cdot e_i)\div(\cnt{B} e_j)dv.
\end{align*}
we get\footnote{Recall that, for a given matrix $A$ $A e_j\cdot e_i=(A)_{ij}$ }
\begin{equation}\label{Ani-PolTen-Def-Volume-Integral}
\ANPTBdr{B}{D_m}=\ANPT{B}{D_m}+\int_{D_m} \Theta_m\otimes\div(\cnt{B}^*) dv,	
\end{equation} or \begin{equation}\label{Ani-PolTen-ProperDef-V-Iequation}
\ANPT{B}{D_m}=\ANPTBdr{B}{D_m}-\int_{D_m} \Theta_m\otimes\div(\cnt{B}^*) dv.
\end{equation}
where the divergence is applied to each line of the matrix taken as a vector field, and $\cnt{B}^*$ stands for the transpose of $\cnt{B}$. 
With the above notations, we have the following lemma.
\begin{lemma}\label{Aniso-Pol-tens-Bdr-Positiveness}(Theorem 3.4 of \cite{kim2003anisotropic})\label{Signe-Of-Anis-Pol-Ten} The polarization tensor $\ANPTBdr{B}{D_m}$ behaves like $\cnt{B}$ in term of positive or negative definiteness and symmetry.
\end{lemma}	

\begin{proof} (of \Cref{Pol-Ten-Behaviour}) 
	The operator $[I-\div\bdmc{S}_{D_m}^{i,\cnt{B}}\nabla]$ is one-to-one on ${\Hs{1}{D_m}}$, provided that $\cnt{B}$ is definite-positive, and it satisfies the following estimates (see \cite{kirsch2009operator} 
	or the proof of \eqref{Lip-Sch-Defint-Positive-part} in the next subsection.) 
	\begin{equation}\label{Estimates_I-divS_{D_m}{i,C}}
	\begin{aligned}
	\norm*{[I-\div\bdmc{S}_{D_m}^{i,\cnt{B}}\nabla](V)}_{\Hs{1,\cnt{B}}{D_m}}:=&\int_{D_m}[I-\div\bdmc{S}_{D_m}^{i,\cnt{B}}\nabla](V)\cdot \ol{V}dx\\
	&+\int_{ D_m}\nabla[I-\div\bdmc{S}_{D_m}^{i,\cnt{B}}\nabla](V)\cdot \cnt{B}\nabla\ol{ V}dx\geq \frac{\norm*{V}_{{\Hs{1}{D_m}}}}{2}.
	\end{aligned}\end{equation}
We have
	\begin{equation}
	\norm*{[\div(\bdmc{S}_{D_m}^{i,\cnt{B}}-\bdmc{S}_{D_m}^{0,\cnt{B}})\nabla V_l]}_{\Hs{1}{D_m}}\leq\frac{3}{\sqrt{\pi}} \rad^2\norm*{V_l}_{\Hs{1}{D_m}},\label{Estim_divS_{D_m}{i,CA}-comp-rem}
	\end{equation} 
	since, due to \eqref{Nabal-Phi_k-Phi_i} and \eqref{NablaNabal-Phi_k-Phi_i}, we have both
	\begin{align}\label{Vm-Estimate}
	\left\{\begin{aligned}\abs*{[\div(\bdmc{S}_{D_m}^{i,\cnt{B}}-\bdmc{S}_{D_m}^{0,\cnt{B}})\nabla](V)(x)}&\leq\frac{1}{4\pi}\int_{D_m} \abs{\nabla V_l(y)}~dy\leq \frac{1}{\sqrt{3}}\frac{\rad^\frac{3}{2}}{4}\norm*{ V_l}_{\Hs{1}{D_m}}\\
	\abs*{\nabla[\div(\bdmc{S}_{D_m}^{i,\cnt{B}}-\bdmc{S}_{D_m}^{0,\cnt{B}})\nabla](V)(x)}&\leq \int_{D_m}\frac{1}{4\pi\abs*{x-y}} 
	\Bigl[1+\bigl(2+\frac{1}{\abs*{x-y}} \Bigr)\abs*{x-y}\Bigr]\abs*{\nabla V_l(y)}~dy,\\
	&\leq \int_{D_m}\frac{1}{2\pi}\Bigl[1+\frac{1}{\abs*{x-y}}\Bigr] \abs*{\nabla V_l(y)}~dy\leq \frac{2}{\sqrt{\pi}}\sqrt{\rad} \norm*{V_l}_{\Hs{1}{D_m}},\end{aligned}\right.\end{align} being, obviously, \begin{align}\label{Polar-Coordi-Integral}
	\int_{D_m}\frac{1}{\abs*{x-y}^2}~dy\leq\lim_{r\rightarrow0}\int_{B(y,{\rad})\setminus B(y,r)}\frac{1}{\abs*{x-y}^2}~dy \leq \lim_{r\rightarrow0} \int_{0}^{2\pi}\int_{0}^\pi\int_r^{\rad} \frac{1}{R^2}~R^2dR\sin(\theta)
	d\theta~d\phi\leq 2\pi\rad.\end{align}
	We write
	\begin{equation}	
	(x-z_i)=[I-\div\bdmc{S}_{D_m}^{i,\cnt{B}}](V)=[I-\div\bdmc{S}_{D_m}^{i,\cnt{B}}](V)+[\div(\bdmc{S}_{D_m}^{i,\cnt{B}}-\bdmc{S}_{D_m}^{k,\cnt{B}})\nabla](V),
	\end{equation}  
	then from \eqref{Estimates_I-divS_{D_m}{i,C}} and \eqref{Estim_divS_{D_m}{i,CA}-comp-rem}, we get 
	\begin{align*}	
	\norm*{(x-z_i)}_{\Hs{1,Q}{D_m}}&\geq \norm*{[I-\div\bdmc{S}_{D_m}^{i,\cnt{B}}](V)}_{\Hs{1,Q}{D_m}}-\norm*{[\div(\bdmc{S}_{D_m}^{i,\cnt{B}}-\bdmc{S}_{D_m}^{i,\cnt{B}})\nabla](V)}_{\Hs{1,Q}{D_m}},\\
	&\geq \frac{\norm*{V}_{\Hs{1}{D_m}}}{2}-\frac{\sqrt{3}\rad^2}{\sqrt{\pi}}\norm*{\cnt{}}_{{\Lp{\infty}{D_m}}}\norm*{V}_{\Hs{1}{D_m}}.
	\end{align*} The remaining part of the proof is due to the fact that, for $\rad$ sufficiently small, $\ANPT{B}{D_m}$ inherits its property from $\ANPTBdr{B}{D_m}$ through \Cref{Aniso-Pol-tens-Bdr-Positiveness}.
	
\end{proof}
Finally, we have the following, obvious, statement.
\begin{proposition}
	For $ U_m^{\cnt{B}}:=\cnt{B}\nabla V_m^{\cnt{B}},$ we have
	\begin{equation}\label{Estimates-Ani-PolTen-intermAnisBdr}
	\int_{D_m}U_m^{\cnt{B}}dv=\int_{D_m}\cnt{B}\nabla V_m^{\cnt{B}}dv=\ANPT{B}{D_m},
	\end{equation} for $\cnt{B}$ symmetric.  
	If $B$ is a constant matrix and not necessarily symmetric, we have 
	\begin{equation}
	\begin{aligned} 
	\int_{D_m}U_m^{\cnt{B}}dv=\cnt{B}\ANPT{B}{D_m}\cnt{B}^{-1}.
	\end{aligned}  
	\end{equation} 		
\end{proposition}

\subsection{Lippmann-Schwinger integral formulation and apriori estimates.} 
 In this section, we establish the wellposedness of our problem (\ref{Maxwell-Eq-Anisotropic}). We start with showing the uniqueness of its solution and then prove its existence using a 
 Lippmann-Schwinger integral representation of 
 the electromagnetic field and at the same time we give an estimate of the total field taking into account the cluster of our small inhomogeneities.
 \bigskip
 
 We first recall that the contrast of magnetic permeability and electric permittivity, which are assumed to be respectively real and complex valued $3\times3$-tensor, 
 essentially uniformly bounded with their derivatives, i.e.
 \begin{align}
 (\norm{\cnt{A}}_{\mathbb{W}^{1,\infty}(\cup_{i=1^m}D_m)})_{A=\prv_r,~\prb_r}\leq \bs{c}_\infty,  
 \end{align} and essentially uniformly coercive, that is, for almost every $x\in D,$ we have
 \begin{equation}\label{Coercivity-Cond-Contrast}
 \begin{aligned}
 \Re\Bigl(\cnt{\prv_r}(x)U\cdot \overline{U}\Bigr)&\geq c^{\prv-}_\infty\abs{U}^2,\\
 \cnt{\prb_r}(x)U\cdot U &\geq c^{\prb-}_\infty\abs{U}^2. 
\end{aligned} \end{equation} 
A sufficient condition to get the above inequality for the contrast of the relative electric permittivity is given by
$$\rho^-(\Re\cnt{\prv_r})-\rho^+(\Im\cnt{\prv_r})>0,$$
where $\rho^+(A)$ and $\rho^-(A)$ stand respectively, for the largest and the smallest eigenvalue of $A$. As $\cnt{\prb_r}$ is a real valued $3\times3$-tensor, it suffices for it to be definite positive. \\ 
Indeed, we have
\begin{equation}\label{AnisCase-Cond-for-Esti-Proof}\begin{aligned}
\Re\scalar{V,\ol{\cnt{\prv_r}V}}=\Re\scalar{\ol{V},\cnt{\prv_r}V}&=
\scalar{\Im V,\Re\cnt{\prv_r}\Im V}+\scalar{\Re V,\Re\cnt{\prv_r}\Re V}\\&\hspace{3cm}+\Bigl(\scalar{\Re V,\Im\cnt{\prv_r}\Im V}-\scalar{\Im V,\Im\cnt{\prv_r}\Re V}\Bigr),\\
&\geq \rho^+(\Re\cnt{\prv_r})(\norm{\Im V}^2+\norm{\Re V}^2)-2\rho^-(\Im\cnt{\prv_r})(\norm{\Im V}\norm{\Re V}),\\
&\geq \Bigr(\rho^-(\Re\cnt{\prv_r})-\rho^+(\Im\cnt{\prv_r})\Bigl)\norm{ V}^2. 
\end{aligned}
\end{equation}
\begin{proposition}
	The Problem  \eqref{Maxwell-Eq-Anisotropic} admits a unique solution. 
\end{proposition}
\begin{proof}
 We suppose that $\prv_r$ and $\prb_r$ are $3\times3$-tensors. Let ($E=E_1-E_2$, $H=H_1-H_2$) be the difference of two solutions of the problem \eqref{Maxwell-Eq-Anisotropic}. 
 Obviously both the normal trace of $E$ and $H$ are continuous across the boundary and satisfy the Silver–M\"uller radiation condition, hence, due to Green's formula, we have 
\begin{equation}
\int_{\p D}\nu\times \DTr{E}{\pm}\cdot \ol{H} ds=ik\int_{ D}\prb_r H\cdot \ol{H}dv+\ol{ik}\int_{ D} E \cdot \prv_r \ol{E}dv=ik\Bigl(\int_{ D}\prb_r H\cdot \ol{H}dv-\int_{ D}  E \cdot\prv_r \ol{E}dv\Bigr),
\end{equation} and taking the real part gives
\begin{equation}\label{Uni-Green-Conclusion}\begin{aligned}
-\Re\int_{\p D}\nu\times E\cdot \ol{H} ds=&\int_{ D}(\prb_r-\prb_r^*) \Im H\cdot \Re{H}dv+\int_{ D}(\prv_r-\prv_r^*) \Im E\cdot \Re{E}dv.
\end{aligned}\end{equation} Furthermore, we have, 
$\Im(H)\times\Re(H)=\Im(E)\times\Re(E)=0,$\footnote{Due to the fact that $H\times H=E\times E=0$.} and, 
with $(\prv_r)_{ij}$ standing for the component of $\prv_r$, we have 
\begin{equation}(\prv_r-\prv_r^*)V=\stackrel{U(\prv_r):=}{\overbrace{\left(\begin{matrix}
&(\prv_r)_{12}-(\prv_r)_{21}\\
&(\prv_r)_{13}-(\prv_r)_{31}\\
&(\prv_r)_{23}-(\prv_r)_{32} 
\end{matrix}\right)}}\times V,
\end{equation}   which guaranties that $(\prb_r-\prb_r^*) \Im H\cdot \Re{H}=(U(\prv_r)\times\Im(H))\cdot\Re H=0.$ 
The same observation concerning the second integral of the right-hand side of \eqref{Uni-Green-Conclusion}, 
implies that $$\Re\int_{\p D}\nu\times E\cdot \ol{H} ds=0.$$ The Rellich lemma (see \cite{ColtonKress:2013}) induces that $E\equiv H\equiv 0$. \end{proof}

To prove the exitence of the solution and derive the needed estimates, we use the equivalent Lippmann-Schwinger equation. For that, we define the operator of  Lippmann-Schwinger to be 
\begin{equation}
\mc{LS}^{({\prv_r},{\prb_r})}:=\left(\begin{matrix}
I-(k^2+\nabla\div)\bdmc{S}_{D}^{k,{\cnt{\prv_r}}}& -ik\curl\bdmc{S}_{D}^{k,{\cnt{\prb_r}}}\\
+ik\curl\bdmc{S}_{D}^{k,{\cnt{\prv_r}}} &I-(k^2+\nabla\div)\bdmc{S}_{D}^{k,{\cnt{\prb_r}}}
\end{matrix}\right).
\end{equation}

We set $(E,H)$ to be a solution, provided that it is solvable, of the integral equation,
\begin{equation}\label{Lip-sch-Eq-Elec-Magn}
\mc{LS}^{({\prv_r},{\prb_r})}\left(\begin{matrix}
E\\ H \end{matrix}\right)= \left(\begin{aligned} 
E^\n\\H^\n\end{aligned}\right)\tag{$\mc{M}_{\mc{L}.\mc{S}}$}. 
\end{equation} and $(E_{\mc{A}},H_{\mc{A}})$ the solution of the  Lippmann-Schwinger equation with averaged parameter, that is\begin{equation}\label{Lip-sch-Eq-Elec-Magn-Averaged-Param}
\mc{LS}^{(\mean{\prv_r},\mean{\prb_r})}\left(\begin{matrix}
E_\mc{A}\\ H_\mc{A} \end{matrix}\right) 
= \left(\begin{aligned} 
E^\n\\H^\n\end{aligned}\right)\tag{$\mc{A}-\mc{M}_{\mc{L}.\mc{S}}$}. 
\end{equation} 

With the previous notations, we have the following proposition.
\begin{proposition}\label{Estim-Fields-Aniso-Case}
A solution of the equation \eqref{Lip-sch-Eq-Elec-Magn} for $\Re(k)\geq0,~ \Im(k)\geq0$ solves the problem \eqref{Maxwell-Eq-Anisotropic}. Further, under the conditions  \eqref{Coercivity-Cond-Contrast}, the operator $\mc{LS}^{({\prv_r},{\prb_r})}$ is an isomorphism of $\mathcal{H}(\curl,\cup_{m=1}^{\nbs} D_m)$ provided that \begin{equation}\label{Repartition-Condition}
\bs{c}_r:=\frac{\dist}{\rad}\geq \frac{{c_0}2\abs{k}\bs{c}_\infty^2}{\max(c^{\prv^-}_\infty,c^{\prb^-}_\infty)}, \end{equation} and $V:=E,H,$ satisfy the following estimates
\begin{align}\label{Esim-Lip-Schwin-Solution}
\norm*{V}_{\Lp{2}{\cup_{m=1}^{\nbs} D_m)}}\leq C \norm*{(E^\n,  H^\n)}_{\Hcurl{}{\cup_{m=1}^{\nbs} D_m)}}.
\end{align} 
Besides for $ \cnt{\prb_r},$ $\cnt{\prv_r}$ in $\mathbb{W}^{1,\infty}(\cup_{i=1^m}D_m)$ and $(E_{\mathcal{A}}, H_{\mc{A}})$  solution of \eqref{Lip-sch-Eq-Elec-Magn-Averaged-Param} then \begin{equation}\begin{aligned}
		\norm{E_{\mc{A}}-E}_{\Lp{2}{\cup D_m}}\leq \bs{c}_{2,k} c_{\infty}\rad\Bigl(\norm{H}_{\Lp{2}{\cup D_m}}+\norm{E}_{\Lp{2}{\cup D_m}}\Bigr),\\
	\norm{H_{\mc{A}}-H}_{\Lp{2}{\cup D_m}}\leq \bs{c}_{2,k} c_{\infty}\rad \Bigl(\norm{H}_{\Lp{2}{\cup D_m}}+\norm{E}_{\Lp{2}{\cup D_m}}\Bigr),
\end{aligned}\label{Esimate-AveragedE-E}\end{equation} for some positive constant which depends only on $k$. 
\end{proposition} 
To prove this proposition, we need the following lemma.
\begin{lemma}\label{Elec-Magn-Lip-sch-Comp-and-DefPosi} The potential $\MSVP{D,V}{i\alpha,{\cnt{A}}}(V):=\MSVP{D}{i\alpha,{\cnt{A}}}(V):=(({i\alpha})^2+\nabla\div) \bdmc{S}_{D}^{i\alpha,{\cnt{A}}}(V)$ solves $$\bigl(\curl^2 \MSVP{D}{i\alpha,{\cnt{A}}}+\alpha^2\MSVP{D}{i\alpha,{\cnt{A}}}\bigr)(V)=-\alpha^2\cnt{A}V$$ and we have for $\alpha\geq 0$
\begin{equation}\label{Lip-Sch-Defint-Positive-part}
-\int_{D}\MSVP{D,V}{i\alpha,{\cnt{A}}}\cdot \ol{\cnt{A}V}dv\geq \norm*{\MSVP{D,V}{i\alpha,{\cnt{A}}}}_{\Hcurl{}{\R^3}}\geq\alpha^2\norm*{\MSVP{D,V}{i\alpha,{\cnt{A}}}}_{\Lp{2}{\R^3\setminus D}}^2\geq 0.	
\end{equation} Further, both $\MCVP{D}{k,{\cnt{A}}}:=ik\curl\bdmc{S}_{D}^{i\alpha,{\cnt{A}}}$ and $\MSVP{D}{i\alpha,{\cnt{A}}}-\MSVP{D}{k,\cnt{A}}$ are compact operators, and it holds that 
\begin{align}
\int_{ D}(\MSVP{D}{i\alpha,{\cnt{A}}}-\MSVP{D}{k,\cnt{A}})(V)\cdot \ol{\cnt{A}V} dv\leq& \frac{\abs{k}(\abs{k}+1)}{4\pi} \Bigl((\abs{k}+5) \rad^\frac{5}{2}+2^3{c_0}\frac{\abs{k}+5}{(1+2\bs{c}_r)^2\bs{c}_r}\Bigr)\norm{\cnt{A}V}^2_{\Lp{2}{D_m}},\label{Estim-Maxwell-Cmpct-Rem}\\
\int_{D}\MCVP{D}{k,{\cnt{A_1}}}(V_1)\cdot \ol{\cnt{A_2}V_2} dv\leq& \Bigl((\frac{1}{2}+\frac{k}{4}\rad)\rad+{c_0}\frac{(1+\abs{k})}{4\pi\bs{c}_r^3}\Bigr) \norm{\cnt{A_2}V_2}_{\Lp{2}{D}}\norm{\cnt{A_1}V_1}_{\Lp{2}{D}}, \label{Estim-Maxwell-Cmpct-Ope}
\end{align} where we recall that $\bs{c}_r=\frac{\dist}{\rad}.$ 
\end{lemma} 
\begin{proof} First, we write
\begin{equation}\label{Departure-Estimate-M_i,D,P}\begin{aligned} 
\int_{D}(\MSVP{D,V}{i\alpha,{\cnt{A}}}\cdot \ol{\cnt{A}V})dv&=-\int_{D}\MSVP{D,V}{i\alpha,{\cnt{A}}}\cdot\Bigl(\curl^2 \ol{\MSVP{D,V}{i\alpha,{\cnt{A}}}} +\alpha^2\ol{\MSVP{D,V}{i\alpha,{\cnt{A}}}}\Bigr) dx,\\&=-\int_{D}\MSVP{D,V}{i\alpha,{\cnt{A}}}\cdot\curl^2 \ol{\MSVP{D,V}{i\alpha,{\cnt{A}}}}dx-\alpha^{2}\norm*{\MSVP{D,V}{i\alpha,{\cnt{A}}}}_{\Lp{2}{D}}^2.
\end{aligned} 	
\end{equation} Due to Green's formula inside $D$, we have
\begin{equation}\label{Ext-Trace-Estimate-M_i,D,P}\begin{aligned} 
\int_{\p D}\nu\times\MSVP{D,V}{i\alpha,{\cnt{A}}}\cdot\curl \ol{\MSVP{D,V}{i\alpha,{\cnt{A}}}}ds=\norm*{\curl\MSVP{D,V}{i\alpha,{\cnt{A}}}}_{\Lp{2}{D}}^2-\int_{ D}\MSVP{D,V}{i\alpha,{\cnt{A}}}\cdot\curl^2 \ol{\MSVP{D,V}{i\alpha,{\cnt{A}}}}dx.
\end{aligned} 	
\end{equation}
As, outside of $D$, $\MSVP{D,V}{i\alpha,{\cnt{A}}}$ satisfies the Maxwell equation for $k=i\alpha,$ a direct application of Green's identity outside of $D$, for a sufficiently large $R>0,$\footnote{The continuity of the normal trace of $\MSVP{D,V}{i\alpha,{\cnt{A}}}$ across $\p D$ is due the facts that the operator $\nu\times\nabla$ is an isomorphism From $\Hs{s}{\p D}\setminus\R$ to $\Hs{1-s}{\p D}\setminus\R$ and $\div \bdmc{S}_{D}^{i\alpha,{\cnt{A}}}(\cdot)$ have a continuous Dirichlet trace.} implies that 
\begin{equation*}\begin{aligned} 
\int_{\p B_R}\nu\times\MSVP{D,V}{i\alpha,{\cnt{A}}}&\cdot\curl \ol{\MSVP{D,V}{i\alpha,{\cnt{A}}}}ds\\
&-\int_{\p D}\nu\times\MSVP{D,V}{i\alpha,{\cnt{A}}}\cdot\curl \ol{\MSVP{D,V}{i\alpha,{\cnt{A}}}}ds=\norm*{\curl\MSVP{D,V}{i\alpha,{\cnt{A}}}}_{\Lp{2}{B_R\setminus D}}^2+\alpha^2\norm*{\MSVP{D,V}{i\alpha,{\cnt{A}}}}_{\Lp{2}{B_R\setminus D}}^2
\end{aligned}\end{equation*} which guaranties\footnote{The obvious thing is that $\Re\int_{\p B_R}\nu\times\MSVP{D,V}{i\alpha,{\cnt{A}}}\cdot\curl \ol{\MSVP{D,V}{i\alpha,{\cnt{A}}}}ds\substack{{\longrightarrow 0}\\{R\longrightarrow\infty}},$ due to the exponential decay of the kernel as $\alpha>0$.} that 	
\begin{equation*}\begin{aligned} 
-\Re\Bigl(\int_{\p D}\nu\times\MSVP{D,V}{i\alpha,{\cnt{A}}}\cdot\curl \ol{\MSVP{D,V}{i\alpha,{\cnt{A}}}}ds\Bigr)\geq\norm*{\curl\MSVP{D,V}{i\alpha,{\cnt{A}}}}_{\Lp{2}{\R^3\setminus D}}^2+\alpha^2\norm*{\MSVP{D,V}{i\alpha,{\cnt{A}}}}_{\Lp{2}{\R^3\setminus D}}^2.
\end{aligned} 	
\end{equation*} Further, we have 
\begin{equation*}
\begin{aligned} 
\Im\Bigl(\int_{\p B_R}\nu_{_{\p B_R}}\times\MSVP{D,V}{i\alpha,{\cnt{A}}}\cdot\curl \ol{\MSVP{D,V}{i\alpha,{\cnt{A}}}}~ds\Bigr)-\Im\Bigl(\int_{\p D}\nu\times\MSVP{D,V}{i\alpha,{\cnt{A}}}\cdot\curl \ol{\MSVP{D,V}{i\alpha,{\cnt{A}}}}~ds\Bigr)=0,
\end{aligned}
\end{equation*} which, with the radiation condition, gives
\begin{equation}\label{Estimation-M_i,D,P}\begin{aligned} 
-\int_{\p D}\nu\times\MSVP{D,V}{i\alpha,{\cnt{A}}}\cdot\curl \ol{\MSVP{D,V}{i\alpha,{\cnt{A}}}}ds \geq \norm*{\curl\MSVP{D,V}{i\alpha,{\cnt{A}}}}_{\Lp{2}{\R^3\setminus D}}^2 +\alpha^2\norm*{\MSVP{D,V}{i\alpha,{\cnt{A}}}}_{\Lp{2}{\R^3\setminus D}}^2.
\end{aligned} 	
\end{equation}
The above inequality in \eqref{Ext-Trace-Estimate-M_i,D,P} gives
\begin{equation*} 
\norm*{\curl\MSVP{D,V}{i\alpha,{\cnt{A}}}}_{\Lp{2}{\R^3\setminus D}}^2 +\alpha^2\norm*{\MSVP{D,V}{i\alpha,{\cnt{A}}}}_{\Lp{2}{\R^3\setminus D}}^2\leq-\norm*{\curl\MSVP{D,V}{i\alpha,{\cnt{A}}}}_{\Lp{2}{D}}^2+\int_{ D}\MSVP{D,V}{i\alpha,{\cnt{A}}}\cdot\curl^2 \ol{\MSVP{D,V}{i\alpha,{\cnt{A}}}}dv.
\end{equation*} 
Then this last inequality in \eqref{Departure-Estimate-M_i,D,P} ends the proof of \eqref{Lip-Sch-Defint-Positive-part}.\\

Let us now prove \eqref{Estim-Maxwell-Cmpct-Ope} and \eqref{Estim-Maxwell-Cmpct-Rem}. We write, for $\alpha=1,$  \begin{equation}\label{Estim-Cmpct-Remainder-Max-Op}\begin{aligned}\norm{[\MSVP{D,V}{i,{\cnt{A}}}-\MSVP{D,V}{k,{\cnt{A}}}]}_{\Lp{2}{D_m}}\leq 
&\Bigl(\sum_{m=1}^\nbs\norm{[\MSVP{D_m,V}{i,{\cnt{A}}}-\MSVP{D_m,V}{k,{\cnt{A}}}]}^2_{\Lp{2}{D_m}}\Bigr)^{\frac{1}{2}}\\&+\Bigl(\sum_{m=1}^\nbs\norm{[\MSVP{D\setminus D_m,V}{i,{\cnt{A}}}-\MSVP{D\setminus D_m,V}{k,{\cnt{A}}}]}_{\Lp{2}{D_m}}\Bigr)^{\frac{1}{2}}.\end{aligned}\end{equation}
Obviously, we have, due to \eqref{Phi_k-Phi_i}, that \[[\MSVP{D/D_m,V}{i,{\cnt{A}}}-\MSVP{D/D_m,V}{k,{\cnt{A}}}](x)=Int_1^m+Int_2^m\] where
\begin{align*}
 \begin{aligned}
 Int_1^m:=&k^2\int_{ D\setminus D_m}\frac{e^{ik\abs{x-y}}}{4\pi\abs{x-y}}(\cnt{A}V)(y)dy-\int_{D\setminus D_m}\frac{e^{-\abs{x-y}}}{4\pi\abs{x-y}}(\cnt{A}V)(y)dy\\
 \mbox{ and }\end{aligned}\end{align*} and
 \begin{equation*} 
Int_2^m:=\int_{D\setminus D_i}\int_{0}^1\frac{e^{\bigl((ik+1)t-1\bigr)\abs*{x-y}} \bigl((ik+1)t-1\bigr)}{4\pi(ik-1)^{-1}\abs*{x-y}}   
\Bigl[I+\bigl((ik+1)t-1-\frac{1}{\abs*{x-y}} \Bigr)\frac{\tensor{x-y}{2}}{\abs*{x-y}}\Bigr]dt(\cnt{A}V)(y)dy.
\end{equation*} 
We have, using H\"older's inequality, \footnote{We have considered $\abs{k}>1.$}    
\begin{align*}
\abs{Int_1^m}&\leq \frac{(\abs{k}^2+1)}{4\pi}\sum_{\stl{j\geq 1}{j\neq m}}^{\nbs}\frac{\rad^\frac{3}{2}}{\dist_{mj}}\norm*{\cnt{A}V}_{\Lp{2}{D_j}},
\end{align*} and 
\begin{align*}
\abs{Int_2^m}&\leq  \int_{D\setminus D_m}\frac{ \abs{k}(\abs{k}+1)}{4\pi} \Bigl[\frac{2}{\abs*{x-y}}+\abs{k}\Bigr]\abs{\cnt{A}V}(y)dy,\\
&\leq \sum_{\stl{j\geq 1}{j\neq m}}^{\nbs}\frac{ \abs{k}(\abs{k}+1)}{4\pi} \Bigl[\frac{2}{\dist_{jm}}+\abs{k}\Bigr]\rad^{\frac{3}{2}}\norm{\cnt{A}V}_{\Lp{2}{D_j}},
\end{align*}  which helps conclude, using H\"older's inequality for the second step once more and (\ref{Double-Sum-Counting}) of Lemma \ref{Counting-Oclusion}, that  
\begin{equation}\label{Estim-star1}
\begin{aligned}
\sum_{m=1}^{\nbs}\norm{\MSVP{D/D_m,V}{i,{\cnt{A}}}-\MSVP{D/D_m,V}{k,{\cnt{A}}}}^2_{\Lp{2}{D_m}}
\leq& \sum_{m=1}^{\nbs}\Bigl(\sum_{\stl{j\geq 1}{j\neq m}}^{\nbs}\frac{ \abs{k}(\abs{k}+1)}{4\pi} \Bigl[\frac{3}{\dist_{jm}}+\abs{k}\Bigr]\rad^{\frac{3}{2}}\norm{\cnt{A}V}_{\Lp{2}{D_j}}\Bigr)^2\rad^3,\\
\leq&{c_0}\Bigl(\frac{(\abs{k}+1)}{4\pi\abs{k}^{-1}}\Bigr)^2\nbs\rad^6\sum_{\stl{j\geq 1}{j\neq m}}^{\nbs} \Bigl[\frac{9}{\dist_{jm}^2}+\frac{2\abs{k}}{\dist_{jm}}+\abs{k}^2\Bigr]~ \norm{\cnt{A}V}_{\Lp{2}{D_j}}^2,\\
\leq&{c_0}\Bigl(\frac{(\abs{k}+1)}{4\pi\abs{k}^{-1}}\Bigr)^2\nbs\rad^6 \Bigl[\frac{9{\nbs}^\frac{1}{3}}{\dist^2}+\frac{2{\nbs}^\frac{2}{3}\abs{k}}{\dist}+\nbs\abs{k}^2\Bigr]~ \norm{\cnt{A}V}_{\Lp{2}{D}}^2.
\end{aligned}
\end{equation} Hence  	
\begin{equation}\label{Cal-Estm-MSoper-Ni-Nk}
\begin{aligned}
\sum_{m=1}^{\nbs}\norm{\MSVP{D/D_m,V}{i,{\cnt{A}}}-\MSVP{D/D_m,V}{k,{\cnt{A}}}}^2_{\Lp{2}{D_m}}\\
&\hspace{-3cm}\leq {c_0}\Bigl(\frac{(\abs{k}^2+\abs{k})}{4\pi}\Bigr)^2 \Bigl[\frac{9\rad^6}{(\rad/2+\dist)^4\dist^2}+\frac{2\rad^6\abs{k}}{(\rad/2+\dist)^5\dist}+\frac{\rad^6\abs{k}^2}{(\rad/2+\dist)^6}\Bigr]~ \norm{\cnt{A}V}_{\Lp{2}{D}}^2,\\
&\hspace{-3cm}\leq {c_0}\Bigl(\frac{(\abs{k}^2+\abs{k})}{4\pi}\Bigr)^2 \Bigl[\frac{9(2^4)}{(1+2\bs{c}_r)^4\bs{c}_r^2}+\frac{2^6\abs{k}}{(1+2\bs{c}_r)^5\bs{c}_r}+\frac{2^6\abs{k}^2}{(1+2\bs{c}_r)^6}\Bigr]~ \norm{\cnt{A}V}_{\Lp{2}{D}}^2,\\
&\hspace{-3cm}\leq 2^3{c_0}\Bigl(\frac{(\abs{k}^2+\abs{k})}{4\pi}\Bigr)^2 \Bigl[\frac{\abs{k}^2+4\abs{k}+18}{(1+2\bs{c}_r)^4\bs{c}^2_r}\Bigr]~ \norm{\cnt{A}V}_{\Lp{2}{D}}^2.
\end{aligned}
\end{equation}
We deal now with the first term of the right-hand side of \eqref{Estim-Cmpct-Remainder-Max-Op}. We write
\begin{equation*}\begin{aligned} 
\abs{[\MSVP{D_m,V}{i,{\cnt{A}}}-\MSVP{D_m,V}{k,{\cnt{A}}}](x)}\leq&
\int_{D_m}\frac{ \abs{k}(\abs{k}+1)}{4\pi}   \Bigl[\bigl(\frac{3}{\abs*{x-y}}+\abs{k} \Bigr)\Bigr]\abs{\cnt{A}V}(y)dy,\\
\leq& \frac{ \abs{k}(\abs{k}+1)}{4\pi}\Bigl[\Bigl(\int_{D_m}\frac{3}{\abs*{x-y}^2}dy\Bigr)^\frac{1}{2}+\abs{k}\rad^{\frac{3}{2}} \Bigr]\norm{\cnt{A}V}_{\Lp{2}{D_m}}
\end{aligned}
\end{equation*} and, as done in \eqref{Polar-Coordi-Integral}, we obtain
\begin{equation}\begin{aligned} 
\sum_{m=1}^\nbs\norm{[\MSVP{D_m,V}{i,{\cnt{A}}}-\MSVP{D_m,V}{k,{\cnt{A}}}](x)}^2_{\Lp{2}{D_m}}\leq& \sum_{m\geq 1}^\nbs\rad^3 \Bigl(\frac{\abs{k}(\abs{k}+1)}{4\pi}\Bigr)^2 \Bigl[\sqrt{6\pi\rad}+\abs{k}\rad^{\frac{3}{2}} \Bigr]^2\norm{\cnt{A}V}^2_{\Lp{2}{D_m}},\\
\leq&\Bigl(\frac{\abs{k}(\abs{k}+1)}{4\pi}\Bigr)^2\Bigl[\sqrt{6\pi\rad}+\abs{k}\rad^{\frac{3}{2}} \Bigr]^2\rad^3\norm{\cnt{A}V}^2_{\Lp{2}{\cup_{m\geq 1}D_m}}.
\end{aligned}\label{First-Rigt-Hand-member-CompRem-Estim}
\end{equation} Gathering \eqref{Cal-Estm-MSoper-Ni-Nk} and \eqref{First-Rigt-Hand-member-CompRem-Estim}, we have
\begin{equation}
\begin{aligned}\norm{[\MSVP{D,V}{i,{\cnt{A}}}-\MSVP{D,V}{k,{\cnt{A}}}]}_{\Lp{2}{D}}\leq&\frac{\abs{k}(\abs{k}+1)}{4\pi}\Bigl((\abs{k}+5)\rad^\frac{5}{2}+2^3{c_0}\frac{\abs{k}+5}{(1+2\bs{c}_r)^2\bs{c}_r}\Bigr)\norm{\cnt{A}V}_{\Lp{2}{D_m}}.
\end{aligned}\end{equation}\\
To derive \eqref{Estim-Maxwell-Cmpct-Ope}, we write

\begin{equation}\label{Estim-Cmpct-Max-Op-2}\begin{aligned},
\sum_{m=1}^\nbs\int_{D_m}\MCVP{D,V_1}{k,{\cnt{A_1}}}\cdot\ol{\cnt{A_2}V_2} dv=& \sum_{m=1}^\nbs\int_{D_m}\MCVP{D_m,V_1}{k,{\cnt{A_1}}}\cdot\ol{\cnt{A_2}V_2} dv +\sum_{m=1}^\nbs \int_{D_m}\sum_{\stackrel{j\geq 1}{j\neq m}}^\nbs\MCVP{D_j,V_1}{k,{\cnt{A_1}}}\cdot\ol{\cnt{A_1}V_2}dv.\end{aligned}\end{equation} The first term of the second member can be estimated using young's inequality 
$$\abs{\int_{\R^3}u(x) (v*w)(x)dx}\leq \norm{u}_{\Lp{2}{\R^3}}\norm{w}_{\Lp{2}{\R^3}}\norm{v}_{\Lp{1}{\R^3}},$$
with $u(x)=\abs{\cnt{A_2}V_2}(x)\ind{D_m}(x), v(x)=\abs{\nabla\Phi_k}(x)\ind{B(0,\rad)}(x), w(x)=\abs{\cnt{A_1}V_1}(x)\ind{D_m}(x)$. We get\footnote{Notice that $D_m\subset B(\z_m,\frac{\rad}{2})\subset B(y,\rad)$ whenever $y\in D_m,$ and $\ind{B(0,\rad)}(x-y)=\ind{B(y,\rad)}(x).$}  
\begin{equation}\label{Young-inequality-Max-Cmpct-Ope}\begin{aligned}
\sum_{m=1}^\nbs\int_{R^3}\abs{\cnt{A_2}V_2}(x)\ind{D_m}(x)\int_{{R^3}}&\abs{\nabla\phi_k(x-y)}\ind{B(0,\rad)}(x-y)~\abs{\cnt{A_1}V_1}(y)\ind{D_m}(y) dy~ dx\\
\leq& \sum_{m=1}^\nbs\norm*{\cnt{A_2}V_2\ind{D_m}}_{\Lp{2}{\R^3}} \norm*{\cnt{A_1}V_1 \ind{D_m}}_{\Lp{2}{\R^3}} \norm{\nabla\Phi_k\ind{B(0,\rad)}}_{\Lp{1}{\R^3}},\\
\leq& \norm*{\cnt{A_2}V_2}_{\Lp{2}{D}}\norm*{\cnt{A_1}V_1}_{\Lp{2}{D}}\int_{B(0,\rad)}\frac{1}{4\pi}\Bigl(\frac{1}{\abs{x-y}^2}+\frac{\abs{k}}{\abs{x-y}}\Bigr)dx,\\
\leq&(\frac{1}{2}+\frac{k}{4}\rad)\rad\norm*{\cnt{A_2}V_2}_{\Lp{2}{D}}\norm*{\cnt{A_1}V_1}_{\Lp{2}{D}}.
\end{aligned}\end{equation}
The later sum in the right hand side of \eqref{Estim-Cmpct-Max-Op-2} is smaller than
\begin{equation}\label{Calculations-dist-ij}
S_1:=\sum_{m=1}^\nbs \int_{D_m}\sum_{\stackrel{j\geq 1}{j\neq m}}^\nbs\int_{D_j}\frac{1}{\dist_{mj}}(\frac{1}{\dist_{mj}}+\abs{k})~\abs{\cnt{A_1}V_1}~\abs{\cnt{A_2}V_2}dv,
\end{equation} and similar calculation, as done above and using (\ref{Double-Sum-Counting}) of Lemma \ref{Counting-Oclusion}, gives successively  \begin{align*}
S_1\leq& \sum_{m=1}^\nbs \rad^3\norm{\cnt{A_2}V_2}_{\Lp{2}{D_m}}\Bigl(c_0(\frac{\nbs^\frac{1}{3}}{\dist^2} +\frac{\abs{k}\nbs^\frac{2}{3}}{\dist})\Bigr)^\frac{1}{2}\Bigl(\sum_{\stackrel{j\geq 1}{j\neq m}}^\nbs(\frac{1}{\dist^2_{ij}} +\frac{\abs{k}}{\dist_{mj}})\norm{\cnt{A_1}V_1}^2_{\Lp{2}{D_j}}\Bigr)^\frac{1}{2},\\
\leq& \Bigl(c_0(\frac{\nbs^\frac{1}{3}}{\dist^2}+\frac{\abs{k}\nbs^\frac{2}{3}}{\dist})\Bigr)^\frac{1}{2} \rad^3\Biggl(\norm{\cnt{A_2}V_2}^2_{\Lp{2}{D}} \sum_{m=1}^\nbs\sum_{\stackrel{j\geq 1}{j\neq m}}^\nbs(\frac{1}{\dist^2_{ij}} +\frac{\abs{k}}{\dist_{mj}})\norm{\cnt{A_1}V_1}^2_{\Lp{2}{D_j}}\Biggr)^\frac{1}{2},\\
\leq& c_0(\frac{\nbs^\frac{1}{3}}{\dist^2} +\frac{\abs{k}\nbs^\frac{2}{3}}{\dist})\rad^3\norm{\cnt{A_2}V_2}_{\Lp{2}{D}}\norm{\cnt{A_1}V_1}_{\Lp{2}{D}},
\end{align*} which implies that
\begin{equation}\label{Cal-Estm-MCoper-Mk}
\abs{\sum_{m=1}^\nbs\int_{D_m}\MCVP{D}{k,{\cnt{A_1}}}(V_1)\cdot\cnt{A_2}V_2 dv}\leq {c_0}\frac{(1+\abs{k})}{4\pi\bs{c}_r^3} \norm{\cnt{A_2}V_2}_{\Lp{2}{D}}\norm{\cnt{A_1}V_1}_{\Lp{2}{D}}.
\end{equation}
\end{proof}	
\begin{proof}(\Cref{Estim-Fields-Aniso-Case})
Consider the equation \eqref{Lip-sch-Eq-Elec-Magn}, which is, with the notation of \Cref{Elec-Magn-Lip-sch-Comp-and-DefPosi} and $\scalar{\cdot,\cdot}$ standing for the scalar product in $\Lp{2}{D}$,  
\begin{equation*}
\begin{aligned} \scalar{E,\ol{\cnt{\prv_r}E}}-\scalar{\MSVP{D,E}{i\alpha,\cnt{\prv_r}},\ol{\cnt{\prv_r}E}}-\scalar{\MSVP{D,E}{k,\cnt{\prv_r}}-\MSVP{D,E}{i\alpha,\cnt{\prv_r}},\ol{\cnt{\prv_r}E}}
- \scalar{ik \MCVP{D,H}{k,{\cnt{\prb_r}}}, \ol{\cnt{\prv_r}E}}=& \scalar{E^\n\cdot \ol{\cnt{\prv_r}E}},\\
\scalar{H,\ol{\cnt{\prb_r}H}}-\scalar{\MSVP{D,H}{i\alpha,{\cnt{\prb_r}}},\ol{\cnt{\prb_r}H}}-\scalar{\MSVP{D,H}{k,{\cnt{\prb_r}}}-\MSVP{D,H}{i\alpha,{\cnt{\prb_r}}},\ol{\cnt{\prb_r}H}}
+\scalar{ik\MCVP{D,E}{k,{\cnt{\prv_r}}},\ol{\cnt{\prb_r}H}} =& \scalar{H^\n,\ol{\cnt{\prb_r}H}},
\end{aligned}  
\end{equation*} and with \eqref{Lip-Sch-Defint-Positive-part}, taking the real part of the above equation, we get
\begin{equation}
\begin{aligned} \Re\scalar{E,\ol{\cnt{\prv_r}E}}-\Re\scalar{\MSVP{D,E}{k,\cnt{\prv_r}}-\MSVP{D,E}{i\alpha,\cnt{\prv_r}},\ol{\cnt{\prv_r}E}}
- \Re\scalar{ik \MCVP{D,H}{k,{\cnt{\prb_r}}}, \ol{\cnt{\prv_r}E}}\leq& \Re\scalar{E^\n\cdot \ol{\cnt{\prv_r}E}},\\
\Re\scalar{H,\ol{\cnt{\prb_r}H}}-\Re\scalar{\MSVP{D,H}{k,{\cnt{\prb_r}}}-\MSVP{D,H}{i\alpha,{\cnt{\prb_r}}},\ol{\cnt{\prb_r}H}}
+\Re\scalar{ik\MCVP{D,E}{k,{\cnt{\prv_r}}},\ol{\cnt{\prb_r}H}} \leq& \Re\scalar{H^\n,\ol{\cnt{\prb_r}H}}.
\end{aligned}  
\end{equation} With the estimations \eqref{Estim-Maxwell-Cmpct-Rem}, \eqref{Estim-Maxwell-Cmpct-Ope} of \Cref{Elec-Magn-Lip-sch-Comp-and-DefPosi} and the assumption \eqref{Coercivity-Cond-Contrast} we get\footnote{ Assuming that $\abs{k}\rad\leq 1.$}
\begin{equation*}
\begin{aligned} c^{\prv^-}_\infty\norm{E}^2-c^2_\infty\frac{\abs{k}(\abs{k}+1)}{4\pi} \Bigl(\rad^\frac{3}{2}+2^3{c_0}\frac{\abs{k}+5}{(1+2\bs{c}_r)^2\bs{c}_r}\Bigr)&\norm{E}^2\\ -&\bs{c}_\infty^2\abs{k}\Bigl(\rad+{c_0}\frac{(1+\abs{k})}{4\pi\bs{c}_r^3}\Bigr)\norm{H} \norm{E}\leq \scalar{E^\n\cdot \cnt{\prv_r}E},\\
c^{\prb^-}_\infty\norm{H}^2 -c^2_\infty\frac{\abs{k}(\abs{k}+1)}{4\pi} \Bigl(\rad^\frac{3}{2}+2^3{c_0}\frac{\abs{k}+5}{(1+2\bs{c}_r)^2\bs{c}_r}\Bigr)&\norm{H}^2 \\ -&\bs{c}_\infty^2\abs{k}\Bigl(\rad+{c_0}\frac{(1+\abs{k})}{4\pi\bs{c}_r^3}\Bigr)\norm{H} \norm{E}\leq \scalar{H^\n,\cnt{\prb_r}H},
\end{aligned}  
\end{equation*} which, under the conditions \begin{equation}
\bs{c}_r\geq \frac{{c_0}2\abs{k}\bs{c}_\infty^2}{\max(c^{\prv^-}_\infty,c^{\prb^-}_\infty)},\,~\frac{\abs{k}(\abs{k}+1)}{4\pi}\rad<1, \mbox{ and } \rad\leq\frac{1}{16},
\end{equation} gives
\begin{equation*}
\norm{E}^2- \frac{1}{4\pi}\norm{H} \norm{E}\leq\frac{\bs{c}_\infty}{c^{\prv^-}_\infty} \norm{E^\n}\norm{E},
\end{equation*} and \begin{equation*}
\norm{H}^2 -\frac{1}{4\pi}\norm{H} \norm{E}\leq \frac{\bs{c}_\infty}{c^{\prb^-}_\infty}\norm{H^\n}\norm{H}.
\end{equation*} More precisely,
\begin{equation}\label{Etimates-Field-C_r-Condition}
\norm{E}  \leq\frac{5\bs{c}_\infty}{4c^{\prv^-}_\infty} \Bigl(\norm{E^\n}+\frac{1}{4\pi}\norm{H^\n}\Bigr),\end{equation} and
\begin{equation} \label{Etimates-Field-C_r-Condition1}
\norm{H} \leq\frac{5\bs{c}_\infty}{4c^{\prb^-}_\infty} \Bigl(\norm{H^\n}+\frac{1}{4\pi}\norm{E^\n}\Bigr).
\end{equation}

Concerning the estimate \eqref{Esimate-AveragedE-E}, identifying both left hand sides of the \cref{Lip-sch-Eq-Elec-Magn,Lip-sch-Eq-Elec-Magn-Averaged-Param}, we have
\begin{align*}
	(H_{\mc{A}}-H)-(k^2+\nabla\div)\bdmc{S}_{D}^{k,{\meancnt{\prb_r}}}(H_{\mc{A}}-H) + ik\curl \bdmc{S}_{D}^{k,{\meancnt{\prv_r}}}(E_{\mc{A}}-E) \\=-(k^2+\nabla\div)\bdmc{S}_{D}^{k,\cnt{\prb_r}-\meancnt{\prb_r}}(H) + ik\curl \bdmc{S}_{D}^{k,\cnt{\prv_r}-\meancnt{\prv_r}}(E), 
\end{align*} then, due to estimates \eqref{Esim-Lip-Schwin-Solution} for the solution of Lippmann-Schwinger integral equation, we obtain
\begin{align*}
\norm{H_{\mc{A}}-H}_{\Lp{2}{D}}
&\leq \norm{-(k^2+\nabla\div)\bdmc{S}_{D}^{k,\cnt{\prb_r}-\meancnt{\prb_r}}(H) + ik\curl\bdmc{S}_{D}^{k,\cnt{\prv_r}-\meancnt{\prb_r}}(E)}_{\Lp{2}{D}},\\
&\leq \bs{c}_{2,k}\Bigl(\norm{(\cnt{\prb_r}-\meancnt{\prb_r})H}_{\Lp{2}{D}} + 
\norm{ (\cnt{\prv_r}-\meancnt{\prb_r}) E }_{\Lp{2}{D}}\Bigr),\\
&\leq \bs{c}_{2,k} c_{\infty} \rad\Bigl(\norm{H}_{\Lp{2}{D}} + \norm{E }_{\Lp{2}{D}}\Bigr).
\end{align*} 	
\end{proof}	 
\begin{remark} We have two observations:
\begin{itemize}
\item   We can consider either both real tensor or complex valued ones for the relative electric permittivity and magnetic permeability inducing similar assumption concerning the corresponding contrast in \eqref{Coercivity-Cond-Contrast}.     
\item	We could improve the condition on the ratio  $\frac{\rad}{\dist}$  by taking a larger $\alpha$. But this would increase the constant that appears in the estimation \eqref{Esim-Lip-Schwin-Solution} which will result in the worsening of the error of approximation in the latter calculation.
\end{itemize}	
  
\end{remark}

\subsection{Field approximation and the related linear system}\label{Lin-Sys-Anisotropic}
We set 
\begin{equation}
\mc{Q}^A_{m}:=\int_{D_m}\cnt{B}H~dv,~~\mc{R}^{A}_{m}:=\int_{D_m}\cnt{B}E~dv
\end{equation} and write, with $(E_{\mc{A}},~H_{\mc{A}})$ solution of \eqref{Lip-sch-Eq-Elec-Magn-Averaged-Param}, \begin{equation}
\mc{Q}_{m}:=\int_{D_m}H_{\mc{A}}~dv,~~\mc{R}_{m}:=\int_{D_m}E_{\mc{A}}~dv.
\end{equation}
For $U_m^{\prb_r^*}=\cnt{\prb_r}^*\nabla V_m^{\prb_r^*}$ and $U_m^{\prv_r^*}=\cnt{\prv_r}^*\nabla V_m^{\prv_r^*},$ where $V_m^{\prb_r^*}, $ $V_m^{\prv_r^*}$ 
are the respective solutions of \eqref{Anisotropic-Pr-Lip-Sch} with $B=\prb_r^*$ and $B=\prv_r^*$, we recall that 
\begin{align}
U_m^{\prb_r^*}-\cnt{\prb_r}^*\nabla\div\bdmc{S}_{D_m}^{0}(U_m^{\prb_r^*})=\cnt{\prb_r}^* \label{Def-Magnetique-Perm-Tensor}\end{align} and
\begin{align}
U_m^{\prv_r^*}-\cnt{\prb_r}^*\nabla\div\bdmc{S}_{D_m}^{0}(U_m^{\prv_r^*})=\cnt{\prv_r}^*.\label{Def-Electric-Perm-Tensor}
\end{align}  Due to \Cref{Estim-Aniso-Pol-Ten}, we get 
\begin{equation}
\norm*{U_m^{\prb_r^*}}_{\Lp{2}{D_m}}\leq\norm*{\cnt{\prb_r}^*\nabla V}_{\Lp{2}{D_m}}\leq  c_{\infty}\rad^\frac{3}{2},
\end{equation} and
\begin{equation} 
\norm*{U_m^{\prv_r^*}}_{\Lp{2}{D_m}}\leq\norm*{\cnt{\prv_r}^*\nabla V}_{\Lp{2}{D_m}}\leq c_{\infty}\rad^\frac{3}{2},
\end{equation} where $$c_{\infty}:=\max_{m\leq\nbs}~\sup_{x\in D_m,\,B=\prv_r,\prb_r.} \frac{(\norm*{B}_{{\Lp{\infty}{D_m}}}+2)^2}{\left(\frac{1}{2}-\rad^2\norm*{B}_{{\Lp{\infty}{D_m}}}\sqrt{\frac{3}{\pi}}\right)}$$ is a positive constant.
 
The following assumption could be seen as a consequence of the scaling \eqref{Ani-PolTen-Scaling} and \Cref{Signe-Of-Anis-Pol-Ten}, for $A=\prv_r, \prb_r$,
\begin{align}\label{Ani-PolTens-Coercivity-RealCase}
{\mu_{B}^{-}}{\rad^3}\abs{V}^2\leq\ANPT{B}{D_m}V\cdot V\leq& {\mu_{B}^{+}}{{\rad^3}}\abs{V}^2	\hspace{1cm}\mbox{whenever $B\in \mathbb{W}^{\infty}(D_m,\R^3\times\R^3)$}, \end{align} and
\begin{align}
{\mu_{B}^{-}}{\rad^3}\abs{V}^2\leq\Re\Bigl(\ANPT{B}{D_m}V\cdot \overline{V}\Bigl)\leq& {\mu_{B}^{+}}{{\rad^3}}\abs{V}^2	\hspace{1cm}\mbox{whenever $B\in \mathbb{W}^{\infty}(D_m,\C^3\times\C^3)$},
\end{align}
where $\mu_{A}^{+}:=\max_m\mu_{A,m}^{+}$,  $\mu_{A}^{-}:=\min_m\mu_{A,m}^{-} $ and $(\mu_{A,m}^{\pm})_{m=1}^\nbs$ are  constants that satisfy the previous inequalities for each $D_m$.  

We set for $U,\,V \in\Lp{2}{D_m}$ 
\begin{equation} 
Er_m(U,V):=O\Biggl(\sum_{\stl{(j\geq 1)}{j\neq m}}^{\nbs}\biggl[\frac{\norm{V}_{\Lp{2}{D_j}}}{\dist_{mj}^4}+
\Bigl(\frac{3\abs{k}}{\dist_{mj}^3}+\frac{3\abs{k}^2+1}{\dist_{mj}^2}+\frac{\abs{k}(\abs{k}^2+1)}{\dist_{mj}}\Bigr) \Bigl(\norm{U}_{\Lp{2}{D_j}}+\norm{V}_{\Lp{2}{D_j}}\Bigr)\biggr]\rad^\frac{11}{2}\Biggr).\end{equation}

With \eqref{Double-Sum-Counting} of \Cref{Counting-Oclusion}, we obtain 
\begin{equation}\label{Error-Estim-Error_m}
\sum_{m=1}^{\nbs} Er_m(U,V)^2=O\Biggl(\norm{U}^2_{\Lp{2}{D}}\frac{\rad^{11}}{\dist^8}+\Bigl(\norm{U}^2_{\Lp{2}{D}}+\norm{V}^2_{\Lp{2}{\cup_m D_m}}\Bigr)\frac{(\abs{k}+2)^3\rad^{11}\abs{\ln(\dist)}}{\dist^6}\Biggr).
\end{equation}
The far field of the scattered wave is given by, see (\cite{ColtonKress:2013}, (6.26)-(6.27) p. 199),
\begin{equation}\label{Far-Field-Anisotropic-Case}
E^\infty(\hat{x})=\frac{k^2}{4\pi}\hat{x}\times\Bigl(\int_{\cup_{m=1}^{\nbs} D_m}e^{-ik\hat{x}\cdot y}\cnt{\prv_r}E(y)dy\times \hat{x}\Bigr)+\frac{ik}{4\pi}\hat{x}\times\int_{\cup_{m=1}^{\nbs} D_m}e^{-ik\hat{x}\cdot y}\cnt{\prb_r}H(y)dy.	
\end{equation} 
Furthermore, for the scattering of plane wave, \eqref{Esim-Lip-Schwin-Solution} gives, with \eqref{Etimates-Field-C_r-Condition} and \eqref{Etimates-Field-C_r-Condition1},
\begin{equation}\label{Estimation-For-Plan-Waves}
\begin{aligned}
\norm{E}_\Lp{2}{D}\leq& \frac{5\bs{c}_\infty}{4c^{\prv^-}_\infty} \Bigl(\abs{P}+\frac{\abs{k}}{4\pi}\abs{\theta\times P}\Bigr)\Bigr(\frac{1}{\bs{c}_r}\Bigl)^\frac{3}{2}\end{aligned}
\end{equation} and 
\begin{equation}\label{Estimation-For-Plan-Waves1}
\begin{aligned}\\ 
\norm{H}_\Lp{2}{D}\leq& \frac{5\bs{c}_\infty}{4c^{\prv^-}_\infty} \Bigl(\frac{1}{4\pi}\abs{P}+\abs{k}\abs{\theta\times P}\Bigr)\Bigr(\frac{1}{\bs{c}_r}\Bigl)^\frac{3}{2}.
\end{aligned}
\end{equation}
With the above estimates, we have the following result. 
\begin{proposition} 
For $\hat{x}\in \mathbb{S}^1,$ we have the following approximations for the far field of the scattered waves
\begin{enumerate}
	\item if $\cnt{\prb_r},$ $\cnt{\prv_r}$ are symmetric.
\begin{equation}\label{Far-Field-Approximation-Anisotropic-Symmetric}
\begin{aligned}E^\infty(\hat{x})=&\sum_{m=1}^{\nbs}\biggr(\frac{k^2}{4\pi}e^{-ik\hat{x}\cdot \z_m}\hat{x}\times\Bigl(\mc{R}_m^{\prv_r}\times \hat{x}\Bigr)+\frac{ik}{4\pi}e^{-ik\hat{x}\cdot \z_m}\hat{x}\times\mc{Q}_m^\prb_r\Biggl)+\Bigl(\frac{\abs{k}^3}{\bs{c}^3_r}\rad\Bigr),\\
\end{aligned}  
\end{equation}
   \item if $ \cnt{\prb_r},$ $\cnt{\prv_r}$ are not symmetric
\begin{equation}\label{Far-Field-Approximation-Anisotropic-Non-Symmetric}
\begin{aligned} E^\infty(\hat{x})=&\sum_{m=1}^{\nbs}\biggr(\frac{k^2}{4\pi}e^{-ik\hat{x}\cdot \z_m}\hat{x}\times\Bigl(\cnt{\prv_r}\mc{R}_m\times \hat{x}\Bigr)+\frac{ik}{4\pi}e^{-ik\hat{x}\cdot \z_m}\hat{x}\times\cnt{\prb_r}\mc{Q}_{m}\Biggl)\\
&+O\Biggl(\frac{\abs{k}^3+\abs{k}^2}{\bs{c}_r^3}\rad+\frac{\abs{k}\bs{c}_\infty(\bs{c}_\infty \bs{c}_{2,k}+1)}{\bs{c}_r^3}\rad\Biggr),\\
\end{aligned}  
\end{equation} 
\end{enumerate} uniformly for $\hat{x}\in \mathbb{S}^2,$ where $\Bigl(\mc{R}^{\prv_r}_m\Bigr)_{m=1}^\nbs$ and $\Bigl(\mc{Q}^{\prb_r}_m\Bigr)_{m=1}^\nbs$ are solutions of the following system
\begin{equation}
\begin{aligned}
\mc{Q}^{\prb_r}_m=\ANPT{\prb_r}{D_m}&\biggl( \sum_{\stackrel{j\geq 1}{j\neq m}}^{\nbs} \Bigl[ \Pi_{k}(\z_m,\z_j) \mc{Q}^{\prb_r}_j- ik\nabla\Phi_{k}(\z_m,\z_j)\times\mc{R}^{\prv_r}_j\Bigr]+H^\n(\z_m)\biggr)\\
&+Er_m(H,E)+O\Bigl(k^2\norm*{\cnt{\prb_r}H}_{\Lp{2}{D_m}}\rad^\frac{7}{2} +\norm*{\cnt{\prv_r}E}_{\Lp{2}{D_m}}\rad^\frac{5}{2}+\abs{k}^2\rad^4\Bigr),\\
\mc{R}^{\prv_r}_m=\ANPT{\prv_r}{D_m}&\biggl(\sum_{\stl{j\geq 1}{j\neq m}}^{\nbs}\Bigl[ \Pi_{k}(\z_m,\z_j) \mc{R}^{\prv_r}_j+ ik\nabla\Phi_{k}(\z_m,\z_j)\times\mc{Q}^{\prb_r}_j\Bigr]+\cdot E^\n(\z_m)\biggr)\\
&+Er_m(E,H)+O\Bigl(k^2\norm*{\cnt{\prv_r}E}_{\Lp{2}{D_m}}\rad^\frac{7}{2} +\norm*{\cnt{\prb_r}H}_{\Lp{2}{D_m}}\rad^\frac{5}{2}+\abs*{k}\rad^4\Bigr),
\end{aligned}\label{Approximation-Linear-Sys-Anisotropic-Symetric-Case}
\end{equation} and $\Bigl(\mc{R}_m\Bigr)_{m=1}^\nbs$ and $\Bigl(\mc{Q}_m\Bigr)_{m=1}^\nbs$ are solutions the following system 
\begin{align}\begin{aligned}
\mc{Q}_m=\ANPT{\meanm{\prb_r^*}{m}}{D_m}{(\mcntm{\prb_r^*}{m})}^{-1} &\biggl(\sum_{\stl{j\geq 1}{j\neq m}}^{\nbs}\Bigl[\Pi_{k}(\z_m,\z_j)\mcntm{\prb_r}{j}\mc{Q}_j- ik\nabla\Phi_{k}(\z_m,\z_j)\times\mcntm{\prv_r}{j}\mc{R}_j\Bigr]+ H^\n(\z_m)\biggl) \\
&+Er_m(H_\mc{A})+O(k^2\norm*{\meancnt{\prb_r}H_{\mc{A}}}_{\Lp{2}{D_m}}\rad^\frac{7}{2}+\norm*{\meancnt{\prv_r}E_{\mc{A}}}_{\Lp{2}{D_m}} \rad^\frac{5}{2}),\\
\mc{R}_m=\ANPT{\meanm{\prv_r^*}{m}}{D_m}{(\mcntm{\prv_r^*}{m})}^{-1} &\biggl(\sum_{\stl{j\geq 1}{j\neq m}}^{\nbs}\Bigl[\Pi_{k}(\z_m,\z_j)\mcntm{\prv_r}{j}\mc{R}_j+ ik\nabla\Phi_{k}(\z_m,\z_j)\times\mcntm{\prb_r}{j}\mc{Q}_j\Bigr]+ E^\n(\z_m)\biggl) \\
&+Er_m(E_\mc{A})+O(k^2\norm*{\meancnt{\prb_r}E_{\mc{A}}}_{\Lp{2}{D_m}}\rad^\frac{7}{2} +\norm*{\meancnt{\prv_r}H_{\mc{A}}}_{\Lp{2}{D_m}}\rad^\frac{5}{2}).
\end{aligned}\label{Approx-LinSys-Anisotropic-NonSym-Case}
\end{align}
 
\end{proposition}
\begin{proof}
For $x\in D_m,$ we have from \eqref{Lip-sch-Eq-Elec-Magn}
\begin{equation}
\begin{aligned} H&-(k^2+\nabla\div) \bdmc{S}^{k,{\cnt{\prb_r}}}_{D_m}(H)\\
&+ik\curl \bdmc{S}^{k,{\cnt{\prv_r}}}_{D_m}(E)
 =\sum_{\stl{j\geq 1}{j\neq m}}^{\nbs}\Bigl[(k^2+\nabla\div) \bdmc{S}_{D_j}^{k,{\cnt{\prb_r}}}(H)-ik\curl \bdmc{S}_{D_j}^{k,{\cnt{\prv_r}}}(E)(x)\Bigr]+ E^\n(x),
\end{aligned}\label{Lip-Schwinger-Linear-Sys-Form-Magne-field}\end{equation} and
\begin{equation}\begin{aligned} E&-(k^2+\nabla\div) \bdmc{S}^{k,{\cnt{\prv_r}}}_{D_m}(E)\\
&-ik\curl \bdmc{S}^{k,{\cnt{\prb_r}}}_{D_m}(H)
=\sum_{\stl{j\geq 1}{j\neq m}}^{\nbs}\Bigl[(k^2+\nabla\div) \bdmc{S}_{D_j}^{k,{\cnt{\prv_r}}}(E)+ik\curl \bdmc{S}_{D_j}^{k,{(Q^c(\prb_r))}}(H)(x)\Bigr]+ H^\n(x),
\end{aligned}\label{Lip-Schwinger-Linear-Sys-Form-Electri-field}  \end{equation}

\begin{itemize}
\item ({\emph {Derivation of} \eqref{Approximation-Linear-Sys-Anisotropic-Symetric-Case}})
 Multiplying the first member of   \eqref{Lip-Schwinger-Linear-Sys-Form-Magne-field}, by $U_m^{\prb_r^*}$ and  integrating over $D_m$, we derive 	
\begin{equation}\label{Lip-Schwinger-LinSys-First-Member-Approx}
\begin{aligned} \int_{ D_m}U_m^{\prb_r^*}&(x)\cdot \Bigl(H-(k^2+\nabla\div) \bdmc{S}^{k,\cnt{\prb_r}}_{D_m}(H)+ik\curl \bdmc{S}^{k,{\cnt{A}}}_{D_m}(E)\Bigr)(x)dx\\ 
=&\int_{ D_m}U_m^{\prb_r^*}\cdot \bigl[I-\nabla\div \bdmc{S}_{D_m}^{0,\cnt{\prb_r}}\bigr](H)(x)dx-k^2\int_{ D_m}U_m^{\prb_r^*}(x)\cdot \bdmc{S}^{k,{\cnt{\prb_r}}}_{D_m}(H)(x)dx\\
&-\int_{ D_m}U_m^{\prb_r^*}(x)\cdot\Bigl(\nabla\div \Bigl[\bdmc{S}^{k,{\cnt{A}}}_{D_m}-\bdmc{S}^{0,{\cnt{A}}}_{D_m}\Bigr](H)-ik\curl \bdmc{S}^{k,{Q^\cmp(\prv_r)}}_{D_m}(E)\Bigr)(x)dx,\\
=&\int_{ D_m}U_m^{\prb_r^*}\cdot \bigl[I-\nabla\div \bdmc{S}_{D_m}^{0,{\cnt{\prb_r}}}\bigr](H)(x)dx\\ &+O\Bigl(k^2\rad^{7/2}\norm*{\cnt{\prb_r}H}_{\Lp{2}{D_m}}+\rad^{5/2}\norm*{\cnt{\prv_r}E}_{\Lp{2}{D_m}}\Bigr).
\end{aligned} 
\end{equation} 

Indeed, due to the third identity \eqref{NablaNabal-Phi_k-Phi_i}, for $\alpha=0,$ it is obvious that
\begin{align*}
\Bigl|\nabla \div\bigl[\bdmc{S}^{k,{\cnt{A}}}_{D_m}-\bdmc{S}^{0,{\cnt{A}}}_{D_m}\bigr]&(H)(x)\Bigr|\\
\leq&\abso{\int_{D_m}\Bigl(\int_{0}^1\frac{e^{\bigl((ik)t\bigr) \abs*{(x-y)}} k^2t}{4\pi\abs*{x-y}} 
\Bigl[\mathbb{I}+\bigl(ikt-\frac{1}{\abs*{x-y}} \Bigr)\frac{\tensor{x-y}{2}}{\abs*{x-y}}\Bigr]dt\Bigr)(\cnt{\prb_r}H)(y)dy},\\
\leq&\int_{D_m} \Bigl[\frac{2k^2}{4\pi\abs*{x-y}}+\frac{k^3}{4\pi} \Bigr)\Bigr]~\abs{(\cnt{\prb_r}H)(y)}dy,\\
\leq&\Bigl[\frac{k^2}{2\pi}\Bigl(\lim_{r\rightarrow 0} \int_{B(y,\rad)\setminus{B(\z_m,r)}} \frac{1}{\abs*{x-y}^2} dy\Bigr)^\frac{1}{2}+\frac{k^3}{4\pi}\rad^\frac{3}{2}\Bigl]\norm{\cnt{\prb_r}H}_{\Lp{2}{D_m}},\\
\leq&	\Bigl[\frac{k^2}{2\pi}\pi\rad^\frac{1}{2} +\frac{k^3}{4\pi}\rad^\frac{3}{2}\Bigl]\norm{\cnt{\prb_r}H}_{\Lp{2}{D_m}},
\end{align*} hence,  
\begin{align*}
\norm{\nabla \div\bigl[\bdmc{S}^{k,{\cnt{\prb_r}}}_{D_m}-\bdmc{S}^{0,{\cnt{\prb_r}}}_{D_m}\bigr](H)}_{\Lp{2}{D_m}}\leq \frac{\pi}{3}	\Bigl[\frac{k^2}{\sqrt{2\pi}} +\frac{k^3}{4\pi}\rad\Bigl]\rad^2\norm{\cnt{\prb_r}H}_{\Lp{2}{D_m}}.
\end{align*} We also have, for similar reasons,
\begin{align*}
\norm{k^2\bigl[\bdmc{S}^{k,{\cnt{\prb_r}}}_{D_m}\bigr](H)}_{\Lp{2}{D_m}}\leq\frac{\pi}{3}\Bigl[\frac{k^2}{2\pi}\sqrt{2\pi}\rad^\frac{1}{2}\Bigl]\rad^\frac{3}{2}\norm{\cnt{\prb_r}H}_{\Lp{2}{D_m}}.
\end{align*}
Regarding the third term of \eqref{Lip-Schwinger-LinSys-First-Member-Approx}, we have  
\begin{align*}
\abso{\int_{ D_m}U_{m}^{\prb_r^*}(x)\cdot\Bigl(ik\curl \bdmc{S}^{k,{Q^\cmp(\prv_r)}}_{D_m}(H)\Bigr)(x)dx}
\leq&\norm*{U_{m}^{\prb_r^*}}_{\Lp{2}{D_m}}\norm{\cnt{\prb_r}H}_{\Lp{2}{D_m}} \int_{B(0,\rad)}\abs{\nabla\Phi_k(x)}dx,\\
 \leq&\bs{c}_{\prb_r}\rad^\frac{3}{2}\norm{\cnt{\prb_r}H}_{\Lp{2}{D_m}} \frac{1}{4\pi}\int_{B(0,\rad)}\Bigl(\frac{1}{\abs{x-y}^2}+\frac{\abs{k}}{\abs{x-y}}\Bigr)dx,\\
 \leq&\bs{c}_{\prb_r}(\frac{1}{2}+\frac{k}{4}\rad)\rad^\frac{5}{2} \norm{\cnt{\prb_r}H}_{\Lp{2}{D_m}}.
\end{align*} Hence, we get from \eqref{Lip-Schwinger-LinSys-First-Member-Approx} 
\begin{equation}\label{Lip-Schwin-LinSys-Fst-Mem-2}
\begin{aligned} \int_{D_m} U_{m}^{\prb_r^*}&\cdot  \Bigl[H-(k^2+\nabla\div) \bdmc{S}^{k,\cnt{\prb_r}}_{D_m}(H)-ik\curl\bdmc{S}^{k,{Q^\cmp(\prv_r)}}_{D_m}(E)\Bigr]dv\\
=&\int_{ D_m}\Bigl[I-\cnt{\prb_r}^*\nabla\div \bdmc{S}^0_{D_m}\Bigr](U_{m}^{\prb_r^*})\cdot H~dv+O(k^2\norm*{\cnt{\prb_r}H}_{\Lp{2}{D_m}}\rad^\frac{7}{2} +\norm*{\cnt{\prv_r}E}_{\Lp{2}{D_m}}\rad^\frac{5}{2}),\\
=&\int_{D_m} \cnt{\prb_r}^*H~dv+O(k^2\norm*{\cnt{\prb_r}H}_{\Lp{2}{D_m}}\rad^\frac{7}{2} +\norm*{\cnt{\prv_r}E}_{\Lp{2}{D_m}}\rad^\frac{5}{2}).
\end{aligned} 
\end{equation}
Multiplying the second member of \eqref{Lip-Schwinger-Linear-Sys-Form-Magne-field} by $U_{m}^{\prb_r^*}$ ind integrating over $D_m,$ gives     
\begin{align*}
\sum_{\stl{j\geq 1}{j\neq m}}^{\nbs}\int_{D_m}U_{m}^{\prb_r^*}(x)\cdot\Bigl[ \int_{D_j}\Pi_{k}(x,y)(\cnt{\prb_r}H)(y)dy&+\, ik\int_{D_j}\nabla\Phi_{k}(x,y)\times (\cnt{\prv_r}E)(y)dy\Bigr]dx\\ &+\int_{D_m}U_{m}^{\prb_r^*}\cdot H^\n dv,
\end{align*}
and a first order approximation, considering \eqref{Green-Func-First-Order-Appr},  gives 
\begin{equation}\begin{aligned}
&\int_{D_m}U_{m}^{\prb_r^*}dv\cdot\sum_{\stl{j\geq 1}{j\neq m}}^{\nbs}\Bigl[ \Pi_{k}(\z_m,\z_j)\int_{D_j}(\cnt{\prb_r}H)(y)dy+ ik\nabla\Phi_{k}(\z_m,\z_j)\times\int_{D_j} (\cnt{\prv_r}E)(y)dy\Bigr] \\
&+\int_{D_m}U_{m}^{\prb_r^*}dv~ H^\n(\z_j)+O(\abs{k}^2\rad^4)+2Er_m(H,E),
\end{aligned}\label{Lip-Schwin-LinSys-Scnd-Mem}
\end{equation} which gives, when the contrasts are symmetric, joined with \eqref{Lip-Schwin-LinSys-Fst-Mem-2}, the first approximation in \eqref{Approximation-Linear-Sys-Anisotropic-Symetric-Case}.
Similar calculation gives the second one starting from \eqref{Lip-Schwinger-Linear-Sys-Form-Electri-field}, and multiplying by $U_m^{\prv_r^*}$ as defined in \eqref{Def-Electric-Perm-Tensor}.
\item ({\emph {Derivation of} \eqref{Approx-LinSys-Anisotropic-NonSym-Case}}) 
When, in the other hand, both $\cnt{\prv_r}, \cnt{\prb_r}$ are in $\mathbb{W}^{1,\infty}(\cup_{m=1}^\nbs D_m),$ and not necessarily symmetric, we get, with a first order approximation, for the combined \cref{Lip-Schwin-LinSys-Fst-Mem-2,Lip-Schwin-LinSys-Scnd-Mem}  of the corresponding integral formulation of the Problem \eqref{Lip-sch-Eq-Elec-Magn-Averaged-Param}   
\begin{align*}
\int_{D_m} &\meancnt{\prb_r}^*H_{\mc{A}}~dv
+O(k^2\norm*{\meancnt{\prb_r}H_{\mc{A}}}_{\Lp{2}{D_m}}\rad^\frac{7}{2} +\norm*{\meancnt{\prv_r}E_{\mc{A}}}_{\Lp{2}{D_m}}\rad^\frac{5}{2})\\
=&\int_{D_m}U_{m}^{\mean{\prb_r^*}}dv\cdot\sum_{\stl{j\geq 1}{j\neq m}}^{\nbs}\Bigl[ \Pi_{k}(\z_m,\z_j)\int_{D_j}(\meancnt{\prb_r}H_{\mc{A}})(y)dy+ ik\nabla\Phi_{k}(\z_m,\z_j)\times\int_{D_j} (\meancnt{\prv_r}E_{\mc{A}})(y)dy\Bigr]\\&+\int_{D_m}U_{m}^{\mean{\prb_r^*}}\cdot H^\n dv+Er^{{(E_\mc{A},H_\mc{A})}}_m(H_\mc{A}),
\end{align*} which we rewrite as
\begin{align*}
\meancnt{\prb_r}^*\mc{R}_m&
+O(k^2\norm*{\meancnt{\prb_r}H_{\mc{A}}}_{\Lp{2}{D_m}}\rad^\frac{7}{2} +\norm*{\meancnt{\prv_r}E_{\mc{A}}}_{\Lp{2}{D_m}}\rad^\frac{5}{2})\\
=&\int_{D_m}U_{m}^{\mean{\prb_r^*}}dv\cdot\sum_{\stl{j\geq 1}{j\neq m}}^{\nbs}\Bigl[ \Pi_{k}(\z_m,\z_j)\meancnt{\prb_r}\mc{R}_m+ ik\nabla\Phi_{k}(\z_m,\z_j)\times \meancnt{\prv_r}\mc{Q}_m\Bigr]\\&+\int_{D_m}U_{m}^{\mean{\prb_r^*}}\cdot H^\n dv+Er^{{(E_\mc{A},H_\mc{A})}}_m(H_\mc{A}),
\end{align*} or more precisely
\begin{align*}
\meancnt{\prb_r^*}\mc{R}_m&
+O(k^2\norm*{\meancnt{\prb_r}H_{\mc{A}}}_{\Lp{2}{D_m}}\rad^\frac{7}{2} +\norm*{\meancnt{\prv_r}E_{\mc{A}}}_{\Lp{2}{D_m}}\rad^\frac{5}{2})\\
=&\meancnt{\prb_r^*}\ANPT{\meanm{\prb_r^*}{m}}{D_m}\meancnt{\prb_r^*}^{-1}\cdot\sum_{\stl{j\geq 1}{j\neq m}}^{\nbs}\Bigl[ \Pi_{k}(\z_m,\z_j)\meancnt{\prb_r}\mc{R}_m+ ik\nabla\Phi_{k}(\z_m,\z_j)\times \meancnt{\prv_r}\mc{Q}_m\Bigr]\\&+\int_{D_m}U_{m}^{\mean{\prb_r^*}}dv \cdot H^\n(\z_m)+Er^{{(E_\mc{A},H_\mc{A})}}_m(H_\mc{A})+O(\abs{k}^2\rad^4).
\end{align*} as for the symmetric case, the slights changes, for the second equation, are retrieved in the error of approximation.    

\item ({\emph {Derivation of} \eqref{Far-Field-Approximation-Anisotropic-Symmetric}}) Concerning the far field approximation \eqref{Far-Field-Approximation-Anisotropic-Symmetric}, a first order approximation gives 
\begin{equation}\begin{aligned}
E^\infty(\hat{x})=&\sum_{m=1}^{\nbs}\Bigl(\frac{k^2}{4\pi}\hat{x}\times\int_{ D_m}e^{-ik\hat{x}\cdot \z_m}\cnt{\prv_r}E(y)\times \hat{x}dy +\frac{ik}{4\pi}\hat{x}\times\int_{D_m}e^{-ik\hat{x}\cdot \z_m}\cnt{\prb_r}H(y)dy\Bigr)\\
&+\sum_{m=1}^{\nbs}\Biggl(\frac{k^2}{4\pi}\hat{x}\times\Bigl(\int_{ D_m}(e^{-ik\hat{x}\cdot y}-e^{-ik\hat{x}\cdot \z_m})\cnt{\prv_r}E(y)\times \hat{x}dy\\
&\hspace{3cm}+\frac{ik}{4\pi}\int_{D_m}(e^{-ik\hat{x}\cdot y}-e^{-ik\hat{x}\cdot \z_m})\cnt{\prb_r}H(y)dy\Bigr)\Biggr).
\end{aligned}\end{equation} As 
\begin{equation}
	\begin{aligned}
	\abs{(e^{-ik\hat{x}\cdot y}-e^{-ik\hat{x}\cdot \z_m})}\leq\abs{ik\hat{x}\int_{0}^1e^{-ik\hat{x}\cdot (ty+(1-t)z_m)}dt\cdot(y-\z_m)}\leq \abs{k}\rad,
	\end{aligned}
\end{equation} we get
\begin{equation}\begin{aligned}
E^\infty(\hat{x})=\sum_{m=1}^{\nbs}&\Bigl(\frac{k^2}{4\pi}\hat{x}\times\int_{ D_m}e^{-ik\hat{x}\cdot \z_m}\cnt{\prv_r}E(y)\times \hat{x}dy +\frac{ik}{4\pi}\hat{x}\times\int_{D_m}e^{-ik\hat{x}\cdot \z_m}\cnt{\prb_r}H(y)dy\Bigr)\\
&+O\Biggl(\sum_{m=1}^{\nbs}\frac{\abs{k}^3+\abs{k}^2}{4\pi}\rad\Bigl(\rad^\frac{3}{2}\norm{\cnt{\prv_r}E}_{\Lp{2}{D_m}}+\rad^\frac{3}{2}\norm{\cnt{\prb_r}H}_{\Lp{2}{D_m}}\Bigr)\Biggr)
\end{aligned}\end{equation} then follows 
\begin{equation}\label{Approx-FarField-SymCase-LastStep}
\begin{aligned}
E^\infty(\hat{x})=\sum_{m=1}^{\nbs}&\Bigl(\frac{k^2}{4\pi}\hat{x}\times\int_{ D_m}e^{-ik\hat{x}\cdot \z_m}\cnt{\prv_r}E(y)\times \hat{x}dy +\frac{ik}{4\pi}\hat{x}\times\int_{D_m}e^{-ik\hat{x}\cdot \z_m}\cnt{\prb_r}H(y)dy\Bigr)\\
&+O\Biggl(\frac{\abs{k}^3+\abs{k}^2}{4\pi}\bigl(\nbs\rad^3)^\frac{1}{2}\Bigl(\norm{\cnt{\prv_r}E}_{\Lp{2}{\cup D_m}}+\norm{\cnt{\prb_r}H}_{\Lp{2}{\cup D_m}}\Bigr)^\frac{1}{2}\rad\Biggr).
\end{aligned}\end{equation} The conclusion comes using the estimations \eqref{Estimation-For-Plan-Waves} and \eqref{Estimation-For-Plan-Waves1}.\\ 
The approximation  \eqref{Far-Field-Approximation-Anisotropic-Non-Symmetric} could be achieved by adding-subtracting to \eqref{Approx-FarField-SymCase-LastStep} above the following expression
\begin{equation*}
E^\infty(\hat{x})=\sum_{m=1}^{\nbs}\Bigl(\frac{k^2}{4\pi}\hat{x}\times\int_{ D_m}e^{-ik\hat{x}\cdot \z_m}\meancnt{\prv_r}E_\mc{A}(y)\times \hat{x}dy +\frac{ik}{4\pi}\hat{x}\times\int_{D_m}e^{-ik\hat{x}\cdot \z_m}\meancnt{\prb_r}H_\mc{A}(y)dy\Bigr).
\end{equation*} Precisely   
\begin{equation}
\begin{aligned}\label{Far-Field-Approx-NonSym-Case-1}
E^\infty(\hat{x})=\sum_{m=1}^{\nbs}&\Bigl(\frac{k^2}{4\pi}\hat{x}\times\int_{ D_m}e^{-ik\hat{x}\cdot \z_m}\meancnt{\prv_r}E_\mc{A}(y)\times \hat{x}dy +\frac{ik}{4\pi}\hat{x}\times\int_{D_m}e^{-ik\hat{x}\cdot \z_m}\meancnt{\prb_r}H_\mc{A}(y)dy\Bigr)\\
+\sum_{m=1}^{\nbs}&\Biggl(\frac{k^2}{4\pi}\hat{x}\times\int_{ D_m}e^{-ik\hat{x}\cdot \z_m}\Bigl(\cnt{\prv_r}E-\meancnt{\prv_r}E_\mc{A}\Bigr)(y)\times \hat{x}dy\\
 &\hspace{0.5cm}+\frac{ik}{4\pi}\hat{x}\times\int_{D_m}e^{-ik\hat{x}\cdot \z_m}\Bigl(\cnt{\prb_r}H-\meancnt{\prb_r}H_\mc{A}\Bigr)(y)dy\Biggr)
+O\Biggl(\frac{\abs{k}^3+\abs{k}^2}{\bs{c}_r^3}\rad\Biggr).
\end{aligned}\end{equation} 
Obviously  
\begin{align*}
\int_{ D_m}e^{-ik\hat{x}\cdot \z_m}\Bigl(\cnt{\prv_r}E-\meancnt{\prv_r}E_\mc{A}\Bigr)(y)dy=&
\int_{ D_m}e^{-ik\hat{x}\cdot \z_m}\Bigl[\cnt{\prv_r}\Bigl(E-E_\mc{A}\Bigr)\Bigl](y)dy\\&+\int_{ D_m}e^{-ik\hat{x}\cdot \z_m}\Bigl[\Bigl(\cnt{\prv_r}-\meancnt{\prv_r}\Bigr)E_\mc{A}\Bigr](y)dy,
\end{align*} achieving that
\begin{align*}
\sum_{m=1}^\nbs\abs{\int_{ D_m}e^{-ik\hat{x}\cdot \z_m}\Bigl(\cnt{\prv_r}E-\meancnt{\prv_r}E_\mc{A}\Bigr)(y)dy}\leq&
\sum_{m=1}^\nbs\rad^\frac{3}{2} \norm{\cnt{\prv_r}}_{\Lp{\infty}{D_m}}\norm*{E-E_\mc{A}}_\Lp{2}{D_m}\\&+\sum_{m=1}^\nbs\rad^\frac{3}{2} \rad\norm{\cnt{\prv_r}}_{\mathbb{W}^{\infty}{(D_m)}}\norm{E_\mc{A}}_\Lp{2}{D_m}.
\end{align*} Using H\"older's inequality, with both \eqref{Esimate-AveragedE-E}  and similar \eqref{Estimation-For-Plan-Waves} and \eqref{Estimation-For-Plan-Waves1} for $E_{\mc{A}},$ we get
\begin{align}
\sum_{m=1}^\nbs\abs{\int_{ D_m}e^{-ik\hat{x}\cdot \z_m}\Bigl(\cnt{\prv_r}E-\meancnt{\prv_r}E_\mc{A}\Bigr)(y)dy}
=O\Biggl(\frac{\bs{c}_\infty^2\bs{c}_{2,k}+\bs{c}_\infty}{\bs{c}_r^3}\rad\Biggr).
\end{align} With similar calculation we get
\begin{align}
\sum_{m=1}^\nbs\abs{\int_{ D_m}e^{-ik\hat{x}\cdot \z_m}\Bigl(\cnt{\prb_r}H-\meancnt{\prb_r}H_\mc{A}\Bigr)(y)dy}
=O\Biggl(\frac{\bs{c}_\infty^2\bs{c}_{2,k}+\bs{c}_\infty}{\bs{c}_r^3}\rad\Biggr).
\end{align}  Injecting both of the above estimations in \eqref{Far-Field-Approx-NonSym-Case-1} gives the results.
\end{itemize}
\end{proof}

\begin{proposition} \label{Inversion-Lin-Sys-Anisotropic-Case}Both the linear systems 
\begin{align}\label{Linear-System-Anisotropic-SymCase}
\begin{aligned}
\ANPT{\prb_r^*}{D_m}^{-1}\wh{\mc{Q}}^{\prb_r}_m=&\sum_{\stackrel{j\geq 1}{j\neq m}}^{\nbs} \Bigl[ \Pi_{k}(\z_m,\z_j) \wh{\mc{Q}}^{\prb_r}_j-ik \nabla\Phi_{k}(\z_m,\z_j)\times\wh{\mc{R}}^{\prv_r}_j\Bigr]+ \wh{H}^\n(\z_m), \\
\ANPT{\prv_r^*}{D_m}^{-1}\wh{\mc{R}}^{\prv_r}_m=&\sum_{\stackrel{j\geq 1}{j\neq m}}^{\nbs}\Bigl[ \Pi_{k}(\z_m,\z_j) \wh{\mc{R}}^{\prv_r}_j+ ik\nabla\Phi_{k}(\z_m,\z_j)\times\wh{\mc{Q}}^{\prb_r}_j\Bigr]+\wh{E}^\n(\z_m),
\end{aligned}
\end{align}	and 
\begin{align}
\begin{aligned}\label{Linear-System-Anisotropic-GenCase}
\mcntm{\prb_r^*}{m}\ANPT{\meanm{\prb_r^*}{m}}{D_m}^{-1}\mc{Q}_m
=&\biggl(\sum_{\stl{j\geq 1}{j\neq m}}^{\nbs}\Bigl[\Pi_{k}(\z_m,\z_j)\mcntm{\prb_r}{j}\wh{\mc{Q}}_j- ik\nabla\Phi_{k}(\z_m,\z_j)\times\mcntm{\prv_r}{j}\wh{\mc{R}}_j\Bigr]+ \wh{H}^\n(\z_m)\biggl), \\
\mcntm{\prv_r^*}{m}\ANPT{\meanm{\prv_r^*}{m}}{D_m}^{-1}\mc{R}_m =&\biggl(\sum_{\stl{j\geq 1}{j\neq m}}^{\nbs}\Bigl[\Pi_{k}(\z_m,\z_j)\mcntm{\prv_r}{j}\wh{\mc{R}}_j+ ik\nabla\Phi_{k}(\z_m,\z_j)\times\mcntm{\prb_r}{j}\wh{\mc{Q}}_j\Bigr]+ \wh{E}^\n(\z_m)\biggl), \\
\end{aligned}
\end{align}	
are invertible, provided \begin{equation}\frac{\dist}{\rad}=\bs{c}_r \geq{3\abs{k}}{\mu_{(\prv_r,\prb_r)}^+}.\end{equation} 
For $\abs{k}>1$ we have the following estimates
\begin{equation}\label{Estim-Sol-Linear-Sys}
\Bigl({\sum_{m=1}^{\nbs}\abs*{\wh{\mc{Q}}^{\prb_r}_m}^2}\Bigr)^\frac{1}{2}
\leq\frac{9\mu_{(\prv_r,\prb_r)}^+\rad^{3}}{8}\Biggl(\Bigl(\sum_{m=1}^\nbs \abs*{H^\n(\z_m)}^2\Bigr)^\frac{1}{2}+\frac{1}{3}\Bigl(\sum_{m=1}^\nbs \abs*{E^\n(\z_m)}^2\Bigr)^\frac{1}{2}\Biggr), \end{equation} and 
\begin{equation}\label{Estim-Sol-Linear-Sys1}
\Bigl({\sum_{m=1}^{\nbs}\abs*{\wh{\mc{R}}^{\prv_r}_m}^2}\Bigr)^\frac{1}{2}
\leq\frac{9\mu_{(\prv_r,\prb_r)}^+\rad^{3}}{8}\Biggl(\frac{1}{3}\Bigl(\sum_{m=1}^\nbs \abs*{H^\n(\z_m)}^2\Bigr)^\frac{1}{2}+\Bigl(\sum_{m=1}^\nbs \abs*{E^\n(\z_m)}^2\Bigr)^\frac{1}{2}\Biggr). 
\end{equation}
Furthermore $(\mcntm{\prv_r}{m}\wh{\mc{R}}_m,\mcntm{\prb_r}{m}\wh{\mc{Q}}_m)_{m=1}^\nbs$ satisfy similar estimates. 
\end{proposition}
\begin{proof} We deal with the case of real valued coefficients (i.e~assumption \eqref{Ani-PolTens-Coercivity-RealCase} is fulfilled). 
When either one of them or both are complex valued, we procceed in the same way taking the real parts in the the coming derivations.  
\begin{itemize} 
	\item (\emph{First linear system})   
 We start writing the system \eqref{Linear-System-Anisotropic-SymCase} as follows
\begin{align}
\begin{aligned}
\bsmc{Q}_m-&\sum_{\stackrel{j\geq 1}{j\neq m}}^{\nbs} \Bigl[ \Pi_{k}(\z_m,\z_j) \ANPT{\prb_r^*}{D_j}\bsmc{Q}_j- ik\nabla\Phi_{k}(\z_m,\z_j)\times\ANPT{\prv_r^*}{D_m}\bsmc{R}_j\Bigr]= H^\n(\z_m),\\ 
\bsmc{R}_m-&\sum_{\stackrel{j\geq 1}{j\neq m}}^{\nbs}\Bigl[ \Pi_{k}(\z_m,\z_j) \ANPT{\prv_r}{D_m}\bsmc{R}_j+ ik\nabla\Phi_{k}(\z_m,\z_j)\times\ANPT{\prb_r^*}{D_m} \bsmc{Q}_j\Bigr]=E^\n(\z_m),
\end{aligned}
\end{align} where $\bsmc{Q}_j:=\ANPT{\prb_r^*}{D_j}^{-1}\wh{\mc{Q}}^{\prb_r}_j$ and $\bsmc{R}_j:=\ANPT{\prb_r^*}{D_j}^{-1}\mc{R}^{\prb_r}_j$ for $j\in\{1,...,\nbs\}.$ Observing that, due to mean value property for harmonic function, 
\begin{equation}\begin{aligned}
{\abs*{B(0,\dist/4)}^2}&\sum_{m=1}^\nbs\sum_{\stackrel{j\geq 1}{j\neq m}}^{\nbs} \Bigl[ \Pi_{0}(\z_m,\z_j) \ANPT{\prb_r^*}{D_j}\bsmc{Q}_j\cdot \ANPT{\prb_r^*}{D_m}\bsmc{Q}_m\Bigr]\\
&=\sum_{m=1}^\nbs\sum_{\stackrel{j\geq 1}{j\neq m}}^{\nbs}\int_{B(\z_m,{\dist/4})}\int_{B(\z_j,{\dist/p})}\Bigl[ \Pi_{0}(x,y) \ANPT{\prb_r^*}{D_j}\bsmc{Q}_j\cdot \ANPT{\prb_r^*}{D_m}\bsmc{Q}_m\Bigr]dydx,\\
&=\sum_{m=1}^\nbs\Biggl(\int_{B(\z_m,{\dist/4})}\int_{B(\z,{\rho})}\Bigl[ \Pi_{0}(x,y) \sum_{j\geq 1}^{\nbs}\ind{D_j}(x)\ANPT{\prb_r^*}{D_j}\bsmc{Q}_j\cdot \ANPT{\prb_r^*}{D_m}\bsmc{Q}_m\Bigr]dydx\\
&-\int_{B(\z_m,{\dist/4})}\int_{B(\z_m,{\dist/4})}\Bigl[ \Pi_{0}(x,y) \ind{D_j}(x)\ANPT{\prb_r^*}{D_m}\bsmc{Q}_m\cdot \ANPT{\prb_r^*}{D_m}\bsmc{Q}_m\Bigr]dydx\Biggr),
\end{aligned}\end{equation} for some $\z$ and $\rho$ such that $\z\in\cup_{m=1}^\nbs B(\z_m,\dist/4)\subset B(\z,\rho).$ 
Also for $Q(x)=\sum_{j\geq 1}^{\nbs}\ind{D_j}(x) \ANPT{\prb_r^*}{D_j}\bsmc{Q}_j$ we get 
\begin{equation*}\begin{aligned}
-{\abs*{B(0,\dist/4)}^2}&\sum_{m=1}^\nbs\sum_{\stackrel{j\geq 1}{j\neq m}}^{\nbs} \Bigl[ \Pi_{0}(\z_m,\z_j) \ANPT{\prb_r^*}{D_j}\bsmc{Q}_j\cdot \ANPT{\prb_r^*}{D_m}\bsmc{Q}_m\Bigr]\\
&=-\int_{B(\z,\rho)}\MSVP{B(\z,\rho)}{0,I}(Q)(x)\cdot Q(x)dx\\
&+\sum_{m=1}^\nbs\int_{B(\z_m,\dist/4)}\int_{B(\z_m,\dist/4)}\Bigl[ \Pi_{0}(x,y) \ANPT{\prb_r^*}{D_m}\bsmc{Q}_m\cdot \ANPT{\prb_r^*}{D_m}\bsmc{Q}_m\Bigr]dydx,
\end{aligned}\end{equation*} in such a way that
\begin{equation}\label{Coercivity-of-Discret-Diadic-Green}\begin{aligned}
-&\sum_{\stl{m,j=1}{j\neq m}}^\nbs \Bigl[\Pi_{0}(\z_m,\z_j)\ANPT{\prb_r^*}{D_j}\bsmc{Q}_j\cdot \ANPT{\prb_r^*}{D_m}\bsmc{Q}_m\Bigr]\\
&\geq\frac{\sum_{m=1}^\nbs}{\abs*{B(0,\dist/4)}^2}\int_{B(\z_m,\dist/4)}\int_{B(\z_m,\dist/4)}\Bigl[ \Pi_{0}(x,y) \ANPT{\prb_r^*}{D_m}\bsmc{Q}_m\cdot \ANPT{\prb_r^*}{D_m}\bsmc{Q}_m\Bigr]dydx.
\end{aligned}\end{equation}	
With this in mind we have, from the first equation of linear system,
\begin{equation*}
\begin{aligned}
&\sum_{m=1}^\nbs\bsmc{Q}_m\cdot\ANPT{\prb_r^*}{D_m}\bsmc{Q}_m-\sum_{\stl{m,j=1}{j\neq m}}^\nbs  \Pi_{k}(\z_m,\z_j) \ANPT{\prb_r^*}{D_j}\bsmc{Q}_j\cdot\ANPT{\prb_r^*}{D_m}\bsmc{Q}_m\\
&+\sum_{\stl{m,j=1}{j\neq m}}^\nbs ik\nabla\Phi_{k}(\z_m,\z_j)\times\ANPT{\prv_r^*}{D_m}\bsmc{R}_j\cdot\ANPT{\prb_r^*}{D_m}\bsmc{Q}_m= \sum_{m=1}^\nbs H^\n(\z_m)\cdot\ANPT{\prb_r^*}{D_m}\bsmc{Q}_m. 
\end{aligned}
\end{equation*} Using \eqref{Coercivity-of-Discret-Diadic-Green} we get 
\begin{align}\label{Fisrt-Ine-Esimation-Inver-First-LS}
\begin{aligned}
&\sum_{m=1}^\nbs\bsmc{Q}_m\cdot\ANPT{\prb_r^*}{D_m}\bsmc{Q}_m-\sum_{\stl{m,j=1}{j\neq m}}^\nbs  (\Pi_{k}-\Pi_{0})(\z_m,\z_j) \ANPT{\prb_r^*}{D_j}\bsmc{Q}_j\ANPT{\prb_r^*}{D_m}\bsmc{Q}_m\\
&+\sum_{\stl{m,j=1}{j\neq m}}^\nbs ik\nabla\Phi_{k}(\z_m,\z_j)\times\ANPT{\prv_r^*}{D_j}\bsmc{R}_j\cdot\ANPT{\prb_r^*}{D_m}\bsmc{Q}_m\\
&+{\frac{\sum_{m=1}^\nbs}{\abs*{\frac{\dist}{4}}^2}\int_{B(\z_m,\frac{\dist}{4})}\int_{B(\z_m,\frac{\dist}{4})}\Bigl[ \Pi_{0}(x,y) \ANPT{\prb_r^*}{D_m}\bsmc{Q}_m\cdot \ANPT{\prb_r^*}{D_m}\bsmc{Q}_m\Bigr]dydx} \leq \sum_{m=1}^\nbs H^\n(\z_m)\cdot\ANPT{\prb_r^*}{D_m}\bsmc{Q}_m. 
\end{aligned}
\end{align} As we have, see (\cite{gilbarg2015elliptic}, Theorem 9.9), 
\begin{equation}
\norm{\int_{\R^3}\Pi_0(\cdot,y)\ind{B(\z_m,1)}(y)dy}_{\Lp{2}{\R^3}}=\norm{\ind{B(\z_m,1)}}_{\Lp{2}{\R^3}}
\end{equation} we get, with $S_\nbs(\Pi_0)$ standing for the third term of the left-hand side of \eqref{Fisrt-Ine-Esimation-Inver-First-LS},  
\begin{align*}
S_\nbs(\Pi_0)
&\leq 
\frac{\sum_{m=1}^\nbs}{\abs*{B(0,1)}^2}\int_{B(\z_m,1)}\int_{B(\z_m,1)}\Bigl[ \Pi_{0}(x,y) \ANPT{\prb_r^*}{D_m}\bsmc{Q}_m\cdot \ANPT{\prb_r^*}{D_m}\bsmc{Q}_m\Bigr]dydx,\\
&\leq \frac{\sum_{m=1}^\nbs}{\abs*{B(0,1)}^2}\norm{\int_{B(\z_m,1)} \Pi_{0}(x,y) \ANPT{\prb_r^*}{D_m}\bsmc{Q}_mdy}_{\Lp{2}{B(\z_m,1)}} \norm{\ANPT{\prb_r^*}{D_m}\bsmc{Q}_m}_{\Lp{2}{B(\z_m,1)}}\\
&\leq \sum_{m=1}^\nbs \abs{\ANPT{\prb_r^*}{D_m}\bsmc{Q}_m}^2,
\end{align*} 
which replaced in \eqref{Fisrt-Ine-Esimation-Inver-First-LS}, together with \eqref{NablaNabal-Phi_k-Phi_i}, for $\alpha=0$, gives
\begin{align}
\begin{aligned}
&\sum_{m=1}^\nbs\bsmc{Q}_m\cdot\ANPT{\prb_r^*}{D_m}\bsmc{Q}_m-\sum_{\stl{m,j=1}{j\neq m}}^\nbs \Biggl({\abs{k}^2} \Bigl(\frac{2}{\dist_{mj}}+\abs{k}\Bigr)+\frac{\abs{k}^2}{\dist_{mj}} \Biggr) \abs{\ANPT{\prb_r^*}{D_j}\bsmc{Q}_j}\abs{\ANPT{\prb_r^*}{D_m}\bsmc{Q}_m}\\
&-\sum_{\stl{m,j=1}{j\neq m}}^\nbs \Biggl(\frac{\abs{k}}{\dist_{mj}}\Bigl(\frac{1}{\dist_{mj}}+\abs{k}\Bigr)\Biggr)\abs{\ANPT{\prv_r^*}{D_j}\bsmc{R}_j}\abs{\ANPT{\prb_r^*}{D_m}\bsmc{Q}_m}\\
&-\sum_{m=1}^\nbs \abs{\ANPT{\prb_r^*}{D_m}\bsmc{Q}_m}^2 \leq \sum_{m=1}^\nbs \abs{H^\n(\z_m)}\abs{\ANPT{\prb_r^*}{D_m}\bsmc{Q}_m}. 
\end{aligned}\end{align} 
From the above inequality, with calculations similar to \eqref{Estim-star1}, we get  
\begin{align*}
\begin{aligned}
&\sum_{m=1}^\nbs\bsmc{Q}_m\cdot\ANPT{\prb_r^*}{D_m}\bsmc{Q}_m-  \Bigl(\frac{3\abs{k}^2}{(\rad/2+\dist)^\frac{1}{2}\dist}+\frac{\abs{k}(\abs{k}^2+1)}{(\rad/2+\dist)^\frac{3}{2}}\Bigr)\frac{1}{(\rad/2+\dist)^\frac{3}{2}} \sum_{m=1}^{\nbs}\abs{\ANPT{\prb_r^*}{D_m} \bsmc{Q}_m}^2\\
&-\Biggl(\frac{\abs{k}}{\dist}\Bigl(\frac{1}{(\rad/2+\dist)\dist}+ \frac{\abs{k}}{(\rad/2+\dist)^2}\Bigr)\Biggr)\Bigl(\sum_{m=1}^\nbs\abs{\ANPT{\prv_r^*}{D_m}\bsmc{R}_m}^2\Bigr)^\frac{1}{2}\Bigl(\sum_{m=1}^\nbs\abs{\ANPT{\prb_r^*}{D_m}\bsmc{Q}_m}^2\Bigr)^\frac{1}{2}\\
&-\sum_{m=1}^\nbs \abs{\ANPT{\prb_r^*}{D_m}\bsmc{Q}_m}^2 \leq \sum_{m=1}^\nbs \abs{H^\n(\z_m)}\abs{\ANPT{\prb_r^*}{D_m}\bsmc{Q}_m},
\end{aligned}
\end{align*}  then, with the original notations, we get 
\begin{align}\label{Est-TotChrg-Prof-Last-Form}
\begin{aligned}
&\sum_{m=1}^\nbs \ANPT{\prb_r^*}{D_m}^{-1}\wh{\mc{Q}}^{\prb_r}_m \cdot{\wh{\mc{Q}}^{\prb_r}_m}-  \Bigl(\frac{3\sqrt{2}\abs*{k}^2}{(1+2\bs{c}_r)^\frac{1}{2}\bs{c}_r}+ \frac{\sqrt{2}^3\abs*{k}(\abs*{k}^2+1)}{(1+2\bs{c}_r)^\frac{3}{2}}\Bigr)\frac{\sqrt{2}^3}{(1+2\bs{c}_r)^\frac{3}{2}} \frac{\sum_{m=1}^{\nbs}\abs*{\wh{\mc{Q}}^{\prb_r}_m}^2}{\rad^{3}}\\
&-\Biggl(\frac{\abs*{k}}{\bs{c}_r}\Bigl(\frac{2}{(1+2\bs{c}_r)}+ \frac{2\abs*{k}}{(1+\bs{c}_r)^2}\Bigr)\Biggr)\frac{\Bigl(\sum_{m=1}^\nbs\abs*{\wh{\mc{R}}^{\prv_r}_m}^2\Bigr)^\frac{1}{2}\Bigl(\sum_{m=1}^\nbs\abs*{\wh{\mc{Q}}^{\prb_r}_m}^2\Bigr)^\frac{1}{2}}{\rad^3}\\
&-\sum_{m=1}^\nbs \abs*{\wh{\mc{Q}}^{\prb_r}_m}^2
\leq \sum_{m=1}^\nbs \abs*{H^\n(\z_m)}\abs*{\mc{Q}^{\prb_r}_{m}}. 
\end{aligned}
\end{align} 
Due to \eqref{Ani-PolTens-Coercivity-RealCase}, considering that,  both 
\begin{equation}\label{Anisotropic-Pol-Tensor-Scaling}
\frac{1}{\rad^3\mu_{(\prv_r,\prb_r)}^+} \abs*{\wh{\mc{Q}}^{\prb_r}_m}^2\leq{\ANPT{\prb_r^*}{D_m}^{-1}\wh{\mc{Q}}^{\prb_r}_m\cdot\wh{\mc{Q}}^{\prb_r}_m}\leq\frac{1}{\rad^3\mu_{(\prv_r,\prb_r)}^-}\abs*{\wh{\mc{Q}}^{\prb_r}_m}^2\end{equation} and \begin{equation}\frac{1}{\rad^3\mu_{(\prv_r,\prb_r)}^+} \abs{\wh{\mc{R}}^{\prv_r}_m}^2\leq{\ANPT{\prv_r^*}{D_m}^{-1}\wh{\mc{R}}^{\prv_r}_m\cdot\wh{\mc{R}}^{\prv_r}_m}\leq \frac{1}{\rad^3\mu_{(\prv_r,\prb_r)}^-}\abs*{\wh{\mc{Q}}^{\prb_r}_m}^2\end{equation}  are satisfied, both \eqref{Est-TotChrg-Prof-Last-Form} and the obtained result from repeating the same calculation for the second equation, gives, with  $\bs{c}_r\geq {3\abs{k}}{\mu_{(\prv_r,\prb_r)}^+}$ and  $\abs{k}>1,$
\begin{equation*}\begin{aligned}
(\frac{7}{9}-\rad^3) {\sum_{m=1}^{\nbs}\abs*{\wh{\mc{Q}}^{\prb_r}_m}^2}
-\frac{2}{9}  {\Bigl(\sum_{m=1}^\nbs\abs*{\wh{\mc{R}}^{\prv_r}_m}^2\Bigr)^\frac{1}{2}\Bigl(\sum_{m=1}^\nbs\abs*{\wh{\mc{Q}}^{\prb_r}_m}^2\Bigr)^\frac{1}{2}}
\leq \mu_{(\prv_r,\prb_r)}^+\rad^3\sum_{m=1}^\nbs \abs*{H^\n(\z_m)}\abs*{\wh{\mc{Q}}^{\prb_r}_m}, 
\end{aligned}\end{equation*} and
\begin{equation*}
\begin{aligned}
(\frac{7}{9}-\rad^3) {\sum_{m=1}^{\nbs}\abs*{\wh{\mc{R}}^{\prb_r}_m}^2}
-\frac{2}{9}  {\Bigl(\sum_{m=1}^\nbs\abs*{\mc{Q}^{\prv_r}_m}^2\Bigr)^\frac{1}{2}\Bigl(\sum_{m=1}^\nbs\abs*{\wh{\mc{R}}^{\prb_r}_m}^2\Bigr)^\frac{1}{2}}
\leq \mu_{(\prv_r,\prb_r)}^+\rad^3\sum_{m=1}^\nbs \abs*{E^\n(\z_m)}\abs*{\wh{\mc{Q}}^{\prb_r}_m}. 
\end{aligned} \end{equation*} 
which, for $\rad\leq\frac{1}{3}$, gives 
\begin{equation*}
{\sum_{m=1}^{\nbs}\abs*{\wh{\mc{Q}}^{\prb_r}_m}^2}
-\frac{1}{3}  {\Bigl(\sum_{m=1}^\nbs\abs*{\wh{\mc{R}}^{\prv_r}_m}^2\Bigr)^\frac{1}{2}\Bigl(\sum_{m=1}^\nbs\abs*{\wh{\mc{Q}}^{\prb_r}_m}^2\Bigr)^\frac{1}{2}}
\leq \mu_{(\prv_r,\prb_r)}^+\frac{\rad^3}{3}\Bigl(\sum_{m=1}^\nbs \abs*{H^\n(\z_m)}^2\Bigr)^\frac{1}{2}\Bigl(\sum_{m=1}^\nbs\abs*{\wh{\mc{Q}}^{\prb_r}_m}^2\Bigr)^\frac{1}{2}, 
\end{equation*} and
\begin{equation*}
{\sum_{m=1}^{\nbs}\abs*{\wh{\mc{R}}^{\prb_r}_m}^2}
-\frac{1}{3}  {\Bigl(\sum_{m=1}^\nbs\abs*{\mc{Q}^{\prv_r}_m}^2\Bigr)^\frac{1}{2}\Bigl(\sum_{m=1}^\nbs\abs*{\wh{\mc{R}}^{\prb_r}_m}^2\Bigr)^\frac{1}{2}}
\leq \mu_{(\prv_r,\prb_r)}^+\frac{\rad^3}{3}\Bigl(\sum_{m=1}^\nbs \abs*{E^\n(\z_m)}^2\Bigr)^\frac{1}{2}\Bigl(\sum_{m=1}^\nbs\abs*{\wh{\mc{Q}}^{\prb_r}_m}^2\Bigr)^\frac{1}{2}. 
\end{equation*} 
 Then follows the conclusion, namely
\begin{equation*}
\Bigl({\sum_{m=1}^{\nbs}\abs*{\wh{\mc{Q}}^{\prb_r}_m}^2}\Bigr)^\frac{1}{2}
\leq \mu_{(\prv_r,\prb_r)}^+\rad^{3}\Bigl(\sum_{m=1}^\nbs \abs*{H^\n(\z_m)}^2\Bigr)^\frac{1}{2}+\frac{1}{3}  {\Bigl(\sum_{m=1}^\nbs\abs*{\wh{\mc{R}}^{\prv_r}_m}^2\Bigr)^\frac{1}{2}}, 
\end{equation*} and
\begin{equation*}
\Bigl({\sum_{m=1}^{\nbs}\abs*{\wh{\mc{R}}^{\prb_r}_m}^2}\Bigr)^\frac{1}{2}
\leq\mu_{(\prv_r,\prb_r)}^+\rad^{3}\Bigl(\sum_{m=1}^\nbs \abs*{E^\n(\z_m)}^2\Bigr)^\frac{1}{2}+\frac{1}{3}  {\Bigl(\sum_{m=1}^\nbs\abs*{\mc{Q}^{\prv_r}_m}^2\Bigr)^\frac{1}{2}}. 
\end{equation*} 

\item (\emph{Non symmetric case}) We could deduce the estimate from the above calculation writing down the linear system as follows, first
we set, for $m\in\{1,...,\nbs\}$ \begin{align*}
{\ANPTT{\meanm{\prb_r^*}{m}}{D_m}^{-1}}:=& \mcntm{\prb_r^*}{m}\ANPT{\meanm{\prb_r^*}{m}}{D_m}^{-1}{(\mcntm{\prb_r}{m})}^{-1},\\{\ANPTT{\meanm{\prv_r^*}{m}}{D_m}^{-1}}:=& \mcntm{\prv_r^*}{m}\ANPT{\meanm{\prv_r^*}{m}}{D_m}^{-1}{(\mcntm{\prv_r}{m})}^{-1},
\end{align*} and, this time 
\begin{align}
\bs{\mc{R}}_m:={\ANPTT{\meanm{\prv_r^*}{m}}{D_m}^{-1}}\mcntm{\prv_r}{m}\wh{\mc{R}}_m&,\\
\bs{\mc{Q}}_m:={\ANPTT{\meanm{\prb_r^*}{m}}{D_m}^{-1}}\mcntm{\prb_r}{m}\wh{\mc{Q}}_m&
\end{align} the system \eqref{Linear-System-Anisotropic-GenCase} can, hence, be written as 
\begin{equation}
\begin{aligned}
\bsmc{Q}_m-&\sum_{\stackrel{j\geq 1}{j\neq m}}^{\nbs} \Bigl[ \Pi_{k}(\z_m,\z_j) \ANPTT{\mean{\prb^*_r}}{D_j}\bsmc{Q}_j- ik\nabla\Phi_{k}(\z_m,\z_j)\times\ANPTT{\mean{\prv^*_r}}{D_m}\bsmc{R}_j\Bigr]= H^\n(\z_m),\\ 
\bsmc{R}_m-&\sum_{\stackrel{j\geq 1}{j\neq m}}^{\nbs}\Bigl[ \Pi_{k}(\z_m,\z_j) \ANPTT{\mean{\prv^*_r}}{D_m}\bsmc{R}_j+ ik\nabla\Phi_{k}(\z_m,\z_j)\times\ANPTT{\mean{\prb^*_r}}{D_m} \bsmc{Q}_j\Bigr]=E^\n(\z_m),
\end{aligned}
\end{equation} hence, following the same line as in the proof for the first linear system, we get 
\begin{equation}\begin{aligned}
\Bigl({\sum_{m=1}^{\nbs} \abs*{\meancnt{\prb_r}\wh{\mc{Q}}_m}^2}\Bigr)^{(1/2)}
\leq\frac{9\mu_{(\prv_r,\prb_r)}^+ \rad^{3}}{8}\Biggl(\Bigl(\sum_{m=1}^\nbs \abs*{H^\n(\z_m)}^2\Bigr)^\frac{1}{2}+\frac{1}{3}\Bigl(\sum_{m=1}^\nbs \abs*{E^\n(\z_m)}^2\Bigr)^\frac{1}{2}\Biggr), \\
\Bigl({\sum_{m=1}^{\nbs}\abs*{\meancnt{\prv_r}\wh{\mc{R}}_m}^2}\Bigr)^{(1/2)}
\leq\frac{9\mu_{(\prv_r,\prb_r)}^+\rad^{3}}{8}\Biggl(\frac{1}{3}\Bigl(\sum_{m=1}^\nbs \abs*{H^\n(\z_m)}^2\Bigr)^\frac{1}{2}+\Bigl(\sum_{m=1}^\nbs \abs*{E^\n(\z_m)}^2\Bigr)^\frac{1}{2}\Biggr),
	\end{aligned}
\end{equation} with unchanged condition on  \begin{equation}
\bs{c}_r={3\abs{k}}{\mu_{(\prv_r,\prb_r)}^+}, \abs{k}>1.
\end{equation} 
For the last statement, it suffices to observe, that $ \cnt{A}$ and its transpose $ \cnt{A}^*$ share the same eigenvalues, and hence leave those of $\ANPT{A^*}{D_m}^{-1}$ identical to those of   $\ANPTT{{A^*}}{D_m}^{-1}.$
\end{itemize}
\end{proof}
\begin{proof}(Of \Cref{Theorem-Anisotropic-Case})\\
\begin{itemize}
\item We start with \eqref{Farfield-Anis-Approximation-FullError},
noticing that, for $(\wh{\mc{Q}}^{\prb_r}_m,\wh{\mc{R}}^{\prb_r}_m)_{m=1}^\nbs$ solution of \eqref{Linear-System-Anisotropic-SymCase} and $({\mc{Q}}^{\prb_r}_m,{\mc{R}}^{\prb_r}_m)_{m=1}^\nbs$ solution of \eqref{Approximation-Linear-Sys-Anisotropic-Symetric-Case} then 	
$(\wh{\mc{Q}}^{\prb_r}_m-\mc{Q}^{\prb_r}_m,\wh{\mc{R}}^{\prb_r}_m-\mc{R}^{\prb_r}_m)_{m=1}^\nbs$ solves \eqref{Linear-System-Anisotropic-SymCase} with 
\begin{equation}
\begin{aligned}
\wh{H}^\n(\z_m)=\ANPT{\prb_r^*}{D_m}^{-1}\Biggl(Er_m(H,E)+O\Bigl(k^2\norm*{\cnt{\prb_r}H}_{\Lp{2}{D_m}}\rad^\frac{7}{2} +\norm*{\cnt{\prv_r}E}_{\Lp{2}{D_m}}\rad^\frac{5}{2}\Bigr)\Biggr)
\end{aligned}
\end{equation} and
\begin{equation}
\begin{aligned}
\wh{E}^\n(\z_m)=\ANPT{\prv_r^*}{D_m}^{-1}\Biggl(Er_m(E,H)+O\Bigl(k^2\norm*{\cnt{\prv_r}E}_{\Lp{2}{D_m}}\rad^\frac{7}{2} +\norm*{\cnt{\prb_r}H}_{\Lp{2}{D_m}}\rad^\frac{5}{2}\Bigr)\Biggr).
\end{aligned}
\end{equation} 
Thus, with the estimates \eqref{Estim-Sol-Linear-Sys},  \eqref{Estim-Sol-Linear-Sys1} and \eqref{Anisotropic-Pol-Tensor-Scaling}, we have, with the aid of \eqref{Error-Estim-Error_m} for the third inequality
\begin{equation*}
\begin{aligned}
\Bigl(&\sum_{m=1}^{\nbs}  \abs*{\wh{\mc{Q}}^{\prb_r}_m-\mc{Q}^{\prb_r}_m}^2\Bigr)^\frac{1}{2}\\
&\leq\frac{9\mu_{(\prv_r,\prb_r)}^+}{8\mu_{(\prv_r,\prb_r)}^-}\Biggl[\biggl\{\sum_{m=1}^\nbs \Bigl(Er_m(H,E)+O\bigl(k^2\norm*{\cnt{\prb_r}H}_{\Lp{2}{D_m}}\rad^\frac{7}{2} +\norm*{\cnt{\prv_r}E}_{\Lp{2}{D_m}} \rad^\frac{5}{2}\bigr)\Bigr)^2\biggr\}^\frac{1}{2}\\
&\hspace{2.5cm}+\frac{1}{3}\biggl\{\sum_{m=1}^\nbs \Bigl(Er_m(E,H)+O\bigl(k^2\norm*{\cnt{\prv_r}E}_{\Lp{2}{D_m}}\rad^\frac{7}{2} +\norm*{\cnt{\prb_r}H}_{\Lp{2}{D_m}}\rad^\frac{5}{2}\bigr)\Bigr)^2\biggl\}^\frac{1}{2}\Biggr],\\
&\leq\frac{9\mu_{(\prv_r,\prb_r)}^+}{8\mu_{(\prv_r,\prb_r)}^-}\Biggl[\Bigl(\sum_{m=1}^\nbs Er_m(H,E)^2\Bigr)^\frac{1}{2}+\Biggl(\sum_{m=1}^\nbs O\Bigl(k^2\norm*{\cnt{\prb_r}H}_{\Lp{2}{D_m}}\rad^\frac{7}{2} +\norm*{\cnt{\prv_r}E}_{\Lp{2}{D_m}}\rad^\frac{5}{2}\Bigr)^2\Biggr)^\frac{1}{2}\\
&\hspace{2cm}+\Bigl(\sum_{m=1}^\nbs Er_m(E,H)^2\Bigr)^\frac{1}{2}+\Biggl(\sum_{m=1}^\nbs O\Bigl(k^2\norm*{\cnt{\prv_r}E}_{\Lp{2}{D_m}}\rad^\frac{7}{2} +\norm*{\cnt{\prb_r}H}_{\Lp{2}{D_m}}\rad^\frac{5}{2}\Bigr)^2\Biggr)^\frac{1}{2}\
\Biggr],\\
&\leq\frac{9\mu_{(\prv_r,\prb_r)}^+}{8\mu_{(\prv_r,\prb_r)}^-}\Biggl[O\Biggl(\norm{H}^2_{\Lp{2}{D}}\frac{\rad^{11}}{\dist^8}+\Bigl(\norm{E}^2_{\Lp{2}{D}}+\norm{H}^2_{\Lp{2}{D}}\Bigr)\frac{(\abs{k}+2)^3\rad^{11}\abs{\ln(\dist)}}{\dist^6}\Biggr)^\frac{1}{2}\\
&\hspace{2.8cm}+\Bigl( O\biggl(k^2(\sum_{m=1}^\nbs\norm*{\cnt{\prb_r}H}^2_{\Lp{2}{D_m}}\rad^7)^\frac{1}{2} +(\sum_{m=1}^\nbs\norm*{\cnt{\prv_r}E}^2_{\Lp{2}{D_m}}\rad^{5})^\frac{1}{2}\biggr)\\
&\hspace{2.8cm}+O\Biggl(\norm{E}^2_{\Lp{2}{D}}\frac{\rad^{11}}{\dist^8}+\Bigl(\norm{E}^2_{\Lp{2}{D}}+\norm{H}^2_{\Lp{2}{D}}\Bigr)\frac{(\abs{k}+2)^3\rad^{11}\abs{\ln(\dist)}}{\dist^6}\Biggr)^\frac{1}{2}\\
&\hspace{2.8cm}+\Bigl( O\bigl(k^2(\sum_{m=1}^\nbs\norm*{\cnt{\prv_r}E}^2_{\Lp{2}{D_m}}\rad^7)^\frac{1}{2} +
(\sum_{m=1}^\nbs\norm*{\cnt{\prb_r}H}^2_{\Lp{2}{D_m}}\rad^{5})^\frac{1}{2}\bigr)^2\Bigr)^\frac{1}{2}\Biggr].\\
\end{aligned}
\end{equation*}  
With the estimates \eqref{Estimation-For-Plan-Waves} and \eqref{Estimation-For-Plan-Waves1}, for the scattering of plane waves, we obtain
\begin{equation}
\begin{aligned}
\Bigl({\sum_{m=1}^{\nbs}\abs*{\wh{\mc{Q}}^{\prb_r}_m-\mc{Q}^{\prb_r}_m}^2}\Bigr)^\frac{1}{2}
\leq\frac{4{\mu_{(\prv_r,\prb_r)}^+}}{8\mu_{(\prv_r,\prb_r)}^-}\times\frac{\bs{c}_\infty(\abs{k}+1)}{\min(c^{\prv^-}_\infty,c^{\prb^-}_\infty)\bs{c}_r^\frac{3}{2}}\Biggl[~O&\Biggl(\frac{\rad^{11}}{\dist^8}+\frac{(\abs{k}+2)^3\rad^{11}\abs{\ln(\dist)}}{\dist^6}\Biggr)^\frac{1}{2}\\
+ O&\bigl(k^2c_{\infty}\rad^{7/2}+c_{\infty} \rad^{5/2}\bigl) \Biggr],\\
\end{aligned}
\end{equation} hence,
\begin{equation}\label{Estim-Error-Approximation-L.SQ}
\begin{aligned}
&\Bigl({\sum_{m=1}^{\nbs}\abs*{\wh{\mc{Q}}^{\prb_r}_m-\mc{Q}^{\prb_r}_m}^2}\Bigr)^\frac{1}{2}
=O\Bigl(\frac{\bs{c}_{(\prv_r,\prb_r)}(\abs{k}+1)}{\bs{c}_r^{3/2}}\Bigl[\frac{1}{\bs{c}_r^4}+\rad\,\abs{\ln(\bs{c}_r\rad)}+ \rad \Bigr]\rad^{3/2}\Bigr),\\
\end{aligned}
\end{equation} and with similar calculations follows
\begin{equation}\label{Estim-Error-Approximation-L.SR}
\begin{aligned} 
&\Bigl({\sum_{m=1}^{\nbs}\abs*{\wh{\mc{R}}^{\prb_r}_m-\mc{R}^{\prb_r}_m}^2}\Bigr)^\frac{1}{2}
=O\Bigl(\frac{\bs{c}_{(\prv_r,\prb_r)}(\abs{k}+1)}{\bs{c}_r^{3/2}}\Biggl[\frac{1}{\bs{c}_r^4}+\rad\,\abs{\ln(\bs{c}_r\rad)}+ \rad \Biggr]\rad^{3/2}\Bigr),
\end{aligned}
\end{equation} with $$\bs{c}_{(\prv_r,\prb_r)}:=\frac{4{\mu_{(\prv_r,\prb_r)}^+}}{8\mu_{(\prv_r,\prb_r)}^-}\times\frac{\bs{c}_\infty(\abs{k}+1)}{\min(c^{\prv^-}_\infty,c^{\prb^-}_\infty)}.$$
Recalling the approximation \eqref{Far-Field-Approximation-Anisotropic-Symmetric}, we have 
\begin{equation}
\begin{aligned}E^\infty(\hat{x})=&\sum_{m=1}^{\nbs}\biggr(\frac{k^2}{4\pi}e^{-ik\hat{x}\cdot \z_m}\hat{x}\times\Bigl(\mc{R}_m^{\prv_r}\times \hat{x}\Bigr)+\frac{ik}{4\pi}e^{-ik\hat{x}\cdot \z_m}\hat{x}\times\mc{Q}_m^\prb_r\Biggl)+O\bigl(\rad\bigr),\\
=&\sum_{m=1}^{\nbs}\biggr(\frac{k^2}{4\pi}e^{-ik\hat{x}\cdot \z_m}\hat{x}\times\Bigl(\wh{\mc{R}}_m^{\prv_r}\times \hat{x}\Bigr)+\frac{ik}{4\pi}e^{-ik\hat{x}\cdot \z_m}\hat{x}\times\wh{\mc{Q}}_m^\prb_r\Biggl)+O\bigl(\rad\bigr)\\ &+\Bigl(\sum_{m=1}^{\nbs}\frac{\abs{k}^2}{16\pi^2}\abs{e^{-ik\hat{x}\cdot z_m}}\Bigr)^\frac{1}{2} \Bigl(\sum_{m=1}^{\nbs}\abs*{\wh{\mc{Q}}^{\prb_r}_m-\mc{Q}^{\prb_r}}^2\Bigr)^\frac{1}{2}\\
&+\Bigl(\sum_{m=1}^{\nbs}\frac{\abs{k}^2}{16\pi^2}\abs{e^{-ik\hat{x}\cdot z_m}}\Bigr)^\frac{1}{2} \Bigl(\sum_{m=1}^{\nbs}\abs*{\wh{\mc{R}}^{\prv_r}_m-\mc{R}^{\prv_r}}^2\Bigr)^\frac{1}{2},
\end{aligned}
\end{equation} which with the estimations \eqref{Estim-Error-Approximation-L.SQ} and \eqref{Estim-Error-Approximation-L.SR} reduces to 
\begin{equation}
\begin{aligned}E^\infty(\hat{x})
=&\sum_{m=1}^{\nbs}\biggr(\frac{k^2}{4\pi}e^{-ik\hat{x}\cdot \z_m}\hat{x}\times\Bigl(\wh{\mc{R}}_m^{\prv_r}\times \hat{x}\Bigr)+\frac{ik}{4\pi}e^{-ik\hat{x}\cdot \z_m}\hat{x}\times\wh{\mc{Q}}_m^\prb_r\Biggl)+O\bigl(\rad\bigr)
\\&+\frac{\abs{k}}{2\pi}\frac{1}{\dist^{3/2}}O\Bigl(\frac{\bs{c}_{(\prv_r,\prb_r)}(\abs{k}+1)}{\bs{c}_r^{3/2}}\Bigl[\frac{1}{\bs{c}_r^4}+\rad\,\abs{\ln(\bs{c}_r\rad)}+ \rad \Bigr]\rad^{3/2}\Bigr) ,
\end{aligned}
\end{equation} or
\begin{equation}
\begin{aligned}E^\infty(\hat{x})=
&\sum_{m=1}^{\nbs}\biggr(\frac{k^2}{4\pi}e^{-ik\hat{x}\cdot \z_m}\hat{x}\times\Bigl(\wh{\mc{R}}_m^{\prv_r}\times \hat{x}\Bigr)+\frac{ik}{4\pi}e^{-ik\hat{x}\cdot \z_m}\hat{x}\times\wh{\mc{Q}}_m^\prb_r\Biggl)+O\bigl(\rad\bigr)
\\&+\frac{\abs{k}}{2\pi}O\Bigl(\frac{\bs{c}_{(\prv_r,\prb_r)}(\abs{k}+1)}{\bs{c}_r^{3}}\Bigl[\frac{1}{\bs{c}_r^4}+\rad\,\abs{\ln(\bs{c}_r\rad)}+ \rad \Bigr]\Bigr).
\end{aligned}
\end{equation} 
\item The approximation \eqref{Farfield-Anis-NonSymCase-Approximation-FullError} can be justified in a similar way replacing, respectively, $\mcntm{\prv_r}{m}\mc{R}_m$ and $\mcntm{\prb_r}{m}\mc{Q}_m$ by $\mc{R}_m$    and $\mc{Q}_m$.
\end{itemize}
\end{proof}

\section{Scattering by perfect conductors, i.e.  (\ref{Maxwell-Eq-Perfect-Cond})}\label{Section-Perfect-Conductor}
\subsection{Preliminaries and notations}
\subsubsection{Boundary traces, Surface gradient and related operator}
Let $\phi$ and $U$ be smooth, respectively, scalar and vector valued functions.
The tangential gradient of $\phi,$ is given by $\nabla_{\ta} \phi:=-\nu\times\bigl(\nu\times\nabla \phi\bigr),$  
where $\nu$ is the outward unit normal to $D$. Then the (weak) surface divergence of $A,$ of tangential fields, that is $\nu\cdot A=0$, is defined by the duality identity
\begin{equation}\label{Def-Surface-Div}
\int_{\p D} \phi \,\Div A~ds=-\int_{\p D}\nabla_{\ta}\phi\cdot A~ds.\end{equation} or by the relation, for $A=\nu\times U$,
\begin{equation}\Div A= -\nu\cdot \curl U.\end{equation}

Obviously, from \eqref{Def-Surface-Div}, if $\Div A$ exists, taking $\phi_m(x)=1$, gives $\int_{\p D} \Div A(x)~ds_x=0,$  
see \cite{DMM96} and Chapter 2 in \cite{ColtonKress:2013} for more details.\\

The spaces $\LDivsp{D}$ denote the space of all tangential fields of ${\Lp{2}{\p D}}$ that have an $\LtwospO{D}$ weak surface divergence, precisely 
$$\LDivsp{D}=\Bigl\{A\in {\Lp{2}{\p D}};\, A\cdot\nu=0, \mbox{ such that } \Div A\in \LtwospO{D}\Bigr\},$$ where
$$\mathbb{L}_0^2{(\p D)}:=\Bigl\{A\in {\Lp{2}{\p D}} \mbox{ such that }  \int_{\p D}Ads=0 \Bigr\}. $$
\bigskip

We recall, the usual Single and Double layer-potentials, either scalar or vector field, defined, respectively, as follows
\begin{align}
\SLP{k}{D}(~\cdot~):=\int_{\partial D}\Phi_{k}(x,y)(~\cdot~(y))~ds_y,\\
\DLP{k}{D}(~\cdot~):=\int_{\partial D}\frac{\p\Phi_{k}(x,y)}{\p \nu_y}(~\cdot~(y))~ds_y.
\end{align} These fields are $C^\infty (\R^3\setminus\p D).$ 
In addition, they admit the following Dirichlet  trace, for almost  every $x\in \p D$  
\begin{align}
\DTr{\SLP{k}{D}(A)(x)}{\pm}&:=[S_{\p D}^k](A)(x):=\lim_{{\substack{{s \rightarrow x}\\{s\in \Gamma_{\pm}}(x)}}} \SLP{k}{D}(A)(s)=\int_{\partial D}\Phi_{k}(x,y)A(y)~ds_y,\label{Single-layer-trace}\\
\DTr{\DLP{k}{D}(A)(x)}{\pm}&:=[\pm{I}/{2} +K^k_{\p D}](A)(x):=\lim_{{\substack{{s \rightarrow x}\\{s\in \Gamma_{\pm}}(x)}}} \DLP{k}{D}(A)(s)=\pm\frac{1}{2}A(x) +p.v.\int_{\partial D}\frac{\p\Phi_{k}(x,y)}{\p \nu_y} A(y)~ds_y,\label{Double-layer-trace}
\end{align} were $\bigl\{\Gamma_{\pm}(x),\,~ x\in\cup_{m=1}^{\nbs}\p D_m\bigr\}$ is a family of doubly truncated cones with a vertex at $x$, which lies in both sides of the boundary,
such that $\Gamma_{\pm}(x)\cap D^\mp=\emptyset$
and the integral \eqref{Double-layer-trace} is taken in the principal value of Cauchy sense.
For every $s\in[0,1],$ all the operators 
\begin{align*}
[S_{\p D}^{0}]&:{H^{-s}}(\partial D)\longrightarrow{H^{1-s}}(\partial D), \\ [\pm{I}/{2} +K^0_{\p D}]&: {H^s}(\partial D)\rightarrow {H^s}(\partial D),\\
[\pm{I}/{2}+(K_{\p D}^0)^*]&: {H^{-s}_0}(\partial D)\rightarrow {H^{-s}_0}(\partial D),
\end{align*} are isomorphisms. More details can be found in \cite{verchota-84 ,MD:InteqnsaOpethe1997}.
\subsubsection{Virtual-mass and Polarization tensor}
The following two quantities will play an important role in the sequel \begin{align}\label{Polarization&virtualmass-Tensor}
\Polt{m}:=\int_{\p D_m}[-{I}/{2}+(\DLP{0}{D_m})^*]^{-1}(\nu) y^*ds_y,~~
\Vmt{m}:= \int_{\p D_m}[\frac{1}{2} I+(\DLP{0}{D_m})^*]^{-1}(\nu) y^*ds_y.
\end{align} 
Both  $-\Polt{m}$ and $\Vmt{m}$ are positive-definite symmetric matrices, (see {Lemma 5} and {Lemma 6} in \cite{Cedio-Fengya-1998} or Theorem 4.11 in \cite{Ammaripola}) and satisfy the following scales
\begin{align}\label{scalingtensors}
\Polt{m}=\rad^3\Poltscale{i},~\text{and }~\Vmt{m}= \rad^3 \Vmtscale{m}.
\end{align}  
For $i\in\{1,...,\nbs\}$ let $(\mu_i^{\mathcal{T}})^+$, $(\mu_i^{\mathcal{P}})^+$ be the respective maximal eigenvalues of $\Vmtscale{m}$, $-\Poltscale{m}$,  and let $(\mu_i^{\mathcal{T}})^{-}$, $(\mu_i^{\mathcal{P}})^{-}$ be their minimal ones. We define
\begin{equation}\label{mudefinition}
\begin{aligned}
\mu^+:=\max_{i\in\{1,...,\nbs\}}((\mu_i^{\mathcal{T}})^{+}, (\mu_i^{\mathcal{P}})^{+}),~~\mu^-:=\min_{i\in\{1,...,\nbs\}}((\mu_i^{\mathcal{T}})^{-}, (\mu_i^{\mathcal{P}})^{-}).
\end{aligned}
\end{equation} Hence for every vector $\mathcal{C}$, we get 
\begin{equation} \label{Tensor-Inequalities}  
\begin{aligned}
\mu^- \abs{\mathcal{C}}^2\rad^3\leq [\mathcal{T}_{\p D_m}]~\mathcal{C}\cdot\mathcal{C}\leq \rad^3 \mu^+ \abs{\mathcal{C}}^2,\\
\mu^- \abs{\mathcal{C}}^2\rad^3\leq -[\mathcal{P}_{\p D_m}]~\mathcal{C}\cdot\mathcal{C}\leq \rad^3\mu^+ \abs{\mathcal{C}}^2.
\end{aligned}
\end{equation}
\subsection{Existence and uniqueness of the solution} 
\subsubsection{Maxwell layer potentials and boundary traces}
The well-posedness of the problem \eqref{Maxwell-Eq-Perfect-Cond} is addressed in \cite{ColtonKress:2013} and \cite{NJC:Spsc2001}, for instance, and the unique solution can be expressed in terms of layer potentials. 
Especially, when $k\in\C\setminus\R^*,$ with $\Im k>0,$ or with additional hypothesis on $k\in \R^+,$ we can use the following representation, 
\begin{equation}
E(x):=E^\n(x)+\curl\SLP{k}{D}(A)(x),\, ~x\in D^+=\R^3\setminus\cup_{m=1}^{\nbs}{\ol{D}_i}\label{Electric-Field-representation}
\end{equation}  where $A$ is the unknown vector density in $\LDivsp{D}$, to be found to solve the problem. 
In this case, the magnetic field is given by \begin{equation}\label{Magnetic-Field-representation}
H(x):=H^\n(x)+\frac{1}{ik}\curl\curl\SLP{k}{D}(A)(x),
\end{equation} with, for almost  every $x\in\cup_{m=1}^{\nbs}\p D_m,$ the following traces are valid, see \cite{DMM96}), are valid \begin{equation}\label{Continuity-Electric-Magnetic-potential}
\begin{aligned}
[{I}/{2}+\bs{M}_{\p D}^{k}](A)(x):=&\nu\times\lim_{{\substack{{s \rightarrow x}\\{s\in \Gamma_{\pm}}(x)}}} \curl  \SLP{k}{ D}(A)(s)=\pm\frac{1}{2} A+ \nu\times\curl\int_{\partial D}\Phi_{k}(x,y)A(y)~ds_y,\\
[\bs{N}_{\p D}^{k}](A^{}[2])(x):=&\nu\times\lim_{{\substack{{s \rightarrow x}\\{s\in \Gamma_{\pm}}(x)}}} \curl^2  \SLP{k}{ D}(A)(s)\\
=&k^2\nu\times\int_{\partial D}\Phi_{k}(x,y)\,A(y)~ds_y+\nabla\int_{\partial D}\Phi_{k}(x,y)\Div A(y)~ds_y.
\end{aligned}
\end{equation} 
For the perfect conductor boundary condition, such a representation gives birth to the following singular integral equation  
\begin{align*}
[{I}/{2}+\bs{M}_{\p D}^{k}](A)=\nu\times E^\n,
\end{align*} which can be written as 
\begin{equation}\label{Equivalent-Representation}
[{I}/{2}+\MDLP{k}{m}+\sum_{\stl{j\geq 1}{j\neq m} }^{\nbs} \MDLPmj{k}](A)=\nu\times E^\n,
\end{equation} with, for $x\in \p D_m$ and $ j\in\{1,...,\nbs\},$
\begin{align}
\MDLPmj{k}A(x):=\nu\times\int_{\partial D_j}\nabla_x\Phi_k(x,y)\times A(y)~ds_y.
\end{align}
For the surface divergence, we have 
\begin{equation} 
\Div [{I}/{2}+\bs{M}_{\p D}^{k}](A)(x)=-[k^2\nu \cdot S_{\p D}^{k}](A)-[{I}/{2}-(K_{\p D}^{k})^*](\Div A).\label{DIV-MDk}\end{equation}
In addition, for every $u\in H^1(\p D_m)$ (cf. Lemma 5.11 in \cite{DMM96}) 
\begin{align}\label{nu-time-nabla-interversion}
[{I}/{2}+\bs{M}_{\p D}^{k}](\nu\times {\nabla} u)=\nu\times \nabla[{I}/{2}+K_{\p D}^{k}](u)-k^2~\nu\times [S_{\p D}^{k}](\nu u).\end{align}  
The fact that the operators appearing in (\ref{Equivalent-Representation}) are isomorphisms on $\LDivsp{D}$ is addressed in the next section together with an estimate of the density.  
\subsubsection{A priori estimates of the densities}

The following result together with the fact that $[\pm{I}/{2}+\bs{M}_{\p D}^{k}]$  is of closed range, makes it describe an isomorphism of $\LDivsp{D}$. 
The proof goes in two steps. In the first step, we prove the estimate when the wave number $k$ has a positive imaginary part, $\Im k>0$, (ex. $k=i$).
As $[\bs{M}_{\p D}^{k}-\bs{M}_{\p D}^{i}]$ is compact, we deduce, in a second step, a similar estimates for $k\in\R^+,$ under an appropriate condition on $\bs{c}_r$ (i.e. the dilution parameter).

\begin{theorem}\label{Injectivity-MDLP}
	Let $k\in \C$ such that $\Im(k)>0$. For any $A\in\LDivsp{D}$, there exists a positive constant $\bs{c}_{L,k}$, which depends only on $k$ and the Lipschitz character of the $D$'s, such that 
	\begin{equation}\label{Coercivity_MDLP}
	\norm*{A}^2_{\LDivsp{D}}=\sum_{m=1}^{\nbs}\norm*{A}^2_{{\LDivsp{D_m}}}\leq \bs{c}_{L,k}\norm*{[\pm{I}/{2}+\bs{M}_{\p D}^{k}]A}^2_{\LDivsp{D}}.
	\end{equation}  
\end{theorem} 
In order to derive this estimate, we need to use the following key Helmholtz decomposition of the densities 

\begin{lemma}\label{Decomposition-LDiv-Space}
Each element $A_m$ of ${\LDivsp{D_m}}$ can be decomposed as $A_m=A^{1}_m + A^{2}_m$ where
\begin{equation}
\begin{aligned}
A^{1}_m=\nablat \mathit{v}^A_m\in {\LNDsp{D_m}},~ \mathit{v}_m\in\Hs{1}{\p D_m}\setminus \C~\mbox{ and } A^{2}_m~&\in {\LDivsp{D_m}}\setminus{\LNDsp{D_m}},\,\textit{and}\,
\end{aligned} 
\end{equation}
with the estimates
\begin{align*}
\norm*{A^{2}_m}_{{\Lp{2}{\p D_m}}}\leq \bs{c}_1\rad\norm*{A_m}_{{\LDivsp{D_m}}},~~ \norm*{A^{2}_m}_{{\LDivsp{D_m}}}\leq \bs{c}_2\norm*{A}_{{\LDivsp{D_m}}},\\
	\norm*{A_m^1}_{{\LNDsp{D_m}}}\leq \bs{c}_3\norm*{A_m}_{{\LDivsp{D_m}}},~~\norm*{\mathit{v_m^A}}_{\Lp{2}{\p D_m}} \leq \bs{c}_4 \rad \norm*{A}_{{\LDivsp{D_m}}}.
	\end{align*}  
	where $(\bs{c}_i)_{i=1,2,3,4.}$ are constants which depend only on the Lipschitz character of the $D_m's$.
\end{lemma}
To prove \Cref{Decomposition-LDiv-Space}, it suffices to seek for the solution of the following equation,  
\begin{equation}\label{Helmholtz-Decomposition}
\left\{\begin{aligned}
[\nablat \SLP{0}{D_m}](w)+[\nu\times\SLP{0}{D_m}](W)&=A_m,\\
[\nu\cdot \curl\SLP{0}{D_m}](W)&=\Div A_m,
\end{aligned}\right.
\end{equation} and get the desired estimations, using the scaling properties of the implied operators, with $A^{2}_m=[\nu\times\SLP{0}{D_i}](W)$
and $A^{1}_m=[\nablat \SLP{0}{D_i}](w)$. The details can be found in \cite{bouzekri2019foldy}.
\bigskip

In what follows, $\bs{c}_L$ and $c_k$ will designate generic constants independent of $\rad$ and depend, respectively, on the Lipschitz character of $D$ and $k$.

Suppose that $E$ and $H= \curl E/{i {k}}$ are two vector fields that satisfy the Maxwell system, with a complex wave number $k,~\Im(k)>0$ 
and such that $\nu\times E$ is in $\mathbb{L}^{2,\Div}( \p D:=\cup_{m=1}^{\nbs}\p D_m)$. Hence, in view of the decomposition of \Cref{Estimated-Decom-OnL2D},
there exist $\mathfrak{e}_i$ and $\mathfrak{E}_i$ such that, in each $\p D_m,$ $\nu_i\times E=\nu\times\nabla\De_i+\nu\times\DE_i$
with the following estimate,\begin{align*}
\norm*{\nu\times\DE_i}_{\Lp{2}{\p D_m}}\leq & \bs{c}_1\rad\norm*{E}_{\LDivsp{D_m}},\,~
\norm*{\De_i}_{\Lp{2}{\p D_m}}\leq \bs{c}_4\rad\norm*{E}_{\LDivsp{D_m}}.
\end{align*} recalling that $\norm*{E}_{\LDivsp{D_m}}=\norm*{\nu\times E}_{\Lp{2}{\p D_m}}+\norm*{\nu\cdot\curl E}_{\Lp{2}{\p D_m}}.$

Let us set $\De:=\sum_{m=1}^{\nbs}\ind{\p D_m}\De_i,\,~ \DE:=\sum_{m=1}^{\nbs} \ind{\p D_m}\DE_i$.  
With this notations, for every $x\in\p D=\cup_{m=1}^{\nbs}\p D_m$, we get \begin{equation}\label{Field-Helmholtz-Decomposition}
\nu\times E=\nu\times\nabla\De+\nu\times\DE,
\end{equation}
and the following estimates hold \begin{equation}\label{Estimate-Field-Helmholtz-Decomposition} \begin{aligned}
\norm*{\nu\times\DE}_{\Lp{2}{\cup_{m=1}^{\nbs}\p D_m}}\leq & \bs{c}_1\rad\norm*{E}_{\LDivsp{\cup_{m=1}^{\nbs} D_m}},\,~ 
\norm*{\De_i}_{\Lp{2}{\cup_{m=1}^{\nbs}\p D_m}}\leq & \bs{c}_4\rad\norm*{E}_{\LDivsp{\cup_{m=1}^{\nbs} D_m}}.
\end{aligned}\end{equation}
Indeed, with \Cref{Estimated-Decom-OnL2D},  
\begin{align*}
\norm*{\nu\times\DE}_{\Lp{2}{\cup_{m=1}^{\nbs}\p D_m}}^2=\sum_{m=1}^{\nbs}\int_{\p D_m}(\nu_i\times\DE)^2ds\leq & \sum_{m=1}^{\nbs} \bs{c}_1^2\rad^2\left(\int_{\p D_m}(\nu\times E_i)^2+(\div\nu\times E_i)^2ds\right),\\
\leq & \bs{c}_1^2\rad^2 \sum_{m=1}^{\nbs} \int_{\p D_m}(\nu\times E_i)^2ds=\bs{c}_1^2{\rad}^2\norm*{E}_{\LDivsp{D}}^2.
\end{align*} 
The second estimation can be done in the same way.

\begin{lemma}For any solution $E$ to the Maxwell system, continuous up to the boundary, we have
\begin{equation}\label{leme1-trace-in-rellich-identity}
\begin{aligned}
\abso{\frac{1}{ik}\int_{ \p D}\nu\cdot E \times\curl\ol{ E}~ds} \leq \bs{c}_L~\rad~ \norm*{H}_{\LDivsp{D}}\norm*{E}_{\LDivsp{D}},
\end{aligned}
\end{equation} with $\bs{c}_L$ independent of $\rad$ and $k$, from which follows  
	\begin{equation}\label{curlEandE-in-LD}
	\norm{E}_{\Hcurl{}{D}}^2 \leq \bs{c}_L c_k \rad\norm*{H}_{\LDivsp{D}}\norm*{E}_{\LDivsp{D}}
	\end{equation}  
	where $c_k:=1/\min\bigl({\Im k/\abs{k}^2},\Im(k)\bigr)$.
	
\end{lemma}
\begin{proof}
	Considering the decomposition \eqref{Field-Helmholtz-Decomposition}, and using \eqref{Def-Surface-Div} for the second identity, we have
	\begin{align*}
	\int_{ \p D}E\cdot \nu\times \ol{\curl E}~ds&=\int_{ \p D}(\nabla\De+\DE)\cdot \nu\times \ol{\curl E}~ds=-\int_{ \p D}\De~\Div(\nu\times \ol{\curl E})~ds+\int_{ \p D}\DE\cdot \nu\times \ol{\curl E}~ds,\\
	&= k^2\int_{ \p D}\De~\nu\cdot E~ds+\int_{ \p D}\DE\cdot \nu\times \ol{\curl E}~ds.  
	\end{align*} 
	Taking the absolute value,  with H\"older's inequality and involving the estimates \eqref{Estimate-Field-Helmholtz-Decomposition}, we get successively 
	\begin{equation*}
	\begin{aligned}
	\abs{\frac{1}{ik}\int_{ \p D}E\cdot \nu\times \curl \ol{E}~ds}\leq
	&  \norm*{\De}_{\Lp{2}{\p D}}\norm*{\frac{k^2}{ik}\nu\cdot E}_{\Lp{2}{\p D}}+\norm*{\DE}_{\Lp{2}{\p D}} \norm*{\frac{1}{ik}\nu\times \curl E}_{\Lp{2}{\p D}},\\
	\leq&  \bs{c}_L~\rad~ \norm*{H}_{\LDivsp{D}}\norm*{E}_{\LDivsp{D}}.
	\end{aligned}
	\end{equation*}
	For \eqref{curlEandE-in-LD}, consider the following Green's identity
	\begin{align}\label{Green-Identity-Interior}
	\frac{1}{ik}\int_{ D} \abs{\curl E}^2dx+ik\int_{ D}\abs{E}^2dx&=
	\frac{1}{ik}\int_{\p D}\nu\times E\cdot \curl \ol{E}~ds.
	\end{align} 
	The real part gives, 
	\begin{equation*}
	\frac{\Im k}{\abs{k}^2}\int_{ D} \abs{\curl E}^2dx+\Im k\int_{ D}\abs{E}^2dx=-\Re\Bigl(\frac{1}{ik}\int_{\p D}\nu\times E\cdot \curl \ol{E}~ds\Bigr),
	\end{equation*} then, with $c_k:=1/\min\bigl({\Im k/\abs{k}^2},\Im(k)\bigr)$ involving the estimates \eqref{leme1-trace-in-rellich-identity}, we obtain
	\begin{align*}
	\int_{ D} \abs{\curl E}^2dx+\int_{ D} \abs{E}^2 dx\leq c_k \bs{c}_L~\rad~ \norm*{H}_{\LDivsp{D}}\norm*{E}_{\LDivsp{D}}.
	\end{align*}\end{proof}
\begin{lemma}\label{Trace-Inequalities-EandH}
We have, the following estimates for the exterior ($E^+$) and interior ($E^-$) traces of the radiating electric field $E$	\begin{align}
	\int_{\p D}\abs{\nu\times(\nu\times E^\pm)}^2\leq &\bs{c}_Lc_k\Bigl(\norm*{H^\pm}_{\LDivsp{D}} \norm*{E^\pm}_{\LDivsp{D}}+\norm*{\nu\cdot E^\pm }^2_{\Lp{2}{\p D}}\Bigr),\end{align} and
	\begin{align} 
	\int_{\p D}\abs{\nu\cdot E^\pm }^2  \leq & \bs{c}_Lc_k\Bigl( \norm*{H^\pm}_{\LDivsp{D}} \norm*{E^\pm}_{\LDivsp{D}}+\norm*{\nu\times(\nu\times E^\pm)}^2_{\Lp{2}{\p D}}\Bigr),
	\end{align} and due to the Maxwell equation symmetry, same estimates remain true for the magnetic field $H$.
\end{lemma}
\begin{proof}                                                                     
Let $E$ and $X$ be two vector fields, with $X$ real valued, and consider the following identity
\begin{align}\label{identityRellich}
\int_{ D=\cup_{m=1}^{\nbs} D_m} \div\left({\frac{1}{2}}\abs{E}^2X-Re (\ol{E}\cdot X)~E\right)~dx=\int_{\p D=\cup_{m=1}^{\nbs} \p D_m}\left({\frac{1}{2}}\abs{E}^2X-Re (\ol{E}\cdot X)~E\right)\cdot\nu.\end{align}
Following \cite{MMitrea94}, let $\widehat{X}_m \in \mathbb{C}^{\infty}(\ol{\bsmc{D}_m})$ be compactly supported in $\R^3,$ such that $\widehat{X}_m\cdot\nu_{\p \bsmc{D}_m}\geq \bs{c}_{L_m}>0$ with support as small as wanted.
Such a choice is possible due to Lemma 1.5.1.9 in \cite{grisvard2011elliptic}. Recall that $\bsmc{D}_m\subset{B(0,\frac{1}{2})}$ and $D_m=\rad\bsmc{D}_m+\z_m.$

Set $\bs{c}_{L}^{-}:=\min_{m\in\{1,...,\nbs\}}\bs{c}_{L_m},$ and define, for every $x\in\R^3,$ $$X(x):=\sum_{m=1}^{\nbs} \Bigl(\widehat{X}_m\bigl(\frac{(x-z_m)}{\epsilon}\bigr)\ind{B(\z_m,\frac{3\rad}{2})}\Bigr).$$ 
Then $X$ satisfies, for any positive scalar function $V,$ \begin{equation}\label{EstimX-.nu}
\begin{aligned}\int_{\p D} V X\cdot\nu ds=\sum_{m=1}^{\nbs} \int_{\p \bsmc{D}_m} V(\epsilon s_m+z_m) \widehat{X}(s_m)\cdot\nu_{\p \bsmc{D}_m} \rad^2 ds_{s_m}
&\geq \sum_{m=1}^{\nbs} \bs{c}_{L_m}\int_{\p \bsmc{D}_m} V(\epsilon s_m+z_m) \rad^2 ds\\
&\geq \bs{c}_{L}^{-}\int_{\p D} V ds,
\end{aligned}
\end{equation} and for $\bs{c}_L^{+}:=\max_m\sup_{\bsmc{D}_m}\Bigl( \abs*{\div\widehat{X}}, \abs*{\nabla~ \widehat{X}}, \abs*{\widehat{X}} \Bigr)$, we obtain 
\begin{equation}\left\{\begin{aligned} \sup_{x\in D_m}\abs{\nabla X(x)}&\leq \frac{1}{\rad}\sup_{s\in \bsmc{D}_m }\abs{\nabla_s \widehat{X}(s)}\leq \frac{ \bs{c}^{+}_L}{\rad },\\
	\sup_{x\in {D}_m}\abs{\div X(x)}&\leq \frac{1}{\rad}\sup_{s\in \bsmc{D}_m }\abs{\div_s \widehat{X}(s)}\leq \frac{ \bs{c}^{+}_L}{\rad}.
	\end{aligned}\right.\label{Estim-X}\end{equation}
Being $\abs{E}^2=E\cdot \ol{E}$, we get both $$\div(\abs{E}^2X/2) =\bigl(\abs{E}^2\div X+ (\nabla E) \ol{E}\cdot X+(\nabla \ol{E})~ E \cdot X\bigr)/2,$$ 
and $$\div ((\ol{E}\cdot X)~E)=(\ol{E}\cdot X)~\div E-(\nabla\ol{E}) X\cdot E -(\nabla X) \ol{E}\cdot E.$$ 
Taking their real parts then their difference\footnote{Noticing that $$\Re\left((\nabla \ol{E})X\cdot  E\right)=\Re\left((\nabla E)X\cdot \ol{E} \right), $$ and $$ \Re\left((\nabla \ol{E})X\cdot  E -(\nabla \ol{E})E\cdot X\right)
=\Re\left(\curl\ol{E}\cdot (E\times X)\right)=\Re\left(\curl E\cdot (\ol{E}\times X)\right).$$}, gives $$\div\bigl({\frac{1}{2}}\abs{E}^2X-\Re \bigl((\ol{E}\cdot X)E\bigr)\bigr)=
\Re\bigl(\frac{1}{2}\abs{E}^2~\div X-\ol{E}\cdot X~\div E-\ol{E}\cdot \nabla X~E +\ol{E}\times X\cdot \curl E \bigr).$$ 
As $E$ is divergence free, replacing in \eqref{identityRellich}, we have 
\begin{equation}\label{identityRellich2}\begin{aligned}
\int_{\p D}\bigl({\frac{\abs*{E}^2X}{2}}-\Re (\ol{E}\cdot X)E\bigr)\cdot\nu ds=\Re\int_{D}\Bigl(\frac{\abs*{E}^2\div X}{2}-\ol{E}\cdot (\nabla X)E+\ol{E}\times X\cdot \curl E\Bigr)dv.
\end{aligned}
\end{equation} 
Since $E=(\nu\cdot E)\nu-\nu\times(\nu\times E)$, 
we obtain $$\Re\int_{\p D} (\ol{E}\cdot X)~E\cdot \nu~ds=-\Re\int_{\p D} \Bigl(\nu\times(\nu\times\ol E)\cdot X~\nu\cdot E -\abs{\nu\cdot E }^2\nu\cdot X\Bigr)ds$$
and then the left hand side of \eqref{identityRellich2} becomes
$$\int_{\p D}\Bigl({\frac{1}{2}}\abs{E}^2-\abs{\nu\cdot E }^2\Bigr)X\cdot \nu~ds+\Re\int_{\p D}\nu\times(\nu\times E)\cdot X~\ol\nu\cdot E ~ds,$$
hence it follows, with \eqref{Estim-X},
\begin{align*}
\Bigl|\int_{\p D}\Bigl({\frac{\abs{E}^2}{2}}-\abs{\nu\cdot E }^2\Bigr)\nu\cdot X ds\Bigr| 
\leq \bs{c}^{+}_L\Bigl(\frac{1}{\rad}\int_{ D}\frac{3\abs{E}^2}{2}dv+ \int_{D}\abs*{\ol{E}} \,\abs*{\curl E}dv\Bigr)
&+\abso{\int_{\p D}\nu\times(\nu\times E)\cdot X~\ol\nu\cdot E ~ds},   
\end{align*} As $\abs{E}^2=\abs{\nu\times(\nu\times E)}^2 +\abs{(\nu\cdot E)\nu}^2,$  using H\"older's inequality, we have
\begin{align}\label{Rellich-Id-2}
\abso{\int_{\p D}{\frac{1}{2}}\left(\abs{\nu\times(\nu\times E)}^2-\abs{\nu\cdot E }^2\right)X\cdot \nu~ds}\leq
\bs{c}^{+}_L\Bigl(\frac{1}{\rad} \norm{E}_{\Hcurl{}{D}^2}+\norm*{\nu\times E}_{\Lp{2}{\p D}}\norm*{\nu\cdot E}_{\Lp{2}{\p D}}\Bigr),
\end{align} which, in view of \eqref{EstimX-.nu}, gives both 
\begin{equation*}
\begin{aligned}
\int_{\p D}\abs{\nu\times(\nu\times E)}^2ds\leq \frac{\bs{c}_{L}^{+}}{\bs{c}_{L}^{-}}\Bigl(\frac{\norm{E}_{\Hcurl{}{D}^2}}{\rad}
+\norm*{\nu\times E}_{\Lp{2}{\p D}}\norm*{\nu\cdot E}_{\Lp{2}{\p D}}\Bigr)+\norm*{\nu\cdot E }^2_{\Lp{2}{\p D}},\end{aligned}\end{equation*}
and \begin{equation*}
\begin{aligned}
\int_{\p D}\abs{\nu\cdot E }^2ds  \leq
\frac{\bs{c}_{L}^{+}}{\bs{c}_{L}^{-}}\Bigl(\frac{\norm{E}_{\Hcurl{}{D}^2}}{\rad}+\norm*{\nu\times E}_{\Lp{2}{\p D}}\norm*{\nu\cdot E}_{\Lp{2}{\p D}}\Bigr)
+ \norm*{\nu\times(\nu\times E)}^2_{\Lp{2}{\p D}}.\end{aligned}\end{equation*} 
Using \eqref{curlEandE-in-LD}, we get \begin{equation*}\begin{aligned}
	\int_{\p D}\abs{\nu\times(\nu\times E)}^2ds\leq
	\bs{c}_{L,k}\biggl( \norm*{H}_{\LDivsp{D}}\norm*{E}_{\LDivsp{D}}
	+\norm*{\nu\times E}_{\Lp{2}{\p D}}\norm*{\nu\cdot E}_{\Lp{2}{\p D}}+\norm*{\nu\cdot E }^2_{\Lp{2}{\p D}}\biggr),\end{aligned}\end{equation*}
and	\begin{equation*}\begin{aligned}
	\int_{\p D}\abs{\nu\cdot E }^2ds \leq \bs{c}_{L,k}\Bigl(\norm*{H}_{\LDivsp{D}} \norm*{E}_{\LDivsp{D}}
	+\norm*{\nu\times E}_{\Lp{2}{\p D}}\norm*{\nu\cdot E}_{\Lp{2}{\p D}}+ \norm*{\nu\times(\nu\times E)}^2_{\Lp{2}{\p D}}\Bigr).    
	\end{aligned}\end{equation*}
Then \footnote{Being $\norm*{\nu\times E}_{\Lp{2}{\p D}}\leq\norm*{E}_{\LDivsp{D}},$ $\norm*{\nu\cdot E}_{\Lp{2}{\p D}}\leq\frac{1}{\abs{k}^2}\norm*{H}_{\LDivsp{D}},$}
	\begin{equation}\label{Rell-inequality}
	\left\{\begin{aligned}
	\int_{\p D}\abs{\nu\times(\nu\times E)}^2ds\leq \bs{c}_{L,k}\Bigl((1+\frac{1}{\abs{k^2}})\norm*{H}_{\LDivsp{D}} \norm*{E}_{\LDivsp{D}}+\norm*{\nu\cdot E }^2_{\Lp{2}{\p D}}\Bigr),\\
	\int_{\p D}\abs{\nu\cdot E }^2ds  \leq \bs{c}_{L,k}\Bigl( (1+\frac{1}{\abs{k^2}})\norm*{H}_{\LDivsp{D}} \norm*{E}_{\LDivsp{D}}+\norm*{\nu\times(\nu\times E)}^2_{\Lp{2}{\p D}}\Bigr).                  
	\end{aligned}\right.
	\end{equation}
Concerning the exterior field, we consider a ball $B_r(0)$, with a sufficiently large radius $r$, which contains $D=\cup_{m=1}^{\nbs} D_m$ and the support of $X$. Let $D_r:=B_r\setminus D,$ 
applying Green's formula \eqref{Green-Identity-Interior} on $D_r$
	gives \begin{align} \label{Curl-and-E-exterior}
	\int_{ D_r}\bigl( \frac{1}{ik}\abs{\curl E}^2+ik\abs{E}^2\bigr)dx=\int_{\p B(0,r)}\nu\times\ol{E}\cdot \frac{1}{ik}\curl E~ds-\int_{\p D}\nu\times\ol{E}\cdot \frac{1}{ik}\curl E~ds.
	\end{align}   
As $\Im k>0,$ due to the asymptotic behavior \eqref{electric-farfield}, we have
 $$2\Re\biggl(\int_{\p B_r} E\cdot~\frac{1}{ik}\curl E\times\nu ~ds\biggr)\longrightarrow 0,$$ 
then, the real part of \eqref{Curl-and-E-exterior} and letting $r$ to $\infty$ gives
	\begin{equation*}\label{Trace-Estimate-Outside}
	-\int_{\R^3\setminus D} \bigl(\frac{\Im(k)}{\abs{k}^2}\abs{\curl E}^2+\Im k\abs{E}^2\bigr)dv\geq -\Re(\int_{\p D}\nu\times\ol{E}\cdot \frac{1}{ik}\curl E~ds),
	\end{equation*} hence
	\begin{align}
	 \int_{ \R^3\setminus \cup_{m=1}^{\nbs}\ol{D_m}} \bigl(\abs{\curl E}^2+\abs{E}^2\bigr)dv\leq &c_k\Re(\int_{\p D}\nu\times\ol{E}\cdot \frac{1}{ik}\curl E~ds), 
	\end{align} which is the analogue formula for \eqref{curlEandE-in-LD}. The rest of the proof is identical. 
\end{proof} 
As a direct consequence of \Cref{Trace-Inequalities-EandH}, we have the following assertion 
\begin{lemma}\label{E-L2DIV-HL2DIV-Lemma} For every solution $(E,H)$ to the Maxwell system, continuous up to the boundary, with wave number $k$, we have both 
	\begin{align}\label{E-L2DIV-HL2DIV}
	\begin{aligned}
	\norm*{E^\pm}_{\LDivsp{D}}\leq \bs{c}_{L,k}^2 \norm*{H^\pm}_{\LDivsp{D}}\\
	\end{aligned}
	\end{align} and
	\begin{align}\label{E-L2DIV-HL2DIV2}
	\begin{aligned}
	\norm*{H^\pm}_{\LDivsp{D}}\leq \bs{c}_{L,k}^2 \norm*{E^\pm}_{\LDivsp{D}}.
	\end{aligned} 
	\end{align}
	In addition, taking $E=\curl\SLP{k}{D}(A)$ for some $A\in\LDivsp{D}$ with the trace limites as in \eqref{Continuity-Electric-Magnetic-potential} we have
\begin{equation}\label{Equivalence-MDLP-MSLP}\begin{aligned}
	\norm*{[\pm{I}/{2}+M^{k}_{\p D}]A}_{\LDivsp{D}}\leq \bs{c}_{L,k} \norm*{[N^{k}_{\p D}]A}_{\LDivsp{D}},
	\end{aligned}\end{equation} and	
	\begin{equation}\label{Equivalence-MDLP-MSLP2}\begin{aligned}
	\norm*{[N^{k}_{\p D}]A}_{\LDivsp{D}}\leq \bs{c}_{L,k} \norm*{[\pm{I}/{2}+M^{k}_{\p D}]A}_{\LDivsp{D}}.
	\end{aligned}\end{equation}
\end{lemma}
\begin{proof}(\Cref{E-L2DIV-HL2DIV-Lemma})
For the first part, being $E=\nu\times(\nu\times E)+\nu\cdot E~\nu$, we have 
\begin{align*}
\norm*{E}^2_{\Lp{2}{\p D}}\leq \norm*{\nu\times(\nu\times E)}^2_{\Lp{2}{\p D}}+\norm*{\nu\cdot E}^2_{\Lp{2}{\p D}}.
\end{align*} 
Hence, with the first inequality of \Cref{Trace-Inequalities-EandH}, we get 
\begin{equation}\label{E-L2DIV-EandH-I1}
\norm*{E^\pm}^2_{\Lp{2}{\p D}}\leq \bs{c}_{L,k}\Bigl(\frac{\abs{k}^2+1}{\abs{k}^2}\norm*{H^\pm}_{\LDivsp{D}} \norm*{E^\pm}_{\LDivsp{D}}
+\norm*{\nu\cdot E^\pm }^2_{\Lp{2}{\p D}}\Bigr)+\norm*{\nu\cdot E^\pm}^2_{\Lp{2}{\p D}}
\end{equation} and with the same argument, for the magnetic field using the second inequality of \Cref{Trace-Inequalities-EandH}, we get
\begin{equation}
\norm*{H^\pm}^2_{\Lp{2}{\p D}}\leq \bs{c}_{L,k}\Bigl(\frac{\abs{k}^2+1}{\abs{k}^2}\norm*{H^\pm}_{\LDivsp{D}} \norm*{E^\pm}_{\LDivsp{D}}+\norm*{\nu\times H^\pm }^2_{\Lp{2}{\p D}}\Bigr)+\norm*{\nu\times H^\pm}^2_{\Lp{2}{\p D}}.
\end{equation} Adding the last two inequalities and observing that 
\begin{equation}\label{E-L2DIV-EandH-I2}
\norm*{E^\pm}_{\LDivsp{D}}\leq \norm*{E^\pm}^2_{\Lp{2}{\p D}}+\norm*{H^\pm}^2_{\Lp{2}{\p D}}
\end{equation} and that
\begin{equation}
\norm*{\nu\cdot E^\pm }^2_{\Lp{2}{\p D}}+\norm*{\nu\times H^\pm}^2_{\Lp{2}{\p D}}\leq(1+{1}/{\abs{k}}) \norm*{H^\pm}_{\LDivsp{D}}(\norm{E}_{\Lp{2}{\p D}}+\norm{H}_{\Lp{2}{\p D}}),
\end{equation}
gives
\begin{align*}
\norm*{H^\pm}^2_{\Lp{2}{\p D}}+\norm*{E^\pm}^2_{\Lp{2}{\p D}}&\\
\leq \bs{c}_{L,k}\Biggl(\Bigl(&\norm*{H^\pm}_{\LDivsp{D}} \norm*{E^\pm}_{\LDivsp{D}}+\norm*{H^\pm}_{\LDivsp{D}}(\norm{E}_{\Lp{2}{\p D}}+\norm{H}_{\Lp{2}{\p D}})\Bigr)\\
 &+\norm*{\nu\times H^\pm }^2_{\Lp{2}{\p D}}+\norm*{\nu\cdot E^\pm}^2_{\Lp{2}{\p D}}\Biggr),\\
\leq ~\bs{c}_{L,k} ~~~&\norm*{H^\pm}_{\LDivsp{D}}\Bigl(\norm*{H^\pm}^2_{\Lp{2}{\p D}}+ \norm*{E^\pm}^2_{\Lp{2}{\p D}}\Bigr)^\frac{1}{2}.
\end{align*}
Hence \begin{equation*}
	\Bigl(\norm*{H^\pm}^2_{\Lp{2}{\p D}}+\norm*{E^\pm}^2_{\Lp{2}{\p D}}\Bigr)^\frac{1}{2}\leq \bs{c}_{L,k} \norm*{H^\mp}_{\LDivsp{D}}\leq \bs{c}_{L,k}\Bigl(\norm*{H^\mp}^2_{\Lp{2}{\p D}}+\norm*{E^\mp}^2_{\Lp{2}{\p D}}\Bigr)^\frac{1}{2}.
	\end{equation*}
Interchanging $E$ and $H$ in \eqref{E-L2DIV-EandH-I1} and \eqref{E-L2DIV-EandH-I2} then repeating the same calculation gives
\begin{equation*}
\Bigl(\norm*{H^\pm}^2_{\Lp{2}{\p D}}+\norm*{E^\pm}^2_{\Lp{2}{\p D}}\Bigr)^\frac{1}{2}\leq \bs{c}_{L,k} \norm*{E^\pm}_{\LDivsp{D}}\leq \bs{c}_{L,k}\Bigl(\norm*{H^\pm}^2_{\Lp{2}{\p D}}+\norm*{E^\pm}^2_{\Lp{2}{\p D}}\Bigr)^\frac{1}{2}.
\end{equation*} The two last inequalities give us the desired result. \end{proof}

The estimates in \Cref{Injectivity-MDLP} follow using the triangular inequality and \eqref{Equivalence-MDLP-MSLP}. Indeed, 
\begin{equation}\begin{aligned}
\norm*{A}_\LDivsp{D}&=\norm{({I}/{2}+M^{k}_{\p D})A-(-{I}/{2}+M^{k}_{\p D})A}_{\LDivsp{D}},\\
&\leq\norm{({I}/{2}+M^{k}_{\p D})A}_{\LDivsp{D}}+\norm{(-{I}/{2}+M^{k}_{\p D})A}_{\LDivsp{D}},\\ &\leq(1+\bs{c}_{L,k}^2)\norm{(\pm{I}/{2}+M^{k}_{\p D})A}_{\LDivsp{D}}.\end{aligned}
\end{equation}  
\begin{remark}
	Similar calculations lead to the following estimates 
	\begin{equation}
	\norm*{[\pm{I}/{2}+M^{0}_{\p D}]A}_{\Lp{2}{D}}\leq \bs{c}_{L} 
	\norm*{[\mp{I}/{2}+M^{0}_{\p D}]A}_{\Lp{2}{D}}+ \norm*{[S^{0}_{\p D}]A}_{\Lp{2}{D}}
	\end{equation} with instead of \eqref{Green-Identity-Interior}, we use the following inequality 
	$$\int_{\p D}\nu\times [S^{0}_{\p D}] A\cdot\DTr{\curl\,[S^{0}_{\p} D]A}{\pm} ds\geq\norm{\curl S^{0}_{\p D}(A)}^2-\norm{\nabla S^{0}_{\p D}(\Div A)}\norm{S^{0}_{\p D}(\Div A)}$$ including the following one 
	$$\int_{\p D}\nu\times\DTr{\curl\,[S^{0}_{\p} D]A}{\pm}\cdot \nabla [S^{0}_{\p D}] \Div A  ds\geq \norm{\nabla S^{0}_{\p D}(\Div A)}, $$
	where the norms are either of $\Lp{2}{D}$ or ${\Lp{2}{\R^3\setminus D}}$ with the use of a similar decomposition (i.e~\Cref{Decomposition-LDiv-Space}).
\end{remark}

\begin{theorem} \label{Density-Estimate-And-Cond}
There exist a positive constant  $\bs{c}_{L,k}$ which depends only on the Lipschitz character of $\bsmc{D}_m$'s and $k$  such that,
\begin{itemize} 
	\item If $\Im k=0$ and $\bs{c}_r=\frac{\dist}{\rad}\geq {{c_0}2\bs{c}_L(\abs{k}+4)}{\bs{c}_{L,1}},$ then
\begin{equation}\label{Estimation-Densities-AikReal}
\sum_{m=1}^{\nbs} \norm*{A}_{{\LDivsp{D_m}}}^2\leq 3\bs{c}_{L,1} \nbs \rad^2 ,
\end{equation}  for $\bs{c}_{L,1}$ stands for the constant that appears in \eqref{Injectivity-MDLP} of \Cref{Injectivity-MDLP} for $k=i.$
\item If $\Im (k)>0$ and $\bs{c}_r=1$ 
\begin{equation}\label{Estimation-Densities-AikComplex}
\sum_{m=1}^{\nbs} \norm*{A}_{{\LDivsp{D_m}}}^2\leq \bs{c}_{L,k} \nbs \rad^2. 
\end{equation}
\end{itemize}
\end{theorem}

The following lemma will help us provide a short cut to derive the desired estimates and approximations, using the results of \cref{Lin-Sys-Anisotropic}.
\begin{lemma}\label{Estimated-Decom-OnL2D}
	For each $m\in\{1,...,\nbs\}$ and $x\in{D_m}$, we set, 
	\begin{equation}
	\Har{1,A}_m(x):=\SLP{0}{D_m}(w)(x)~ \mbox{ and }~ \Har{2,A}_m(x):=\SLP{0}{D_m}(W)(x). 
	\end{equation} where $w$ and $W$ are solutions of \eqref{Helmholtz-Decomposition}, then 
	\begin{equation}\label{Density-Decomposition-notation}
	A=\nu\times\nabla\Har{1,A}_m(x)+\nu\times\Har{2,A}_m(x),
	\end{equation} and we have the following estimates holds 
	\begin{equation}
	\begin{aligned}\label{Esimation-Har-A1}
	\norm*{\nabla\Har{1,A}_m}_{\Lp{2}{D_m}}\leq \bs{c}_L\rad^\frac{1}{2} \norm{A_m}_{{\LDivsp{D_m}}},
	\end{aligned}\end{equation}	and 
	\begin{equation} 
	\begin{aligned}\label{Esimation-Har-A2}
	\norm*{\curl\Har{2,A}_m}_{\Lp{2}{D_m}}\leq \bs{c}_L\rad^\frac{1}{2} \norm{A_m}_{{\LDivsp{D_m}}}
	\end{aligned}\end{equation} are satisfied with some positive constant $\bs{c}_L$ depending only on the Lipschitz character of $D_m$'s.
\end{lemma}
\begin{proof}
With the aid of the Green formula, being $\curl^2\Har{2,A}=0$ in $D_m,$
\begin{align*}
\norm*{\curl\Har{2,A}_m}_{\Lp{2}{D_m}}^2=&\int_{\p D_m}\nu\times\Har{2,A}_m \cdot \DTr{\curl \Har{2,A}_m}{-}ds,\\
\leq&\norm{\nu\times\Har{2,A}_m}_\Lp{2}{\p D_m}\norm{\nu\times\DTr{\curl\Har{2,A}_m}{-}}_\Lp{2}{\p D_m},\\
\leq&\bs{c}_L\rad\norm{A_m}_\LDivsp{ D_m} \bs{c}_L\norm{\nu\cdot\DTr{\curl\Har{2,A}_m}{-}}_\Lp{2}{\p D_m}\leq \bs{c}^{+}_L \rad\norm{A_m}^2_\LDivsp{ D_m} .
\end{align*}
	Again, using Green formula, we have \begin{align*}
	\int_{D_m} \abs{\nabla\Har{1,A}_m}^2dv=&\int_{\p D_m}\Har{1,A}_m  \NTr{\nabla \Har{1,A}_m}{-}ds\\
	&\leq\norm{\Har{1,A}_m}_{\Lp{2}{\p D_m}}~\norm{\NTr{\nabla \Har{1,A}_m}{-}}_{\Lp{2}{\p D_m}},\\
	&\leq\bs{c}_L\rad\norm{A_m}_{\LDivsp{D_m}} ~\bs{c}_L\norm{\TTr{\nabla \Har{1,A}_m}{-}}_{\Lp{2}{\p D_m}}\leq \bs{c}^{+}_L\rad\norm{A_m}^2_{\LDivsp{D_m}}.
	\end{align*} 
	
\end{proof}

\begin{lemma}\label{Estim-Compacte-Reminder}
For every $k\in\R^+$, the operator $\bs{M}_{\p D}^{k}-\bs{M}_{\p D}^{i}$ is compact and we have, under the condition on $c_r$,  
\begin{equation} \label{Estimates-for-Mk-M0}
	\norm*{\Bigl(\bs{M}_{\p D}^{k}-\bs{M}_{\p D}^{i}\Bigr)A}_{\LDivsp{D}}\leq \Bigl(\frac{1}{16\pi^2\bs{c}_{L,1}^2}+\bs{c}_{L,k}\rad^2\Bigr)^\frac{1}{2} {\norm*{A}_{\LDivsp{D}}}.
	\end{equation}
\end{lemma}
\begin{proof}
We have, due to the representation \eqref{Equivalent-Representation}
\begin{equation}\label{Estimate-for-Mk-M0-1st-step}
\begin{aligned}
\norm*{\Bigl(\bs{M}_{\p D}^{k}-\bs{M}_{\p D}^{i}\Bigr)A}^2_{\LDivsp{D}}
\leq& \Biggl[\Bigl(\sum_{m=1}^{\nbs} \norm*{\Bigl(\MDLP{k}{m}-\MDLP{i}{m}\Bigr)A_m}^2\Bigr)^\frac{1}{2}\\
&+\Bigl(\sum_{m=1}^{\nbs}\norm*{\sum_{\stl{j\geq 1}{j\neq m} }^{\nbs}(\MDLPmj{k}-\MDLPmj{i})A_j}^2\Bigr)^\frac{1}{2}\Biggr]^2.
	\end{aligned}
	\end{equation}
Using \eqref{Nabal-Phi_k-Phi_i}, we obtain  
\begin{align}\label{Estimates1-Diagpart-MDLP}
\abs*{[\MDLP{k}{m}-\MDLP{i}{m}](A)}\leq & {\bs{c}_{L,k}}\rad \norm*{A_m}_{\LDivsp{D_m}}.\end{align}
Due to  \eqref{DIV-MDk}, \eqref{Phi_k-Phi_i} for $\alpha=1$ and \eqref{Nabal-Phi_k-Phi_i}, we get
\begin{align*}
\abs*{[\Div(\MDLP{k}{m}-\MDLP{i}{m})](A)}\leq & \bs{c}_{L} \Bigl({(c_k+\bs{c}_{L})}\Bigr)\rad\norm*{A}_{\LDivsp{D_m}}+(1+k^2)\abs*{[\nu\cdot \SLP{0}{D_i}](A)}.
\end{align*}  
 With $$C^1_{L,k}:=(1+k^2)\norm*{[\SDO{0}{m}]}_{\mathcal{L}(\Lp{2}{\p \bsmc{D}_m})},$$ we get
\begin{equation}\label{Estimates2-Diagpart-MDLP}
\norm*{[\MDLP{k}{m}-\MDLP{i}{m}](A)}^2_{\LDivsp{D_m}}\leq \bs{c}_{L,k}\rad^4 \norm*{A_m}_{\LDivsp{D_m}}^2+C^1_{L,k} \rad^2\norm*{A_m}_{\LDivsp{D_m}}^2.
\end{equation} 
Summing over $m$
\begin{equation}\label{Miik-Mii0-L2DIV-estimate}
\sum_{m=1}^{\nbs}\norm*{[\MDLP{k}{m}-\MDLP{i}{m}](A)}^2_{\LDivsp{D_m}}\leq \bs{c}_{L,k}\rad^2 \sum_{m=1}^{\nbs}\norm*{A_m}_{\LDivsp{D_m}}^2.
\end{equation}\medskip
For $x\in \p D_m$, we have, with the notation of \Cref{Elec-Magn-Lip-sch-Comp-and-DefPosi} \begin{equation}\label{Boundary-To-Volume}
\begin{aligned}
\MDLPmj{k}(A)(x)=&-\nu\times\MSVP{D/D_m}{k,{\Id}}(\nabla \Har{1,A}_j+\Har{2,A}_j)+\nu\times\MCVP{D/D_m}{k,\Id}(\curl\Har{2,A}_j),
\end{aligned}
\end{equation} and
\begin{equation}\label{Surface-Div-Boundary-To-Volume}
\begin{aligned}
\div\MDLPmj{k}(A)(x)=&-\nu\cdot\MSVP{D/D_m}{k,\Id}(\curl\Har{2,A}_j)+k^2\nu\cdot\MCVP{D/D_m}{k,{\Id}}(\nabla \Har{1,A}_j+\Har{2,A}_j).
\end{aligned}
\end{equation}

Indeed, 
\begin{align*}
\MDLPmj{k}(A)(x)=&\nu\times\curl \int_{\p D_j}\Phi_k(x,y) (\nu\times\nabla\Har{1,A}_j+\nu\times\Har{2,A}_j)(y)ds_y,\\
=&\nu\times\curl \int_{ D_j} \curl_y \Bigl(\Phi_k(x,y)(\nabla\Har{1,A}_m+\Har{2,A}_m)(y)\Bigr)ds_y,\\
=&-\nu\times\Bigl(\curl^2 \int_{D_j} \Phi_k(x,y)(\nabla \Har{1,A}_j+\Har{2,A}_j)(y)ds_y-\curl \int_{D_j}\Phi_k(x,y)\curl\Har{2,A}_j(y)ds_y\Bigr).	
\end{align*} taking the surface divergence clear the second equation.  
Now, from \eqref{Boundary-To-Volume} similarly to what was done to get \eqref{Cal-Estm-MSoper-Ni-Nk}, \eqref{Cal-Estm-MCoper-Mk} we get both
\begin{align*}
\sum_{m=1}^\nbs\norm*{\sum_{\stl{j\geq 1}{j\neq m}}^\nbs \Bigl(\MDLPmj{k}-\MDLPmj{i}\Bigr)(A)}^2_{\Lp{2}{\p D_m}}
\leq& {c_0}\frac{(1+\abs{k})^2}{16\pi^2\bs{c}_r^6} \sum_{j=1}^m\frac{\norm*{\curl\Har{2,A}_j(y)}^2_{\Lp{2}{D_j}}}{\rad}\\ +2^3{c_0}\Bigl(&\frac{\abs{k}(\abs{k}+1)}{4\pi}\Bigr)^2 \Bigl[\frac{(\abs{k}+4)^2}{(1+2\bs{c}_r)^4\bs{c}^2_r}\Bigr] \sum_{j=1}^m\frac{\norm*{\nabla \Har{1,A}_j+\Har{2,A}_j}_{\Lp{2}{D_j}}^2}{\rad},\\
\sum_{m=1}^\nbs\norm*{\sum_{\stl{j\geq 1}{j\neq m}}^\nbs\Div \Bigl(\MDLPmj{k}-\MDLPmj{i}\Bigr)(A)}_{\Lp{2}{\p D_m}}
\leq&2^3{c_0}\frac{\abs{k}^2(1+\abs{k})^2}{\bs{c}_r^6}\sum_{j=1}^\nbs\frac{\norm*{\nabla \Har{1,A}_j+\Har{2,A}_j}_{\Lp{2}{D_j}}^2}{{\rad}}\\
+2^3{c_0}\Bigl(&\frac{\abs{k}(\abs{k}+1)}{4\pi}\Bigr)^2 \Bigl[\frac{(\abs{k}+4)^2}{(1+2\bs{c}_r)^4\bs{c}^2_r}\Bigr] \sum_{j=1}^m\frac{\norm*{\curl\Har{2,A}_j(y)}^2_{\Lp{2}{D_j}}}{\rad},
\end{align*} 
From \Cref{Estimated-Decom-OnL2D} we get 
\begin{equation}
\begin{aligned}
\sum_{m=1}^\nbs\norm*{\sum_{\stl{j\geq 1}{j\neq m}}^\nbs \Bigl(\MDLPmj{k}-\MDLPmj{i}\Bigr)(A)}^2_{\LDivsp{D_m}}&\\
\leq2^3{c_0}\bs{c}_L\Bigl(\frac{\abs{k}(\abs{k}+1)}{4\pi}&\Bigr)^2  \Bigl[\frac{1}{\bs{c}_r^6} +\frac{(\abs{k}+4)^2}{(1+2\bs{c}_r)^4\bs{c}^2_r}\Bigr] \sum_{j=1}^m{\norm*{A}_{\LDivsp{D_m}}^2},
\end{aligned}
\end{equation} which give, with $\bs{c}_r$ as stated in \Cref{Density-Estimate-And-Cond},
\begin{equation}
\begin{aligned}
\sum_{m=1}^\nbs\norm*{\sum_{\stl{j\geq 1}{j\neq m}}^\nbs \Bigl(\MDLPmj{k}-\MDLPmj{i}\Bigr)(A)}^2_{\LDivsp{D_m}}
\leq\frac{1}{(4\pi)^2\bs{c}_{L,1}^2} {\norm*{A}_{\LDivsp{D}}^2}.
\end{aligned}
\end{equation}

\end{proof}

\begin{proof} (\Cref{Density-Estimate-And-Cond}). 
Now, we end up the proof of Proposition \ref{Density-Estimate-And-Cond} as follows

\begin{equation*}
\begin{aligned}
\norm{[{I}/{2}+\bs{M}_{\p D}^{k}](A)}\geq&\norm{[{I}/{2}+\bs{M}_{\p D}^{i}](A)}-\norm{[\bs{M}_{\p D}^{i}-\bs{M}_{\p D}^{k}](A)},\\
\geq&\norm{[{I}/{2}+\bs{M}_{\p D}^{i}](A)}-\Bigl(\frac{1}{16\pi^2\bs{c}_{L,1}^2}+\bs{c}_{L,k}\rad^2\Bigr)^\frac{1}{2} \norm*{A}_{\LDivsp{D}}^2\\
\geq&\Bigl(\frac{1}{\bs{c}_{L,1}}-\Bigl(\frac{1}{16\pi^2\bs{c}_{L,1}^2}+\bs{c}_{L,k}\rad^2\Bigr)^\frac{1}{2}\Bigr) \norm*{A}_{\LDivsp{D}}^2.
\end{aligned}
\end{equation*} Then for $\bs{c}_{L,k}\rad\leq\frac{1}{3\bs{c}_{L,1}}$
\begin{equation*}
\begin{aligned}
\norm{\nu\times E^\n}_{\LDivsp{D}}^2=\norm{[{I}/{2}+\bs{M}_{\p D}^{k}](A)}
\geq&\frac{1}{3\bs{c}_{L,1}} \norm*{A}_{\LDivsp{D}}^2.
\end{aligned}
\end{equation*} 

\end{proof}

\subsection{Fields approximation and the linear algebraic systems}\label{Section-first-approximation}
Based on the representation (\ref{Electric-Field-representation}), the expression of the  far field pattern is given by
\begin{equation}\label{Far-Feld-Perfect-CondCase}
E^\infty(\hat{x})=\frac{ik}{4\pi}\hat{x}\times\int_{\p D} A(y)e^{-ik\hat{x}.y}ds_y,
\end{equation}
where $\hat{x}=(x/\abs{x}) \in\mathbb{S}^2$ and $A$ is the solution of the (\ref{Equivalent-Representation}). As it was done in \cite{bouzekri2019foldy}, 
let us consider $(\psi_l^m)_{l=1}^3$ as the solution of 
\begin{align}\label{Def-psi}
[-{I}/{2}+\MDLP{0}{i}](\nu\times \psi_m^l)=-\nu\times V_l,
\end{align}  with $$V_1^*=(0,0,(x-\z_i)\cdot e_2),~V_2^*=((x-\z_i)\cdot e_3,0,0),~\mbox{ and } V_3^*=(0,(x-\z_i)\cdot e_1,0).$$  
Taking the surface divergence of \eqref{Def-psi} gives  \begin{align}\label{Divergence-of-psil}
[{I}/{2}+(\DLP{0}{D_m})^*](\nu\cdot\curl\psi_m^l)=-\nu_i^l.
\end{align}  
Notice that
\begin{equation}\label{PsiTrace-to-VirtM-Tensor}
\begin{aligned}
	\int_{\p D_m}\nu\times\psi_m^l ds&=-\int_{\p D_m}(y-\z_m)\nu\cdot\curl\psi_m^l(y)ds_y\\
	&=-\int_{\p D_m}[{I}/{2}+(\DLP{0}{D_m})]^{-1}(y-\z_m)~[{I}/{2}+(\DLP{0}{D_m})^*]\nu\cdot\curl\psi_m^l(y)ds_y\\
	&=\int_{\p D_m}[{I}/{2}+(\DLP{0}{D_m})]^{-1}(y-\z_m)~\nu_i^lds_y.
\end{aligned}	
\end{equation}
As $\curl V_l=e_l$, it holds that \begin{equation}\label{Total-Charge-curlFreepart}
\int_{\p D_m}\nu^l u ds=\int_{\p D_m} \nu^l\cdot\curl V ~ u ds=-\int_{\p D_m} \nu\times V_l \cdot \nabla u ds=\int_{\p D_m}  V_l \cdot \nu\times\nabla u ds
\end{equation} for any scalar function $u \in \Hs{1}{\p D_m}.$ Finally, let $\phi_m$ be the solution to the following integral equation \begin{equation}
[-\frac{1}{2} I+\DLP{0}{D_m}](\phi_m)(x)=(x-\z_i). \label{Definition-Of-phi}
\end{equation} 

\begin{lemma} Due to the scale invariance of both the double-layer and the Maxwell-dipole operators, the following estimates hold
	\begin{align}\label{psi-Estimation}
	\norm*{\nu\times\psi_m^l}_{\Lp{2}{\p D_i}}
	\leq \bs{c}_{L} \rad^2,~\,&\norm*{\nu\times\psi_m^l}_{\LDivsp{Di}}
	\leq \bs{c}_{L}\rad,~\,
	\mbox{ and }~ \norm*{\phi_m}_{\Lp{2}{\p D_m}}\leq \bs{c}_{L} \rad^2.
	\end{align}
\end{lemma} 
\begin{proposition}\label{Field-Approximation-andlinsys}
For $\Im k=0$, the far field pattern can be approximated by
\begin{align}\label{Far-Field-Approximation}
E^\infty(\hat{x})=\frac{ik}{4\pi}
\sum_{m=1}^{\nbs}e^{-ik\hat{x}.\z_m}&\hat{x}\times 
\left\{\Qtot{1}{m}-ik\hat{x}\times \Qtot{2}{m}\right\} +O\Bigl(\frac{\abs{k}^3}{\bs{c}_r}~ \rad\Bigr).
\end{align}
The elements $(\Qtot{1}{m})_{i=1}^{\nbs}$ and $(\Qtot{2}{m})_{i=1}^{\nbs}$ are solutions of the following linear algebraic system 
\begin{align}
\begin{aligned}\label{First-LinSys-Approx-PerfCond}
\Qtot{2}{m}=-\Polt{m} &\sum_{\stl{j\geq 1}{j\neq m} }^{\nbs} \Bigl(\Pi_k(\z_m,\z_j)\Qtot{2}{j}-k^2\nabla\Phi_k(\z_m,\z_j)\times\Qtot{1}{j}\Bigr) -\Polt{m}\curl  E^\n(\z_i),\\
&+ Er_m\bigl(\curl\Har{2,A},\nabla \Har{1,A}+\Har{2,A}\bigr)+O\Bigr((\rad+1)\abs{k}^2\rad^3\norm*{A_m}_{\LDivsp{D_m}}+\abs{k}\rad\Bigl),
\end{aligned}\\
\begin{aligned}\label{Second-LinSys-Approx-PerfCond}
\Qtot{1}{m}=\Vmt{m}&\sum_{\stl{j\geq 1}{j\neq m} }^{\nbs}\Bigl(-\nabla\Phi_k(\z_m,\z_j)\times \Qtot{2}{j}+\Pi_k(\z_m,\z_j)~\Qtot{1}{j}\Bigr)-\Vmt{m}E^\n(\z_i) \\
&+Er_m\bigl((\nabla \Har{1,A}+\Har{2,A}),\curl\Har{2,A}\bigr)+O\Bigr((\rad+1)\abs{k}^2\rad^3\norm*{A_m}_{\LDivsp{D_m}}+\rad\Bigl),.
\end{aligned}                                                        
\end{align}  
\end{proposition} 
\begin{proof} 
\begin{itemize}
\item ({\em The far field Approximation}) Starting from the expression \eqref{Far-Feld-Perfect-CondCase} with representation of the density as in \eqref{Density-Decomposition-notation} we get with a simple integration by part 
\begin{equation}\label{Far-Field-PerfCond-VolIntegral}
\begin{aligned}	
E^\infty(\wh{x})=\frac{ik}{4\pi}\sum_{m=1}^\nbs\Biggl(&\hat{x}\times\int_{ D_m} e^{-ik\hat{x}.y}\curl\Har{2,A}_m(y)ds_y\\
&+{ik\hat{x}}\times \int_{ D_m}  e^{-ik\hat{x}.y}\Bigl(\nabla\Har{1,A}_m(y)+\Har{2,A}_m(y)\Bigr)\times \hat{x}ds_y\Biggr).
\end{aligned}
\end{equation} 
This representation \eqref{Far-Field-PerfCond-VolIntegral} is the analogous of \eqref{Far-Field-Anisotropic-Case}, stated for the transmission problem, 
 in such a way that, the far field approximation is done in similar way, with the appropriate changes that involve \eqref{Estimated-Decom-OnL2D} (i.e~ the appearance of the constant $\bs{c}_L$). 
\item ({\em Derivation of the linear system}) 

For \eqref{First-LinSys-Approx-PerfCond}, multiplying \eqref{Equivalent-Representation} by $\nabla\phi_m$ and integrating over $D_m$, we obtain
\begin{equation}\label{Approximation-phi-interpart}
\begin{aligned}
\int_{\p D_m}\nabla\phi_m\cdot [{I}/{2}+\MDLP{k}{m}]&(A_m)~ds\\
&=	-\sum_{\stackrel{j\geq 1}{j\neq m}}\int_{\p D_m}\nabla \phi_m\cdot [\MDLPmj{k}](A_j)ds+\int_{\p D_m}\nabla \phi_m\cdot\nu\times E^\n ds. 
\end{aligned}
\end{equation} Integrating by part and considering \eqref{Surface-Div-Boundary-To-Volume},
we get \begin{equation}\label{ApproxDeriv-Line-SysPerf-Cond}
\begin{aligned}
\int_{\p D_m}\nabla\phi_m\cdot &[{I}/{2}+\MDLP{k}{m}](A_m)~ds\\ =	&\sum_{\stackrel{j\geq 1}{j\neq m}}\int_{\p D_m}\Bigl(-\phi_m \nu\cdot\MSVP{D_j}{k,\Id}(\curl\Har{2,A}_j)+k^2\nu\cdot\MCVP{D_j}{k,{\Id}}(\nabla \Har{1,A}_j+\Har{2,A}_j)\Bigr)ds\\&+\int_{\p D_m}\nabla \phi_m\cdot\nu\times E^\n ds.
\end{aligned}\end{equation} 
We have
\begin{equation}\label{ApproxDeriv-Line-SysPerf-Cond-First-Member}
\int_{\p D_m}\nabla\phi_m\cdot [{I}/{2}+\MDLP{k}{m}](A_m)~ds =\Qtot{2}{m}+O\Bigl(\abs{k}^2\rad^3\norm*{A_m}_{\LDivsp{D_m}}\Bigr).
\end{equation}
Indeed, using the relation \eqref{Def-Surface-Div}, we have
\begin{equation*}
\begin{aligned}
\int_{\p D_m} \phi_m  \Div &[{I}/{2}+\MDLP{k}{m}]A~ds\\
&=\int_{\p D_m}\phi_m\biggl([{I}/{2}-(\DLP{k}{D_i})^*]\Div A-k^2\nu_{x_i}\cdot [\SLP{k}{D_i}]A\biggr)~ds.\\
&=\int_{\p D_m}\phi_m\biggl([{I}/{2}-(\DLP{0}{D_m})^*+(\DLP{0}{D_m}-\DLP{k}{D_i})^*]\Div A +k^2\nu\cdot [\SLP{k}{D_i}]A\biggr)~ds,
\end{aligned} 
\end{equation*} Hence, with the definition \eqref{Definition-Of-phi}, we get 
\begin{equation} 
\begin{aligned}\int_{\p D_m}\phi_m[{I}/{2}-(\DLP{0}{D_m})^*]\Div a~ds=\int_{\p D_m}[{I}/{2}-\DLP{0}{D_m}]\phi_m\Div a~ds=&-\int_{\p D_m}(x-\z_i)\Div A(x)~ds_x,\\   =&\int_{\p D_m}\curl\Har{2,A}~dv\end{aligned}\end{equation} and the left-hand side of \eqref{ApproxDeriv-Line-SysPerf-Cond} ends up to be
\begin{align}
\Qtot{2}{m}+O((\rad+1)\abs{k}^2\rad^3\norm*{A_m}_{\LDivsp{D_m}}).
\end{align}
As already done in \eqref{Lip-Schwin-LinSys-Scnd-Mem}, with a first order approximation, the second member of \eqref{ApproxDeriv-Line-SysPerf-Cond} can be approximated by  
\begin{equation}\label{ApproxDeriv-Line-SysPerf-Cond-Second-Member}
\begin{aligned}
-\Bigl[\int_{\p D_m}\phi_m\otimes \nu ds\Bigr]~&\sum_{\stackrel{j\geq 1}{j\neq m}}^\nbs\Biggl(\Pi_k(\z_m,\z_j)\int_{D_j}\curl\Har{2,A}_jdv-k^2\nabla\Phi_k(\z_m,\z_j)\times\int_{D_j}(\nabla \Har{1,A}_j+\Har{2,A}_j)dv\Biggr)\\&+Er_m\bigl(\curl\Har{2,A},\nabla \Har{1,A}+\Har{2,A}\bigr)
\end{aligned}\end{equation} where \begin{align*}
\nabla \Har{1,A}+\Har{2,A}=&\sum_{m=1}^{\nbs}(\nabla \Har{1,A}_j+\Har{2,A}_j)\ind{D_m}(x)\end{align*} and 
\begin{align*}
\curl\Har{2,A}=&\sum_{m=1}^{\nbs}\curl\Har{2,A}_j\ind{D_m}(x).
\end{align*}
Replacing both \eqref{ApproxDeriv-Line-SysPerf-Cond-First-Member} and \eqref{ApproxDeriv-Line-SysPerf-Cond-Second-Member} in \eqref{ApproxDeriv-Line-SysPerf-Cond} gives the result;
\begin{equation}\begin{aligned}
\Qtot{2}{m}=-\Polt{m}	&\Biggl[\sum_{\stackrel{j\geq 1}{j\neq m}}^\nbs\Biggl(\Pi_k(\z_m,\z_j)\Qtot{2}{j}-k^2\nabla\Phi_k(\z_m,\z_j)\times\Qtot{1}{j}\Biggr)+\curl E^\n(\z_m)\Biggr] \\+&Er_m\bigl(\curl\Har{2,A},\nabla \Har{1,A}+\Har{2,A}\bigr)+O\Bigr((\rad+1)\abs{k}^2\rad^3\norm*{A_m}_{\LDivsp{D_m}}+\abs{k}\rad\Bigl).
\end{aligned}\end{equation}

Concerning \eqref{Second-LinSys-Approx-PerfCond} we have, for $[\Psi_m]=(\psi_m^l)_{(l=1,2,3)}$
as described in \eqref{Def-psi},
\begin{equation}\label{Approximation-Psi}
\int_{\p D_m}[\Psi]_m\cdot [{I}/{2}+\MDLP{k}{m}](A_m)~ds =	-\sum_{\stackrel{j\geq 1}{j\neq m}}\int_{\p D_m}[\Psi]_m\cdot [\MDLPmj{k}](A_j)ds+\int_{\p D_m} [\Psi]_m\cdot\nu\times E^\n ds. 
\end{equation}  The left-hand-side is equal to
\begin{equation}\label{Second-LinSys-LeftHandSide-PerfCond}
\Qtot{1}{m}+  O(\rad^3\norm{A}_{\LDivsp{D_m}}).
\end{equation}
Indeed,
\begin{equation}\label{Approximation-psi-diagpart} 
\int_{\p D_m} \psi_m^l\cdot[{I}/{2}+\MDLP{k}{m}]A_m~ds=O(\rad^3 \norm*{A_m}_{\LDivsp{D_m}})+\int_{\p D_m} \psi_m^l\cdot[{I}/{2}+\MDLP{0}{i}](A_m^{1})~ds
\end{equation} and, in view of the decomposition \Cref{Estimated-Decom-OnL2D}, we have,  due to the scale invariance of $[{I}/{2}+\MDLP{0}{m}]$ 
(see \cite{bouzekri2019foldy}), the estimates \eqref{psi-Estimation}, those of \Cref{Decomposition-LDiv-Space},  and similar estimate for \eqref{Estimates2-Diagpart-MDLP} with $\alpha=0$, that
\begin{align*}
\biggl|\int_{\p D_m}~ \psi_m^l\cdot[{I}/{2}+\MDLP{k}{m}]A_m^{2}~ds\Biggr|\leq &
\norm*{\psi_m^l}\left(\norm*{[{I}/{2}+\MDLP{0}{m}]A_m^{2}}+\norm*{[\MDLP{k}{m}-\MDLP{0}{m}]A_m^{2}}\right),\\
\leq & \norm*{\psi_m^l}_{{\Lp{2}{\p D_m}}} \Bigl(\bs{c}_L+\frac{\abs{k}^2(\abs{\p \bsmc{D}})\rad^2}{4\pi}\Bigr)\norm*{A_m^{2}}_{{\Lp{2}{\p D_m}}},\\
\leq & \Bigl(\bs{c}_l
+\frac{\abs{k}^2(\abs{\p \bsmc{D}})\rad^2}{4\pi}
\Bigr)\rad^3\norm*{A_m}_{{\LDivsp{D_m}}}.
\end{align*}  
Further with \eqref{Divergence-of-psil}, \eqref{nu-time-nabla-interversion} and the representation \cref{Decomposition-LDiv-Space} we have, 
for the second term of the left-hand member of \eqref{Approximation-psi-diagpart}   
\begin{equation}\begin{aligned}
\int_{\p D_m} \psi_m^l\cdot\nu\times\nabla [{I}/{2}+\DLP{0}{D_m}](u_m^{A})ds =&\int_{\p D_m} -[{I}/{2}+(\DLP{0}{D_m})^*](\nu\cdot\curl\psi_m^l)~u_m^{A}ds.\\ =&\int_{\p D_m}\nu_lu_m^Ads.
	\end{aligned} \end{equation} 
Due to \eqref{Total-Charge-curlFreepart} and \Cref{Decomposition-LDiv-Space}, we have
\begin{equation}\begin{aligned}
\int_{\p D_m}\nu_lu_m^Ads=&\int_{\p D_m}V_l\cdot \nu\times \Bigl(\nabla u_m^A+\Har{1,A}\Bigr)ds+O(\rad^3\norm{A}_{\LDivsp{D_m}}),
\end{aligned} \end{equation}
hence follows, integrating by part for the second step,  
\begin{equation}\begin{aligned}
\int_{\p D_m}\nu_lu_m^Ads=\int_{\p D_m}V_l\cdot \nu\times \nabla u_m^Ads=&\int_{\p D_m}V_l\cdot \nu\times \Bigl(\nabla u_m^A+\Har{1,A}\Bigr)ds+O(\rad^3\norm{A}_{\LDivsp{D_m}}),\\
=&\int_{D_m} \nabla (\Har{2,A}+\Har{1,A})\cdot e_l~dv-\int_{D_m}V_l\cdot\curl\Har{1,A}dv\\  &+O(\rad^3\norm{A}_{\LDivsp{D_m}}).
\end{aligned} \end{equation}  

Then \eqref{Esimation-Har-A2} gives 	
\begin{equation}
\begin{aligned}
\int_{\p D_m} [\Psi]_m\nu\times\nabla [{I}/{2}+\DLP{0}{D_m}](u_i^{A})ds=&\int_{\p D_m}\nu u_m^Ads=
\Qtot{1}{m}+ O(\rad^3\norm{A}_{\LDivsp{D_m}}).
\end{aligned} \end{equation}
	
For the first term of the right-hand-side of \eqref{Approximation-Psi}, we have, with \eqref{Boundary-To-Volume} and a first order approximation, 
\begin{equation}\label{Second-Part-Lemma-Appr}   
\begin{aligned}
\sum_{\stackrel{j\geq 1}{j\neq m}}\int_{\p D_m} \psi_m^l\cdot [\MDLPmj{k}]A_jds&\\
=-\sum_{\stackrel{j\geq 1}{j\neq m}}\int_{\p D_m}\psi_m^l\cdot\Bigl(\nu\times&\MSVP{D_j}{k,{\Id}}(\nabla \Har{1,A}_j+\Har{2,A}_j)-\nu\times\MCVP{D_j}{k,\Id}(\curl\Har{2,A}_j)\Bigr)\\
=-\sum_{\stackrel{j\geq 1}{j\neq m}}\int_{\p D_m}\psi_m^l\cdot\Biggl(\nu\times&\Bigl(\Pi_k(\z_m,\z_j)\int_{D_j}(\nabla \Har{1,A}_j+\Har{2,A}_j)dv\Bigr)\\
&-\nu\times\Bigl(\nabla\Phi_k(\z_m,\z_j)\times\int_{D_j}\curl\Har{2,A}_jdv\Bigr)\Biggr)\\
+Er_m\bigl((\nabla \Har{1,A}+\Har{2,A}&),\curl\Har{2,A}\bigr).
\end{aligned}
\end{equation} We rewrite \eqref{Second-Part-Lemma-Appr} as follows 
\begin{equation}   
\begin{aligned}
\sum_{\stackrel{j\geq 1}{j\neq m}}\int_{\p D_m} \psi_m^l\cdot [\MDLPmj{k}]A_jds
=\sum_{\stackrel{j\geq 1}{j\neq m}}&\int_{\p D_m}\nu\times\psi_m^l\cdot\Biggl(\Pi_k(\z_m,\z_j)\Qtot{1}{j}-\nabla\Phi_k(\z_m,\z_j)\times\Qtot{2}{j}\Biggr)\\
&+Er_m\bigl((\nabla \Har{1,A}+\Har{2,A}),\curl\Har{2,A}\bigr),
\end{aligned}
\end{equation} or, more precisely with \eqref{PsiTrace-to-VirtM-Tensor} in mind and $l=1,2,3$ gives, according to \eqref{Polarization&virtualmass-Tensor}, 
\begin{equation} \label{Second-LinSys-righttHandSide-PerfCond}  
\begin{aligned}
\sum_{\stackrel{j\geq 1}{j\neq m}}\int_{\p D_m} [\Psi]_m\cdot [\MDLPmj{k}]A_jds
=\Vmt{m}\sum_{\stackrel{j\geq 1}{j\neq m}}& \Biggl(\Pi_k(\z_m,\z_j)\Qtot{1}{j}-k^2\nabla\Phi_k(\z_m,\z_j)\times\Qtot{2}{j}\Biggr)\\
&+Er_m\bigl((\nabla \Har{1,A}+\Har{2,A}),\curl\Har{2,A}\bigr).
\end{aligned}
\end{equation} Assembling both \eqref{Second-LinSys-righttHandSide-PerfCond} and 
\eqref{Second-LinSys-LeftHandSide-PerfCond} in \eqref{Approximation-Psi} gives us \eqref{Second-LinSys-Approx-PerfCond}. \end{itemize}
\end{proof}

As in the proof of \Cref{Inversion-Lin-Sys-Anisotropic-Case}, with $\mu^+$ and $\mu^-$ as defined in \eqref{mudefinition}, we have the following proposition.
\begin{proposition}\label{Inversion-Lin-Sys-PerCond-Case}
Under the condition that \begin{equation}\frac{\dist}{\rad}=\bs{c}_r \geq{3\abs{k}}{\mu^+}.\end{equation} the following linear system is invertible
\begin{equation}
\begin{aligned}
\wh{\Qtot{2}{m}}=-\Polt{m} &\sum_{\stl{j\geq 1}{j\neq m} }^{\nbs} \Bigl(\Pi_k(\z_m,\z_j)\wh{\Qtot{2}{j}}-k^2\nabla\Phi_k(\z_m,\z_j)\times\wh{\Qtot{1}{j}}\Bigr) -\Polt{m}\curl  E^\n(\z_i),\\
\wh{\Qtot{1}{m}}=\Vmt{m}&\sum_{\stl{j\geq 1}{j\neq m} }^{\nbs}\Bigl(-\nabla\Phi_k(\z_m,\z_j)\times \wh{\Qtot{2}{j}}+\Pi_k(\z_m,\z_j)~\wh{\Qtot{1}{j}}\Bigr)-\Vmt{m}E^\n(\z_i).
\end{aligned}                                                        
\end{equation} and the solution satisfies the following estimates
	\begin{equation}\label{Estimate-Solutions-of-LinSys1}
	\begin{aligned}
	\left(\sum_{m=1}^{\nbs}\bigl(\abs*{\widehat{\Qtot{1}{m}}}^2\bigr)\right)^\frac{1}{2} \leq 
	\frac{9\mu^+\rad^{3}}{8}  \left(\sum_{m=1}^{\nbs}\bigl(\abs*{E(\z_i)}^2+\abs*{\curl E(\z_i)}^2\bigr)\right)^\frac{1}{2},
	\end{aligned}
	\end{equation} and
	\begin{equation}\label{Estimate-Solutions-of-LinSys2} 
	\begin{aligned}
	\left(\sum_{m=1}^{\nbs}\bigl(\abs*{\widehat{\Qtot{2}{m}}}^2\right)^\frac{1}{2} \leq 
	\frac{9\mu^+\rad^{3}}{8}  \left(\sum_{m=1}^{\nbs}\bigl(\abs*{E(\z_i)}^2+\abs*{\curl E(\z_i)}^2\bigr)\right)^\frac{1}{2}.
	\end{aligned}
	\end{equation}
\end{proposition}	
The proof of \Cref{Theorem-PerfCond-Case} is similar to the anisotropic transmission problem with 
$\bs{c}_{L,k,\mu}:=\bs{c}_{L,k}(\abs{k}+1)\frac{4{\mu^+}}{8\mu^-}\bs{c}_{L,k}(\abs{k}+1)$ and $\bs{c}_{L,k}$ describes the ratio of the largest Lipschitz constant that is involved in
\cref{Section-Perfect-Conductor} and the smallest one.  	

\appendix
\section{Appendix}
\subsection{Green function approximations}
A simple application of mean value theorem gives
\begin{equation}
\begin{aligned}
\bigl(\Phi_k(x,y)-\Phi_k(\z_m,y)\bigr)=&\int_{0}^{1}\nabla\Phi_k(tx+(1-t)\z_m,y)dt\circ(x-\z_m)= 
O\Bigl(\frac{1}{\dist_{mj}}\bigl(\frac{1}{\dist_{mj}}+\abs{k}\bigr)\rad\Bigr),\\
\nabla\bigl(\Phi_k(x,y)-\Phi_k(\z_m,y)\bigr)=&\int_{0}^{1}D^2\Phi_k(tx+(1-t)\z_m,y)dt\circ(x-\z_m)= 
O\Bigl(\frac{1}{\dist_{mj}}\bigl(\frac{1}{\dist_{mj}}+\abs{k}\bigr)^2\rad\Bigr),\\
\nabla\nabla\bigl(\Phi_k(x,y)-\Phi_k(\z_m,y)\bigr)=& \int_{0}^{1}D^3\Phi_k(tx+(1-t)\z_m,y)dt\circ(x-\z_m)= 
O\Bigl(\frac{1}{\dist_{mj}}\bigl(\frac{1}{\dist_{mj}}+\abs{k}\bigr)^3\rad\Bigr).
\end{aligned}\label{Green-Func-First-Order-Appr}\end{equation} whenever $y\in D_j$ and $ j\neq m.$ 
We will also need the following first order expansion of the Green's function\footnote{By $\tensor{x}{p}$ we mean the $p$-times repeated tensor product of $x$.} 
\begin{align}
&[\Phi_k-\Phi_{i\alpha }](x)=\frac{(ik-\alpha)}{4\pi}\int_{[0,1]}e^{\bigl((ik)t-(1-t)\alpha\bigr)\abs*{x}}~dt, \label{Phi_k-Phi_i}\\
&\nabla[\Phi_k-\Phi_{i\alpha}](x)= \Bigr[\frac{(ik-\alpha)}{4\pi}\int_{[0,1]}{e^{(ikt-(1-t)\alpha)\abs*{x}}}~\bigl(ikt+(1-t)\alpha\bigr)dt\Bigl]~\frac{x}{\abs*{x}},\label{Nabal-Phi_k-Phi_i}\\
&[\otimes\nabla]^2[\Phi_k-\Phi_{i\alpha}](x)=\int_{0}^1\frac{e^{(ikt-(1-t)\alpha)\abs*{x}} \bigl(ikt+(1-t)\alpha\bigr)}{4\pi~(ik-\alpha)^{-1} \abs*{x}}\Bigl[I+\bigl(ikt+(1-t)\alpha-\frac{1}{\abs*{x}} \Bigr) \frac{[\otimes x]^2}{\abs*{x}}\Bigr]dt.\label{NablaNabal-Phi_k-Phi_i}
\end{align}
\subsection{Counting lemma}
  \begin{lemma}\label{Counting-Oclusion}
	For any non negative function $g$ we have 
	\begin{equation}\label{Sum-Counting}
	\sum_{{i\geq1},{i\neq j}}^{\nbs} g(\dist_{mj})\leq 48\sum_{\stl{1\leq l\leq \nbs}{O\leq i\leq k\leq l}} g\Bigl(\bigl[\bigl(l^2+k^2+i^2\Bigr)^\frac{1}{2}(\bs{c}_r+1)-1\bigr]\rad\Bigr),
	\end{equation} and for any non negative sequence $(\alpha_m)_{m=1}^{\nbs}$
	\begin{equation}\label{Double-Sum-Counting}
	\sum_{m=1}^{\nbs} \Bigl(\sum_{\stl{j\geq 1}{j\neq m} }^{\nbs}\frac{\alpha_m}{\dist_{mj}^q}\Bigr)^2\leq \Bigl(\frac{{c_0}}{\dist^q}\sum_{l=1}^{\nbs^\frac{1}{3}}l^{2-q}\Bigr)^2 \sum_{m=1}^{\nbs}\alpha_m^2.
	\end{equation}
\end{lemma}
\begin{figure}[h]
	\definecolor{rvwvcq}{rgb}{0.08235,0.3961,0.7529}
	\definecolor{wrwrwr}{rgb}{0.3804,0.3804,0.3804}
	\definecolor{sqsqsq}{rgb}{0.1255,0.1255,0.1255}
	\begin{center}
		\fcolorbox{black}{white}{	
			\begin{tikzpicture}[line cap=round,line join=round,>=triangle 45,x=1cm,y=1cm,scale=0.5]
			\clip(-12,-6) rectangle (8.5,5.12);
			\fill[line width=0.4pt,fill=black,fill opacity=0.1] (-11.513,-1.1569) -- (-4.01,4.579) -- (8.212,4.591) -- (0.7169,-1.1787) -- cycle;
			\fill[line width=0.4pt,color=sqsqsq,fill=sqsqsq,fill opacity=0.1] (-7.222728835835313,-3.0713635077480737) -- (-3.4977,-0.018872) -- (3.83221,-0.035359) -- (0.158,-3.175) -- cycle;
			\draw [line width=0.8pt] (-11.51,-1.16)-- (-4.01,4.58);
			\draw [line width=0.8pt] (-4.01,4.59)-- (8.21,4.591);
			\draw [line width=0.8pt] (8.212,4.591)-- (0.717,-1.18);
			\draw [line width=0.8pt] (0.72,-1.179)-- (-11.513,-1.157);
			\draw [line width=0.4pt,dash pattern=on 2pt off 3pt,color=sqsqsq] (-7.22,-3.07)-- (-3.498,-0.019);
			\draw [line width=0.4pt,dash pattern=on 2pt off 3pt] (-3.498,-0.019)-- (3.83,-0.035);
			\draw [line width=0.8pt,color=sqsqsq] (3.832,-0.0353)-- (0.1584,-3.173);
			\draw [line width=0.8pt,color=sqsqsq] (0.15843,-3.174)-- (-7.223,-3.07136);
			\draw [line width=0.8pt,color=wrwrwr] (-4.9,-1.17)-- (-7.222728835835313,-3.0713635077480737);
			\draw [line width=0.8pt,dash pattern=on 2pt off 3pt,color=wrwrwr] (-1.648,1.7085)-- (-1.6105,-5.814130);
			\draw [line width=1pt,dotted,color=wrwrwr] (-6.58,0.27)-- (-3.523,0.2703);
			\draw [line width=1pt,dotted,color=wrwrwr] (-3.52,0.27)-- (-0.47,0.27);
			\draw [line width=1pt,dotted,color=wrwrwr] (3.28,3.15)-- (-0.4664,0.267);
			\draw [line width=1pt,dotted,color=wrwrwr] (3.28,3.15)-- (-2.8297,3.1458);
			\draw [line width=1pt,dotted,color=wrwrwr] (-2.81,3.15)-- (-6.58,0.2737);
			\draw [rotate around={-166.675:(-1.67,1.71)},line width=0.9pt,dotted] (-1.67,1.71) ellipse (8.193cm and 2.6645cm);
			\draw [line width=1pt,color=wrwrwr] (-1.611,-5.81)-- (-1.62,-3.15);
			\draw [line width=0.5pt,dash pattern=on 2pt off 3pt] (5.15621,4.59)-- (-1.65,1.71);
			\draw [line width=0.8pt,dash pattern=on 2pt off 3pt] (-11.5129,-1.1569)-- (-11.387,-6.553);
			\draw [line width=0.8pt,dash pattern=on 2pt off 3pt] (-7.223,-3.0714)-- (-8.565,-4.302);
			\draw [line width=0.4pt,dash pattern=on 2pt off 3pt] (-8.565,-4.302)-- (-11.4546,-4.302242);
			\draw [rotate around={132.137:(-1.65,1.71)},line width=0.9pt] (-1.65,1.70) ellipse (0.493cm and 0.6645cm);
			\begin{scriptsize}
			\draw [fill=black] (-11.5129,-1.156) circle (3pt);\draw [fill=black] (-4.00990,4.57877) circle (3pt);\draw [fill=black] (8.2115,4.591) circle (3pt);\draw [fill=black] (0.72,-1.19) circle (3pt);	\draw [fill=black] (-7.22,-3.0713) circle (3pt);\draw [fill=black] (-3.49,-0.018) circle (0.5pt);\draw [fill=black] (3.83,-0.035) circle (3pt);\draw [fill=black] (0.158436,-3.174837) circle (3pt);\draw [fill=black] (2.1008,4.58496) circle (3pt);\draw [fill=black] (5.1562,4.588061) circle (3pt); \draw[color=black] (5.2,4.0) node {$\mathlarger{\z_k}$};	\draw[color=black] (-5.6,4.0) node {$\mathlarger{SQ_{\mathlarger{2}}}$};\draw [fill=black] (-0.9545,4.5818) circle (3pt);\draw [fill=black] (-5.398013,-1.16780) circle (3pt);\draw [fill=black] (-2.3405,-1.173) circle (3pt);
			\draw [fill=black] (4.4642,1.7062) circle (3pt);\draw [fill=black] (6.337,3.14869) circle (3pt);\draw [fill=black] (-8.4554,-1.1623) circle (3pt);\draw [fill=black] (-7.7614,1.7109) circle (3pt);\draw [fill=black] (-5.89,3.1448) circle (3pt);\draw [fill=black] (-9.637,0.2770) circle (3pt);\draw [fill=black] (-3.5233,0.27038) circle (3pt);\draw [fill=black] (-0.46635,0.2670) circle (3pt);\draw [fill=black] (-6.5802,0.27369) circle (3pt);\draw [fill=black] (-1.649,1.709) circle (3pt);\draw[color=black] (-0.4,1.6) node {$\mathlarger{\z_{l+1}}$};\draw[color=black] (-4.4,2.8) node {$\mathlarger{SQ_{\mathlarger{1}}}$};\draw [fill=black] (0.2261,3.14677) circle (3pt);\draw [fill=black] (3.2820,3.147732) circle (3pt);\draw [fill=black] (-2.8297,3.145) circle (3pt);\draw [fill=black] (1.4078,1.7074) circle (3pt);\draw [fill=black] (-4.70,1.7098) circle (3pt);\draw [fill=black] (-1.610,-5.82) circle (3pt);\draw[color=black] (-1.241,-5.125) node {$\mathlarger{\z_i}$};\draw [fill=black] (-3.532,-3.123) circle (3pt);\draw [fill=black] (1.9953,-1.60509) circle (3pt);\draw [fill=black] (-5.3602,-1.54512) circle (3pt);\draw [fill=black] (-1.68,-1.5751) circle (3pt);\draw[color=black] (-1.17,-1.7) node {$\mathlarger{\z_{l}}$};\draw [fill=black] (2.5728,0.291551) circle (3pt);\draw[color=black] (1.3,4) node {$\mathlarger{d(\z_k,\z_{l+1})}$};\draw [fill=rvwvcq] (-11.387,-6.55) circle (0.5pt);\draw[color=black] (-10.3,-2.479) node {$\mathlarger{\dist+\rad}$};\draw [fill=black] (-8.565,-4.302) circle (0.5pt); \draw[color=black] (-6.625,-3.781) node {$\mathlarger{\dist+\rad}$};\draw [fill=black] (-11.455,-4.302) circle (0.5pt);\draw[color=black] (-10.178,-4.823) node {$\mathlarger{\dist+\rad}$};
			\end{scriptsize}
			\end{tikzpicture}}
		\caption[]{Disposition of the faces $F_l.$}
		\label{FaceF}
	\end{center}
\end{figure}
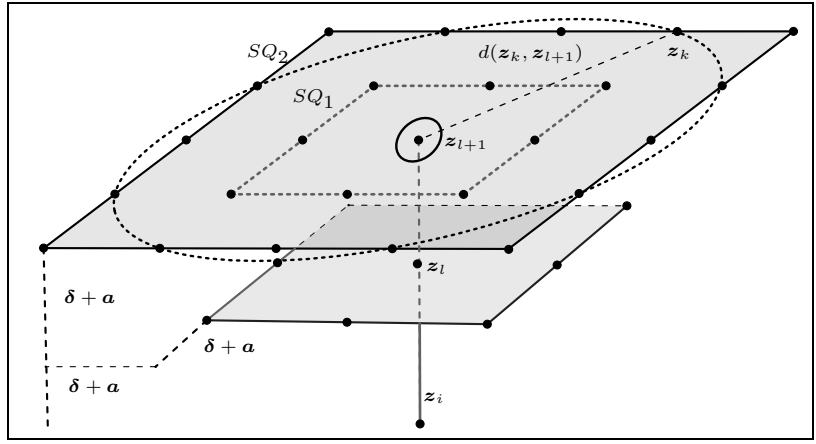
\begin{figure}
	\begin{center}	
		\fcolorbox{black}{white}{
			\begin{tikzpicture}[line cap=round,line join=round,>=triangle 45,x=1cm,y=1cm,scale=0.5]
			\clip(-9.95,-6.411) rectangle (10.,5);
			\fill[line width=0.4pt,fill=black,fill opacity=0.05] (4,-4) -- (4,4) -- (-4,4) -- (-4,-4) -- cycle;
			\draw [line width=0.3pt] (4,-4)-- (4,4);
			\draw [line width=0.3pt] (4,4)-- (-4,4);
			\draw [line width=0.3pt] (-4,4)-- (-4,-4);
			\draw [line width=0.3pt] (-4,-4)-- (4,-4);
			\draw [line width=0.3pt,dash pattern=on 2pt off 3pt] (0,0) circle (4.12cm);
			\draw [line width=0.3pt,dash pattern=on 2pt off 3pt] (0,0) circle (4.48cm);
			\draw [line width=0.3pt,dash pattern=on 2pt off 3pt] (0,0) circle (5cm);
			\draw [line width=0.3pt,dash pattern=on 2pt off 3pt] (0,0) circle (5.67cm);
			\draw [line width=0.5pt,dash pattern=on 2pt off 3pt] (0,0)-- (4,0);
			\draw [line width=0.5pt,dash pattern=on 2pt off 3pt] (0,0)-- (4,2);
			\begin{scriptsize}
			\draw [fill=black] (4,-4) circle (3pt); \draw [fill=black] (4,4) circle (3pt); \draw [fill=black] (-4,4) circle (3pt); \draw [fill=black] (-4,-4) circle (3pt); \draw[color=black] (-1.012,1.51) node {$\mathlarger{\bf{SQ_k}}$};
			\draw [fill=black] (-4,-3) circle (3pt); \draw [fill=black] (-4,-2) circle (3pt); \draw [fill=black] (-4,-1) circle (3pt); \draw [fill=black] (-4,0) circle (3pt); \draw [fill=black] (-4,1) circle (3pt); \draw [fill=black] (-4,2) circle (3pt); \draw [fill=black] (-4,3) circle (3pt); \draw [fill=black] (-3,4) circle (3pt); \draw [fill=black] (-2,4) circle (3pt); 
			\draw [fill=black] (-1,4) circle (3pt); \draw [fill=black] (3,4) circle (3pt); \draw [fill=black] (2,4) circle (3pt); \draw [fill=black] (1,4) circle (3pt); \draw [fill=black] (0,4) circle (3pt); \draw [fill=black] (4,3) circle (3pt); \draw [fill=black] (4,2) circle (3pt); \draw [fill=black] (4,0.9770838735084553) circle (3pt); \draw[color=black] (3.58,2.66) node {$\mathlarger{\z_q}$}; \draw [fill=black] (4,-1) circle (3pt);
			\draw [fill=black] (4,-2) circle (3pt);\draw [fill=black] (4,-3) circle (3pt);\draw [fill=black] (3,-4) circle (3pt);\draw [fill=black] (2,-4) circle (3pt);\draw [fill=black] (1,-4) circle (3pt);\draw [fill=black] (0,-4) circle (3pt);\draw [fill=black] (-1,-4) circle (3pt);\draw [fill=black] (-2,-4) circle (3pt);\draw [fill=black] (-3,-4) circle (3pt);
			\draw [fill=black] (0,0) circle (3pt);\draw[color=black] (-0.383,0.6196) node {$\mathlarger{\z_l}$};\draw [fill=black] (4,0) circle (3pt);
			\draw[color=black] (2.39,-0.5) node {$\mathlarger{k\times(\dist+\rad)}$};
			\draw[color=black] (1.8,1.8) node {$\mathlarger{d(\z_q,\z_l)}$};
			\end{scriptsize}
			\end{tikzpicture}}
		\caption[]{Counting on the square $SQ_k$.}
		\label{SquareSQ}
	\end{center}
\end{figure}
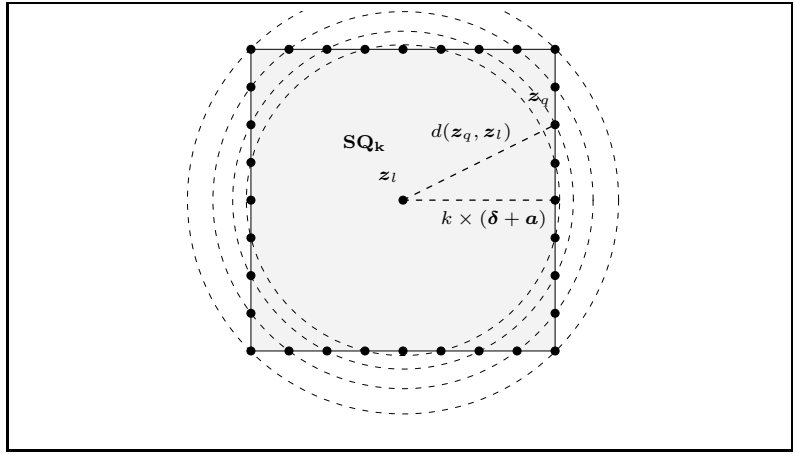
\begin{proof}
	From a given position $\z_i$ we split the space into equidistant cubes $(CU_l)$, centered at $\z_i$, such that each of its faces support some of the $(\z_j)_{\stackrel{j\geq 1}{j\neq m}}^{\nbs}$, and each of its faces, $(F_l)$ are distant from $F_{l\pm 1}$ with distance $\dist+\rad$, (see \cref{FaceF}).
	Obviously the distance from a point $\z_j\in F_l$ to $\z_i$ is 
	\begin{equation}\label{Conting-On-Fl}
	\begin{aligned}
	d(\z_i,\z_j)&=\sqrt{(d(\z_i,F_l)^2+d(\z_j,\z_l)^2)},\\ d(\z_i,F_l)&=l(\dist+\rad),
	\end{aligned} 
	\end{equation} $z_l$ being the orthogonal projection of $z_i$ on $F_l.$ 
	Repeating the same splitting on each faces $F_l$, we draw concentric squares $(SQ_k)_{k=1}^l$, centered at $\z_l$, and  
	there is  $4$ or $8$ location that are equidistant from a given square $SQ_k$ to $\z_l$  which correspond to the intersection of a circle with a square sharing the same center,
	similarly, for a point $z_p\in SQ_k$, we get, with $z_k$ standing for one of its orthogonal projection on $SQ_k,$   
	\begin{equation}\label{Conting-On-SQk}
	\begin{aligned}
	d(\z_p,\z_l)&=\sqrt{(d(\z_p,\z_k)^2+d(SQ_k,\z_l)^2)},\\ d(SQ_k,\z_l)&=k(\dist+\rad),
	\end{aligned} 
	\end{equation} further $$d(B_{\z_i}^{\rad}, B_{\z_j}^{\rad})=d(\z_i,\z_j)+\rad.$$
	So, for a non negative function $g$, it comes with \eqref{Conting-On-Fl} that
	\begin{align*}
	g\bigl(d(B_{\z_i}^{\rad}, B_{\z_j}^{\rad})\Bigr)=&\sum_{j(\neq i)=1}^{\nbs} g\bigl(d(\z_i,\z_j)-\rad\bigr)=
	\sum_{l=1}^{\nbs} 6\sum_{\z_j\in F_l}g\bigl(d(\z_j,\z_i)-\rad\bigr),\\
	=&6\sum_{l=1}^{\nbs}\sum_{z_j\in F_l}g\Bigl(\bigl(d(F_l,\z_i)^2+d(\z_l,\z_j)^2\Bigr)^\frac{1}{2}-\rad\Bigr)\\
	\leq& 6\sum_{l=1}^{\nbs}\sum_{k=0}^l8\sum_{z_p\in SQ_k} g\Bigl(\bigl(l^2(\dist+\rad)^2+d(\z_l,\z_p)^2\Bigr)^\frac{1}{2}-\rad\Bigr),
	\end{align*} and with \eqref{Conting-On-SQk}
	\begin{align*}
	\sum_{j(\neq i)=1}^{\nbs} g\bigl(d(B_{\z_i}^{\rad}, B_{\z_j}^{\rad})\Bigr)\leq &6\times 8\sum_{l=1}^{\nbs}\sum_{k=0}^l\sum_{p=0}^k g\Bigl(\bigl(l^2(\dist+\rad)^2+d(\z_l,SQ_k)^2+d(\z_p,\z_k)^2\Bigr)^\frac{1}{2}-\rad\Bigr),\\
	\leq & 6\times 8\sum_{l\geq 1}^{\nbs}\sum_{k=0}^l\sum_{p=0}^k g\Bigl(\bigl(l^2(\dist+\rad)^2+k^2(\dist+\rad)^2+i^2(\dist+\rad)^2\Bigr)^\frac{1}{2}-\rad\Bigr),
	\end{align*} which guaranties that
	\begin{align*}
	\sum_{j(\neq i)=1}^{\nbs} g\bigl(d(B_{\z_i}^{\rad}, B_{\z_j}^{\rad})\Bigr)
	\leq &6\times 8\sum_{\stl{1\leq l\leq \nbs}{O\leq i\leq k\leq l}} g\Bigl(\bigl[\bigl(l^2+k^2+i^2\Bigr)^\frac{1}{2}(\bs{c}_r+1)-1\bigr]\rad\Bigr).
	\end{align*} 
	For \eqref{Double-Sum-Counting}, using H\"older's inequality for the inner sum,  and considering \eqref{Sum-Counting}, we obtain 
	\begin{align*}
	\sum_{m=1}^{\nbs} \Bigl(\sum_{\stl{j\geq1}{j\neq m}}^{\nbs}\frac{\alpha_j}{\dist_{mj}^q}\Bigr)^2
	\leq& \sum_{m=1}^{\nbs} \Biggl( \Biggl[\sum_{\stl{j\geq 1}{j\neq m}}^{\nbs}\Bigl(\frac{\alpha_j}{\dist_{mj}^\frac{q}{2}}\Bigr)^2\Biggr]^\frac{1}{2}\Biggl[\sum_{\stl{j\geq 1}{j\neq m}}^{\nbs}\Bigl(\frac{1}{\dist_{mj}^\frac{q}{2}}\Bigr)^2\Biggr]^\frac{1}{2}\Biggr)^2,\\
	\leq&  \sum_{m=1}^{\nbs} \Biggl( \sum_{\stl{j\geq 1}{j\neq m}}^{\nbs}\Bigl(\frac{\alpha_j}{\dist_{mj}^q}\Bigr)~\sum_{\stl{j\geq 1}{j\neq m}}^{\nbs}\frac{1}{\dist_{mj}^q}\Biggr),\\
	\leq& \frac{1}{\dist^q}\sum_{l= 1}^{\nbs^\frac{1}{3}}l^{(2-4)} \sum_{m=1}^{\nbs}  \sum_{\stl{j\geq 1}{j\neq m}}^{\nbs}\frac{\alpha_j}{\dist_{mj}^q}.      
	\end{align*} 
	The proof ends with $$\sum_{m=1}^{\nbs} \sum_{{i\geq1},{m\neq j}}^{\nbs}\frac{\alpha_j}{\dist_{mj}^q}=\sum_{m=1}^{\nbs} \sum_{\stl{\nbs\geq j\geq1}{i\neq j}}\frac{\alpha_j}{\dist_{mj}^q}
	=\sum_{j=1}^{\nbs} \alpha_j\sum_{\stl{m\geq 1}{m\neq j} }\frac{1}{\dist_{mj}^q}.$$ \end{proof}

\bibliographystyle{abbrv}

\end{document}